\documentclass[a4paper,11pt]{amsart}

\usepackage{amsmath, amsthm, amsfonts, amssymb}
\usepackage[all,cmtip]{xy}

\usepackage{enumerate}
\usepackage[colorlinks=true, linkcolor=blue, citecolor=red, menucolor=black]{hyperref}

\usepackage{color} 
\usepackage{graphicx}
\usepackage{mathrsfs}

\numberwithin{equation}{section}

\newtheorem{letterthm}{Theorem}

\newtheorem{lettercor}[letterthm]{Corollary}

\newtheorem{letterapp}{Application}

\newtheorem*{haagerup-theorem}{Relative Bicentralizer Theorem}

\newtheorem{thm}{Theorem}[section]

\newtheorem{lem}[thm]{Lemma}

\newtheorem{cor}[thm]{Corollary}
\newtheorem{prop}[thm]{Proposition}

\theoremstyle{definition}
\newtheorem{rem}[thm]{Remark}
\newtheorem{remark}[thm]{Remarks}
\newtheorem*{example}{Examples}

\newtheorem{df}[thm]{Definition}

\newtheorem{claim}[thm]{Claim}

\newtheorem*{question}{Question}

\newcommand{\R}{\mathbf{R}}
\newcommand{\C}{\mathbf{C}}
\newcommand{\Z}{\mathbf{Z}}

\newcommand{\Q}{\mathbf{Q}}
\newcommand{\N}{\mathbf{N}}
\newcommand{\T}{\mathbf{T}}

\newcommand{\Ad}{\operatorname{Ad}}
\newcommand{\id}{\text{\rm id}}

\newcommand{\Aut}{\mathord{\text{\rm Aut}}}

\newcommand{\Inn}{\mathord{\text{\rm Inn}}}
\newcommand{\rL}{\mathord{\text{\rm L}}}
\newcommand{\rB}{\mathord{\text{\rm B}}}
\newcommand{\rC}{\mathord{\text{\rm C}}}

\newcommand{\rd}{\mathord{\text{\rm d}}}
\newcommand{\co}{\mathord{\text{\rm co}}}

\newcommand{\rS}{\mathord{\text{\rm S}}}
\newcommand{\rD}{\mathord{\text{\rm D}}}
\newcommand{\rE}{\mathord{\text{\rm E}}}

\newcommand{\spec}{\mathord{\text{\rm sp}}}

\newcommand{\core}{\mathord{\text{\rm c}}}
\newcommand{\rZ}{\mathord{\text{\rm Z}}}

\newcommand{\Tr}{\mathord{\text{\rm Tr}}}
\newcommand{\Ball}{\mathord{\text{\rm Ball}}}

\newcommand{\ovt}{\mathbin{\overline{\otimes}}}
\newcommand{\bigovt}{\mathbin{\overline{\bigotimes}}}

\newcommand{\op}{\mathord{\text{\rm op}}}
\newcommand{\AC}{\mathord{\text{\rm AC}}}
\newcommand{\Bic}{\mathord{\text{\rm B}}}

\newcommand{\dpr}{^{\prime\prime}}

\newcommand{\I}{{\rm I}}
\newcommand{\II}{{\rm II}}
\newcommand{\III}{{\rm III}}

\begin{document}

\title[Structure of bicentralizer algebras and inclusions of type $\III$ factors]{Structure of bicentralizer algebras \\ and inclusions of type $\III$ factors}

\begin{abstract}
We investigate the structure of the relative bicentralizer algebra $\rB(N \subset M, \varphi)$ for inclusions of von Neumann algebras with normal expectation where $N$ is a type $\III_1$ subfactor and $\varphi \in N_*$ is a faithful state. We first construct a canonical flow $\beta^\varphi : \R^*_+ \curvearrowright \rB(N \subset M, \varphi)$ on the relative bicentralizer algebra and we show that the W$^*$-dynamical system $(\rB(N \subset M, \varphi), \beta^\varphi)$ is independent of the choice of $\varphi$ up to a canonical isomorphism. In the case when $N=M$, we deduce new results on the structure of the automorphism group of $\rB(M,\varphi)$ and we relate the period of the flow $\beta^\varphi$ to the tensorial absorption of Powers factors. For general irreducible inclusions $N \subset M$, we relate the ergodicity of the flow $\beta^\varphi$ to the existence of irreducible hyperfinite subfactors in $M$ that sit with normal expectation in $N$. When the inclusion $N \subset M$ is discrete, we prove a relative bicentralizer theorem and we use it to solve Kadison's problem when $N$ is amenable. 
\end{abstract}

\author[H.\ Ando]{Hiroshi Ando}
\email{hiroando@math.s.chiba-u.ac.jp}
\address{Department of Mathematics and Informatics \\ Chiba University \\ 1-33 Yayoi-cho \\
Inage \\ Chiba \\ 263-0022 \\ JAPAN}

\author[U.\ Haagerup]{Uffe Haagerup}

\author[C.\ Houdayer]{Cyril Houdayer}
\email{cyril.houdayer@math.u-psud.fr}

\author[A.\ Marrakchi]{Amine Marrakchi}
\email{amine.marrakchi@math.u-psud.fr}
\address{Laboratoire de Math\'ematiques d'Orsay\\ Universit\'e Paris-Sud\\ CNRS\\ Universit\'e Paris-Saclay\\ 91405 Orsay\\ FRANCE}

\thanks{HA is supported by JSPS KAKENHI 16K17608}

\thanks{CH and AM are supported by ERC Starting Grant GAN 637601}



\maketitle

\section{Introduction and statement of the main results}
Let $M$ be any $\sigma$-finite von Neumann algebra and $\varphi \in M_*$ any faithful state. Following \cite{Co80, Ha85}, we define the \emph{bicentralizer} of $M$ with respect to $\varphi$ by
$$\rB(M, \varphi) = \left\{ x \in M \mid x a_{n} - a_{n} x \to 0 \text{ strongly}, \forall (a_{n})_{n} \in \AC(M, \varphi)\right\}$$
where $$\AC(M, \varphi) = \left \{ (a_{n})_{n} \in \ell^{\infty}(\N, N) \mid \lim_{n} \|a_{n} \varphi - \varphi a_{n}\| = 0\right\}$$ is the {\em asymptotic centralizer} of $\varphi$. One of the most famous open problems in the theory of type $\III$ factors is Connes' bicentralizer problem asking whether for any type $\III_1$ factor $M$ with separable predual and any faithful state $\varphi \in M_*$, we have $\rB(M, \varphi)=\C 1$. This question was solved affirmatively by Haagerup in \cite{Ha85} for \emph{amenable} $M$, thus settling the problem of the classification of amenable factors with separable predual (see \cite{Co75b, Co85}). 

Nowadays, the bicentralizer problem is still of premium importance. Indeed by \cite[Theorem 3.1]{Ha85}, for any type ${\rm III_{1}}$ factor $M$ with separable predual, $M$ has trivial bicentralizer if and only if there exists a faithful state $\varphi \in M_{\ast}$ with an irreducible centralizer, meaning that $(M_{\varphi})' \cap M = \C 1$. Then by  \cite[Theorem 3.1]{Ha85} and \cite[Theorem 3.2]{Po81}, $M$ has trivial bicentralizer if and only if there exists a maximal abelian subalgebra $A \subset M$ that is the range of a normal conditional expectation (see \cite[Question]{Ta71} where the problem of finding such maximal abelian subalgebras is mentioned). For these reasons, Connes' bicentralizer problem appears naturally when one tries to use Popa's deformation/rigidity theory in the type $\III$ context (see for instance \cite[Theorem C]{HI15}). The bicentralizer problem is known to have a positive solution for particular classes of nonamenable type $\III_{1}$ factors: factors with a Cartan subalgebra; Shlyakhtenko's free Araki--Woods factors (\cite{Ho08}); (semi-)solid factors (\cite{HI15}); free product factors (\cite{HU15}). However, the bicentralizer problem is still wide open for arbitrary type $\III_1$ factors.

In his attempt to solve the bicentralizer problem, Connes observed that for any type $\III_1$ factor $M$, the bicentralizer $\rB(M,\varphi)$ does not depend on the choice of the state $\varphi$ up to a canonical isomorphism. In around 2012--2013, Haagerup found out that the idea of Connes' isomorphism (denoted by $\beta_{\psi,\varphi}$ below) can be enhanced to construct a canonical flow ($u$-continuous action) $\beta^{\varphi}\colon \mathbf{R}_+^*\curvearrowright \Bic(M,\varphi)$ with interesting properties. This flow was independently discovered by Marrakchi and this was the starting point of our joint research.

Let $N \subset M$ be any inclusion of $\sigma$-finite von Neumann algebras {\em with expectation}, meaning that there exists a faithful normal conditional expectation $\rE_{N} : M \to N$. Following \cite[Definition 4.1]{Ma03}, we define the {\em relative bicentralizer} $\rB(N \subset M, \varphi)$ of the inclusion $N \subset M$ with respect to the faithful state $\varphi \in N_*$ by 
$$\rB(N \subset M, \varphi) = \left\{ x \in M \mid x a_{n} - a_{n} x \to 0 \text{ strongly}, \forall (a_{n})_{n} \in \AC(N, \varphi)\right\}.$$
Observe that we always have $N' \cap M \subset \rB(N\subset M, \varphi) \subset (N_{\varphi})' \cap M$. When $N = M$, we simply have $\rB(N \subset M,\varphi)=\rB(M, \varphi)$.

Our first main result deals with the construction of the canonical flow on the relative bicentralizer $\rB(N \subset M, \varphi)$.

\begin{letterthm} \label{thm: Connes' isomorphism and bicentralizer flow} 
Let $N \subset M$ be any inclusion of $\sigma$-finite von Neumann algebras with expectation. Assume that $N$ is a type $\III_1$ factor. Then the following assertions hold:
\begin{itemize}
 \item [$(\rm i)$] For every pair of faithful states $\varphi, \psi \in N_*$, there exists a canonical isomorphism $$\beta_{\psi, \varphi} : \rB(N \subset M,\varphi) \rightarrow \rB(N \subset M,\psi)$$ characterized by the following property: for any uniformly bounded sequence $(a_n)_{n \in \N}$ in $N$ and any $x \in \rB(N \subset M,\varphi)$, we have 
$$  \|a_n \varphi - \psi a_n\| \to 0 \quad  \Rightarrow \quad a_nx-\beta_{\psi, \varphi}(x)a_n \to 0 \; \ast\text{-strongly}. $$

 \item [$(\rm ii)$] There exists a canonical flow $$\beta^{\varphi}: \R^*_+ \curvearrowright \Bic(N \subset M,\varphi)$$ characterized by the following property: for any uniformly bounded sequence $(a_n)_{n \in \N}$ in $N$, any $x \in \rB(N \subset M,\varphi)$ and any $\lambda > 0$, we have 
$$  \|a_n \varphi -\lambda\varphi a_n\| \to 0 \quad  \Rightarrow \quad a_nx-\beta_\lambda^{\varphi}(x)a_n \to 0 \; \ast\text{-strongly}.$$

\item [$(\rm iii)$] We have $\beta_{\varphi_3,\varphi_2} \circ \beta_{\varphi_2, \varphi_1}=\beta_{\varphi_3, \varphi_1}$ for every faithful state $\varphi_i \in N_*$, $i \in \{1,2,3\}$, and $\beta_\lambda^{\psi} \circ \beta_{\psi, \varphi} = \beta_{\psi, \varphi} \circ \beta_\lambda^{\varphi}$ for every pair of faithful states $\psi, \varphi \in N_*$ and every $\lambda > 0$.

\item [$(\rm iv)$] For every pair of faithful states $\psi, \varphi \in N_*$ and every $\lambda > 0$, we have
$$\rE_{N' \cap M}^{\psi} \circ \beta_{\psi, \varphi} = \rE_{N' \cap M}^{\varphi}= \rE_{N' \cap M}^{\varphi} \circ \beta_{\lambda}^{\varphi} $$
where $\rE_{N' \cap M}^{\varphi} : M \rightarrow N' \cap M$ is the unique normal conditional expectation such that $\rE_{N' \cap M}^{\varphi}(x)=\varphi(x)1$ for all $x \in N$.

 \end{itemize}
 \end{letterthm}

The proof of Theorem \ref{thm: Connes' isomorphism and bicentralizer flow} uses ultraproduct von Neumann algebras \cite{Oc85, AH12} and relies on Connes--St\o rmer transitivity theorem \cite{CS76} and the fact that any $\lambda > 0$ is an approximate eigenvalue for the faithful state $\varphi \in N_{\ast}$ (see Lemma \ref{lem: family of partial isometries}). 

The meaning of the compatibility relations given in item $(\rm iii)$ is that the W$^*$-dynamical system $(\rB(N \subset M, \varphi), \beta^\varphi)$ does not depend on the choice of $\varphi \in N_\ast$ up to the canonical isomorphism $\beta_{\psi, \varphi}$. Thus, $(\rB(N \subset M, \varphi), \beta^\varphi)$ is an invariant of the inclusion $N \subset M$. We call it the \emph{relative bicentralizer flow} of the inclusion $N \subset M$. When $N=M$, we simply call it the \emph{bicentralizer flow} of $M$. In this paper, we study this invariant and we relate it to some structural properties of the inclusion $N \subset M$.

\subsection*{Self-bicentralizing factors}
For any von Neumann algebra $M$ and any faithful state $\varphi \in M_{\ast}$, we have $\rB(\rB(M,\varphi), \varphi|_{\rB(M,\varphi)})=\rB(M,\varphi)$ and if $M$ is a type $\III_1$ factor then $\rB(M,\varphi)$ is either trivial or a type $\III_1$ factor. Therefore, the bicentralizer problem reduces to the following question: does there exist a type $\III_1$ factor $M$ which satisfies $M=\rB(M,\varphi)$ for some faithful state $\varphi \in M_{\ast}$? We call such a state $\varphi$ a \emph{bicentralizing} state on $M$. If such a factor exists, by \cite{Ha85}, it must be nonamenable and by \cite{HI15}, it must be McDuff, that is, $M\cong M \ovt R$ where $R$ is the hyperfinite type $\II_1$ factor. In our next result, we use Theorem \ref{thm: Connes' isomorphism and bicentralizer flow} to further understand the structure of these mysterious self-bicentralizing type $\III_1$ factors. We show that the bicentralizing state is unique up to conjugacy by a unique approximately inner automorphism. We also show that their automorphism group splits as a semi-direct product. Finally, we relate the period of the bicentralizer flow to the tensorial absorption of Powers factors.

\begin{letterthm} \label{thm: self-bicentralizing}
Let $M$ be any type $\III_1$ factor with separable predual and with a bicentralizing state $\varphi \in M_{\ast}$. Then the following properties hold:

\begin{itemize}
\item [$(\rm i)$] Let $\Delta(M)$ be the set of all bicentralizing states of $M$. Then the map 
$$\overline{\Inn}(M) \ni \alpha \mapsto \alpha(\varphi) \in \Delta(M)$$ is a homeomorphism and its inverse is given by
$$ \Delta(M) \ni \psi \mapsto \beta_{\psi, \varphi} \in \overline{\Inn}(M).$$
 \item [$(\rm ii)$] Define $\Aut_\varphi(M) = \{\alpha \in \Aut(M) \mid \alpha(\varphi) = \varphi\}$ and consider the conjugation action $\Aut_\varphi(M) \curvearrowright \overline{\Inn}(M)$. Then the natural homomorphism
$$ \iota : \overline{\Inn}(M) \rtimes \Aut_{\varphi}(M) \ni (g,h)\mapsto g \circ h  \in \Aut(M)$$
is an isomorphism of topological groups. In particular, $\sigma_t^{\varphi} \notin \overline{\Inn}(M)$ for all $t \neq 0$.
\item [$(\rm iii)$] For every $0<\lambda<1$, we have 
$$ M\cong M\overline{\otimes}R_{\lambda} \; \Leftrightarrow \; \beta_\lambda^{\varphi}=\id \; \Leftrightarrow  \; \beta_\lambda^{\varphi} \in \overline{\Inn}(M).$$ 
In particular, we have $M\cong M\overline{\otimes}R_{\infty}$ if and only if the bicentralizer flow $\beta^{\varphi}$ is trivial.
 \item [$(\rm iv)$]  For every $\lambda > 0$, the automorphism $\beta_\lambda^{\varphi} \odot \id$ of $M \odot M^{\op}$ extends to the $\rC^{*}$-algebra $\rC^{*}_{\lambda \cdot \rho}(M)$ generated by the standard representation of $M \odot M^{\op}$ on $\rL^{2}(M)$.
\end{itemize}
\end{letterthm}

In Connes' strategy to prove the uniqueness of the amenable type $\III_1$ factor \cite{Co85}, a crucial step was to show that $\sigma_t^{\varphi}\in \overline{\rm{Inn}}(M)$ for every $t \in \R$. The amenability of $M$ implies that $\rC^*_{\lambda . \rho}(M)=M \otimes_{\min} M^{\op}$ and so any automorphism of $M$ satisfies the property $(\rm iv)$ above. Connes conjectured \cite[Section IV]{Co85} that for any type $\III_1$ factor $M$, any automorphism satisfying the property $(\rm iv)$ above must be approximately inner. It is in trying to prove this conjecture that he encountered the bicentralizer problem. As we see from item $(\rm iii)$ and item $(\rm iv)$ above, if this conjecture is true then the bicentralizer flow $\beta^\varphi$ must be trivial.

\subsection*{Irreducible hyperfinite subfactors in inclusions of type $\III$ factors}
Let $N \subset M$ be any irreducible inclusion of factors with separable predual and with expectation. In \cite{Po81}, Popa proved that if $N$ is \emph{semifinite}, then there exists a hyperfinite subfactor with expectation $P \subset N$  such that $P' \cap M=\C1$. We extend this theorem to the case when $N$ is a type $\III_\lambda$ factor $(0 < \lambda < 1)$ in Theorem \ref{thm: hyperfinite subfactor III_lambda}. In the case when $N$ is a type $\III_1$ factor, we relate this question to the ergodicity of the relative bicentralizer flow.

\begin{letterthm} \label{thm: hyperfinite subfactor III_1}
Let $N \subset M$ be any inclusion of von Neumann algebras with separable predual and with expectation. Assume that $N$ is a type $\III_1$ factor. Let $\varphi \in N_*$ be any faithful state. The following assertions are equivalent:
\begin{itemize}
\item [$(\rm i)$] $\rB(N \subset M, \varphi)^{\beta^{\varphi}}=N' \cap M$.

\item [$(\rm ii)$]  There exists a hyperfinite subfactor with expectation $P \subset N$ such that $P' \cap M=N' \cap M$.
\end{itemize}
We can always choose $P = R_{\infty}$ to be the hyperfinite type $\III_1$ factor. 
\begin{itemize}
\item We can moreover choose $P = R_{\lambda}$ to be the hyperfinite type $\III_\lambda$ factor $(0 < \lambda < 1)$ if and only if $\rB(N \subset M, \varphi)^{\beta^{\varphi}_\lambda}=N' \cap M$.

\item We can moreover choose $P = R$ to be the hyperfinite type $\II_1$ factor if and only if $\rB(N \subset M, \varphi)=N' \cap M$. 
\end{itemize}

\end{letterthm} 

The proof of Theorem \ref{thm: hyperfinite subfactor III_1} generalizes the methods developed by Popa in \cite[Theorem 3.2]{Po81} and Haagerup in \cite[Theorem 3.1]{Ha85}. 

Let us mention that by \cite[Theorem 2]{Po83} (see also \cite[Theorem 5.1]{Lo83} for the infinite case), any factor $M$ with separable predual possesses an irreducible hyperfinite subfactor that is typically {\em not} the range of a normal expectation. In contrast, constructing an irreducible hyperfinite subfactor that is  the range of a normal expectation is a more subtle problem. In \cite[Corollaries 1.3 and 3.4]{Po85}, Popa shows that any type ${\rm III_{\lambda}}$ factor $M$ with separable predual with $0 \leq \lambda < 1$ possesses an irreducible hyperfinite subfactor with expectation. Our Theorem \ref{thm: hyperfinite subfactor III_1} shows in particular that a type ${\rm III_{1}}$ factor $M$ with separable predual possesses an irreducible hyperfinite subfactor with expectation if and only if the bicentralizer flow is ergodic.

Following \cite{Co72, Co74}, a $\sigma$-finite von Neumann algebra $Q$ is {\em almost periodic} if $Q$ possesses an {\em almost periodic} state, that is, a faithful normal state for which the corresponding modular operator is diagonalizable. By \cite{Co72, Co74}, any $\sigma$-finite type ${\rm III_{\lambda}}$ factor with $0 \leq \lambda < 1$ is almost periodic. When $N \subset M$ is an irreducible inclusion of factors with separable predual and with expectation, a sufficient condition for the relative bicentralizer flow $\beta^{\varphi} : \R^{*}_{+} \curvearrowright \rB(N \subset M, \varphi)$ to be ergodic is the existence of an almost periodic subfactor with expectation $Q \subset N$  such that $Q' \cap M = \C 1$. Using Theorem \ref{thm: hyperfinite subfactor III_1}, we derive the following application that gives a partial solution to \cite[Problem 4]{Po85}.

\begin{letterapp}\label{app-almost-periodic}
Let $M$ be any ${\rm III_{1}}$ factor with separable predual. Assume that there exists an irreducible almost periodic subfactor with expectation $Q \subset M$.

Then there exists an irreducible hyperfinite subfactor with expectation $P \subset M$.
\end{letterapp}

We point out that it is unclear whether we can choose $P$ as a subfactor of $Q$. We can do so if $Q$ possesses an almost periodic faithful state $\varphi \in Q_\ast$ such that its centralizer $Q_\varphi$ is a type ${\rm II_1}$ factor (see Theorem \ref{thm: hyperfinite subfactor III_lambda}). However, when $Q$ is a type ${\rm III_0}$ factor, no such almost periodic state exists on $Q$ and so we really need to exploit the ergodicity of the relative bicentralizer flow to construct the hyperfinite subfactor $P \subset N$.

A sufficient condition for an inclusion of factors $N \subset M$ to be irreducible is the existence of an abelian subalgebra $A \subset N$ that is maximal abelian in $M$. One of Kadison's well-known problems in \cite{Ka67} asks whether the converse is true. We will say that an irreducible inclusion of factors with expectation $N \subset M$ satisfies {\em Kadison's property} if there exists an abelian subalgebra with expectation $A \subset N$ that is maximal abelian in $M$. Popa proved in \cite[Theorem 3.2]{Po81} that any irreducible inclusion $N \subset M$ with separable predual and with expectation such that $N$ is semifinite satisfies Kadison's property. 

Combining Theorem \ref{thm: hyperfinite subfactor III_1} with \cite[Theorem 3.2]{Po81}, we obtain the following characterization:

\begin{lettercor} \label{kadison bicentralizer}
Let $N \subset M$ be any irreducible inclusion of factors with separable predual and with expectation. Assume that $N$ is a type $\III_1$ factor. Then the following assertions are equivalent:
\begin{itemize}
\item [$(\rm i)$] $\rB(N \subset M, \varphi)=\C1$ for some (or any) faithful state $\varphi \in N_*$.

\item [$(\rm ii)$] The inclusion $N \subset M$ satisfies Kadison's property.
\end{itemize}
\end{lettercor}

In the case when $N \subset M$ has finite index, Corollary \ref{kadison bicentralizer} follows from \cite[Theorem 4.2]{Po95}. In order to find new examples of inclusions $N \subset M$ that satisfy Kadison's property, we will prove a relative bicentralizer theorem for {\em discrete} inclusions. 

\subsection*{Relative bicentralizers of discrete inclusions}
Following \cite[Definition 1.1]{Po98}, we say that an inclusion of von Neumann algebras $Q \subset P$ satisfies the {\em weak relative Dixmier property} if for every $x \in P$, we have
$$\mathcal K_{Q}(x) \cap (Q' \cap P) \neq \emptyset$$
where $\mathcal K_{Q}(x) = \overline{\co}^{w} \left \{uxu^{*} \mid u \in \mathcal U(Q)\right \}$. Recall that $Q$ is amenable if and only if the inclusion $Q \subset \mathbf B(H)$ satisfies the weak relative Dixmier property for some (or any) unital faithful normal representation of $Q$ on $H$ (see \cite{Sc63}). In particular, any inclusion $Q \subset P$ with $Q$ amenable satisfies the weak relative Dixmier property. 

In \cite{Ha85}, a connection was established between the bicentralizer problem and the weak relative Dixmier property. Indeed, by \cite[Theorem 3.1]{Ha85}, a type $\III_1$ factor with separable predual has trivial bicentralizer if and only if the inclusion $M_\psi \subset M$ satisfies the weak relative Dixmier property for some (or any) dominant weight $\psi$ on $M$. This solved in particular the bicentralizer problem for amenable $M$. 

Following \cite[Definition 3.7]{ILP96}, an inclusion of von Neumann algebras $N \subset M$ with separable predual and with expectation is {\em discrete} if the inclusion $N \subset \langle M, N \rangle = (JNJ)' \cap \mathbf B(\rL^{2}(M))$ is with expectation (this is indeed equivalent to \cite[Definition 3.7]{ILP96} thanks to \cite[Theorem 6.6]{Ha77}).

\begin{example}
Here are some fundamental examples of discrete inclusions of factors with separable predual.
\begin{itemize}
\item [$(\rm i)$] The inclusion $N \otimes \C 1 \subset N \ovt Q$, where $N, Q$ are any factors.
\item [$(\rm ii)$] Any finite index inclusion $N \subset M$ (see \cite{Jo82, PP84, Ko85}).
\item [$(\rm iii)$] The inclusion $N \subset N \rtimes \Gamma$, where $\Gamma$ is any countable discrete group, $N$ is any factor and $\Gamma \curvearrowright N$ is any outer action (see \cite[Section 3]{ILP96}).
\item [$(\rm iv)$] The inclusion $M^{\mathbf G} \subset M$, where $\mathbf G$ is any compact second countable group, $M$ any factor and $\alpha : \mathbf G \curvearrowright M$ any {\em minimal} action, meaning that $\alpha$ is faithful and $(M^{\mathbf G})' \cap M = \C 1$ (see \cite[Section 3]{ILP96}).
\end{itemize}
\end{example}

For discrete inclusions of von Neumann algebras $N \subset M$ where $N$ is a type ${\rm III_{1}}$ factor, we obtain the following {\em relative bicentralizer theorem} that generalizes  Haagerup's result \cite[Theorem 3.1]{Ha85} (when $N = M$) and Popa's result \cite[Theorem 4.2]{Po95} (when $N \subset M$ has finite index). 

\begin{letterthm}\label{thm-relative-bicentralizer}
Let $N \subset M$ be any inclusion of von Neumann algebras with separable predual and with expectation. Assume that $N$ is a type $\III_1$ factor. Consider the following assertions:

\begin{itemize}
\item [$(\rm i)$] For some (or any) faithful state $\varphi \in N_{\ast}$, we have $\rB(N\subset M, \varphi) = N' \cap M$.

\item [$(\rm ii)$] For some (or any) dominant weight $\psi$ on $N$, we have $(N_{\psi})' \cap M = N' \cap M$ and the inclusion $N_{\psi} \subset M$ satisfies the weak relative Dixmier property.

\end{itemize}
Then $(\rm i) \Rightarrow (\rm ii)$. If the inclusion $N \subset M$ is discrete, then $(\rm ii) \Rightarrow (\rm i)$.
\end{letterthm}

Let us point out that Theorem \ref{thm-relative-bicentralizer} gives a positive answer to \cite[Further Remarks 4.9 (2)]{Po95}. Moreover, by combining Theorem \ref{thm-relative-bicentralizer} with Corollary \ref{kadison bicentralizer} and a generalization of Connes--Takesaki relative commutant theorem (see Theorem \ref{thm-relative-connes-takesaki}), we solve Kadison's problem for all discrete irreducible inclusions $N \subset M$ where $N = R_{\infty}$ is the hyperfinite type ${\rm III_{1}}$ factor. This answers a question raised in \cite[Problem 6]{HP17}.

\begin{lettercor}\label{cor-characterization}
Let $N \subset M$ be any discrete irreducible inclusion of factors with separable predual and with expectation. Assume that $N = R_{\infty}$ is the hyperfinite type ${\rm III_{1}}$ factor. Then the following assertions are equivalent:

\begin{itemize}

\item [$(\rm i)$] The inclusion of continuous cores $\core(N) \subset \core(M)$ is irreducible, i.e.\ $\core(N)' \cap \core(M) = \C 1$.

\item [$(\rm ii)$] The inclusion $N \subset M$ satisfies Kadison's property.
\end{itemize}
\end{lettercor}

We emphasize the fact that condition $(\rm i)$ in Corollary \ref{cor-characterization} can be concretely checked in many situations. In particular, for crossed products by discrete groups, combining Corollary \ref{cor-characterization} with \cite[Proposition 5.4]{HS88}, we obtain:

\begin{letterapp}\label{app-crossed-product}
Let $N = R_{\infty}$ be the hyperfinite type $\III_1$ factor. Let $\Gamma$ be any countable discrete group and $\alpha : \Gamma \curvearrowright N$ any outer action. Then the following assertions are equivalent:
\begin{itemize}
\item [$(\rm i)$]  The automorphism $\alpha(g) \circ \sigma_t^\varphi$ is outer for all $g \in \Gamma \setminus \{e\}$ and all $t \in \R$.

\item [$(\rm ii)$] The factor $N \rtimes \Gamma$ is of type $\III_1$.

\item [$(\rm iii)$] The inclusion $N \subset N \rtimes \Gamma$ satisfies Kadison's property.
\end{itemize}
\end{letterapp}

In particular, the above conditions are satisfied by all Bernoulli actions $\Gamma \curvearrowright \bigovt_{\Gamma} (R_{\infty}, \phi)$ where $\Gamma$ is an arbitrary countable discrete group and $\phi \in (R_{\infty})_{\ast}$ is an arbitrary faithful state. This result is new in the case when $\phi$ is ergodic, that is, when $(R_{\infty})_{\phi} = \C 1$.

For minimal actions of compact second countable groups, combining Corollary \ref{cor-characterization} with \cite[Corollary 5.14]{Iz01}, we obtain:

\begin{letterapp}\label{app-minimal-action}
Let $M = R_{\infty}$ be the hyperfinite type $\III_1$ factor. Let $\mathbf G$ be any compact connected semisimple Lie group and $\alpha : \mathbf G \curvearrowright M$ any minimal action. 

Then the inclusion $M^{\mathbf G} \subset M$ satisfies Kadison's property.
\end{letterapp}

\subsection*{Bicentralizers of tensor product factors}

It is straightforward to see that if $M$ and $N$ have trivial bicentralizer, then $M \ovt N$ also has trivial bicentralizer (see Proposition \ref{tensor-formula}). The following result provides a partial converse.

\begin{letterthm} \label{tensor_products_trivial}
Let $M$ be any $\sigma$-finite type $\III_1$ factor. Suppose that there exists a $\sigma$-finite factor $N$ such that $M \ovt N$ has trivial bicentralizer. Then $M \ovt R_\infty$ has trivial bicentralizer. 

If $N$ is a type $\III_\lambda$ factor for $0 < \lambda < 1$, then $M \ovt R_\lambda$ has trivial bicentralizer. If $N$ is semifinite, then $M$ has trivial bicentralizer.
\end{letterthm}

We already mentioned the importance of the bicentralizer problem in the framework of Popa's deformation/rigidity theory for type $\III$ factors. In this respect, Theorem \ref{tensor_products_trivial} has a direct application to Unique Prime Factorization results. Indeed, we can remove all the assumptions on the unknown tensor product decomposition  in \cite[Theorem B]{HI15} to obtain the following W$^{*}$-rigidity result.
\begin{letterapp}\label{UPF}
Let $m,n \geq 1$ be any integers. For each $1 \leq i \leq m$, let $M_i$ be a nonamenable factor in the class $\mathcal{C}_{\rm (AO)}$. For each $1 \leq j \leq n$, let $N_j$ be any non type $\I$ factor and suppose that
$$M= M_1 \ovt \cdots \ovt M_m = N_1 \ovt \cdots \ovt N_n.$$
Then there exists a surjection $\sigma : \{1,\dots,m\} \rightarrow \{1, \dots, n\}$, a family of type $\I$ factors $F_1, \dots, F_n$, and a unitary $u \in M \ovt F_1 \ovt \cdots \ovt F_n$ such that for all $1 \leq j \leq n$, we have
$$ u(F_j \ovt N_j)u^*=F_j \ovt \overline{\bigotimes}_{ i \in \sigma^{-1}(j)} M_i.$$
In particular, for all $1 \leq j \leq n$, the factor $N_j$ is stably isomorphic to $\overline{\bigotimes}_{ i \in \sigma^{-1}(j)} M_i$.
\end{letterapp}
 
\subsection*{Open question}
 
Let $N \subset M$ be any irreducible inclusion of factors with separable predual and with expectation $\rE_N : M \rightarrow N$. Assume that $N$ is a type $\III_1$ factor. On the one hand, by Theorem \ref{thm: Connes' isomorphism and bicentralizer flow}, one can associate with the inclusion $N \subset M$ a canonical state preserving W$^*$-dynamical system that we called the {\em relative bicentralizer flow}
$$ (\rB(N \subset M, \varphi), \beta^\varphi , (\varphi \circ \rE_N) |_{ \rB(N \subset M, \varphi)} ).$$
This state preserving W$^*$-dynamical system does not depend on the choice of the faithful state $\varphi \in N_*$ up to isomorphism. On the other hand, one can associate with the inclusion $N \subset M$ yet another canonical state preserving W$^*$-dynamical system that we call the \emph{relative flow of weights}
 $$ ( (N_\psi)' \cap M, \theta^\psi, \rE_N |_{(N_\psi)' \cap M} )$$
where $\psi$ is any dominant weight on $N$ and $\theta^\psi : \R^*_+ \curvearrowright (N_\psi)' \cap M$ is the flow given by $\theta^\psi_\lambda(x)=uxu^*$ for all $x \in (N_\psi)' \cap M$ and any unitary $u \in N$ such that $u\psi u^*=\lambda \psi, \; \lambda > 0$. This state preserving W$^*$-dynamical system does not depend on the choice of $\psi$ up to isomorphism and it is ergodic (see Subsection \ref{subsection:flow} for further details). 

There is a striking analogy between the relative bicentralizer flow and the relative flow of weights. Moreover, Theorem \ref{thm-relative-bicentralizer} shows that they are indeed closely related. Therefore, one is naturally led to ask the following question.

\begin{question} 
Are the relative bicentralizer flow and the relative flow of weights isomorphic? More precisely, does there exist an isomorphism 
$$\pi : \rB(N \subset M, \varphi) \rightarrow (N_\psi)' \cap M$$ 
such that $\theta^\psi =\pi \circ \beta^\varphi\circ \pi^{-1}$ and $\rE_N(\pi(x))=\varphi(\rE_N(x))1$ for all $x \in \rB(N \subset M, \varphi)$?
\end{question}

Note that the above question is equivalent to the bicentralizer problem when $N=M$. Therefore, in this generality, it is out of reach at the present time. Nevertheless, in Proposition \ref{prop:open-question}, we provide examples for which the above question has a positive solution. 

We also point out that since the relative flow of weights is always ergodic, it is reasonable to think that the same is true for the relative bicentralizer flow. In particular, when $N=M$, we have at the same time a good reason to believe that the bicentralizer flow is ergodic and we also have a good reason to believe that it must be trivial (see the discussion after Theorem \ref{thm: self-bicentralizing}). In other words, we strongly believe that the bicentralizer problem has a positive solution for all type $\III_1$ factors with separable predual.

\subsection*{Acknowledgments} Cyril Houdayer is grateful to Sorin Popa for thought-provoking discussions that took place in Orsay in the Spring of 2017, and later inspired the idea of the relative bicentralizer theorem for discrete inclusions (Theorem \ref{thm-relative-bicentralizer}). He also thanks him for mentioning the reference \cite{Po85} in connection with Theorem \ref{thm: hyperfinite subfactor III_1}.

{\hypersetup{linkcolor=black} {\setlength{\parskip}{0.3ex}
\tableofcontents} }

\section{Preliminaries}

\subsection{Standard form}
For any von Neumann algebra $M$, we denote by $(M, \rL^2(M), J, \rL^{2}(M)^{+})$ its \emph{standard form} \cite{Ha73}. Recall that $\rL^2(M)$ is naturally endowed with the structure of a $M$-$M$-bimodule: we will simply write $x \xi y = x Jy^*J \xi$ for all $x, y \in M$ and all $\xi \in \rL^2(M)$. For any positive functional $\varphi \in M_\ast^+$, there exists a unique vector $\xi \in \rL^{2}(M)^{+}$ such that $\varphi(x)=\langle x \xi, \xi \rangle$ for every $x \in M$. We denote this vector by $\xi_\varphi \in \rL^{2}(M)^{+}$. We then have $\|x\|_\varphi = \|x \xi_\varphi\|$ for all $x \in M$.

\subsection{Ultraproduct von Neumann algebras}

Let $M$ be any $\sigma$-finite von Neumann algebra, $I$ any directed set and $\omega$ any cofinal ultrafilter on $I$ (cofinal means that $\{ j \in I \mid j \geq i \} \in \omega$ for every $i \in I$). Let $A=(M,\| \cdot \|_\infty)^\omega$ be the ultraproduct Banach space of $M$ with respect to $\omega$. Then $A$ is naturally a $\rC^*$-algebra but it is not a von Neumann algebra in general. Let $A^{**}$ be the bidual of $A$ which is a von Neumann algebra. Let $(M_*)^\omega$ be the ultraproduct Banach space of $M_*$. Then $(M_*)^\omega$ can be naturally identified with a closed subspace of $A^*$ via the embedding
\[ (\varphi_i)^\omega \mapsto \left( (x_i)^\omega \mapsto \lim_{i \rightarrow \omega} \varphi_i(x_i) \right) \]
The orthogonal of $(M_*)^\omega$ in $A^{**}$ defined by
\[ \mathfrak{J}=\{ x \in A^{**} \mid \forall \varphi \in (M_*)^\omega, \; \varphi(x)=0 \} \]
is a weak$^{*}$ closed ideal in the von Neumann algebra $A^{**}$ which means that the quotient $M^\omega_{GR} = A^{**}/\mathfrak{J}$ is a von Neumann algebra. It is called the \emph{Groh--Raynaud ultraproduct} of $M$ with respect to $\omega$. By construction, the predual of $M^\omega_{GR}$ is exactly $(M_*)^\omega$ and $M^\omega_{GR}$ contains the ultraproduct Banach space $A=(M,\| \cdot \|_\infty)^\omega$ as a dense $\rC^*$-subalgebra. The $\ast$-homomorphism $M \to M^\omega_{GR} : x \mapsto x^\omega$ is not normal in general and so $M$ is not a von Neumann subalgebra of $M^\omega_{GR}$. The von Neumann algebra $M^\omega_{GR}$ is very large (not separable and not even $\sigma$-finite in general). The main interest in this ultraproduct comes from the fact that, as explained in \cite{AH12}, there is a natural identification $\rL^2(M_{GR}^\omega)=\rL^2(M)^\omega$.

Choose a faithful state $\varphi \in M_*$. Then we have $\varphi^\omega \in (M^\omega_{GR})_*^+$ but $\varphi^{\omega}$ is not faithful in general. Let $e$ be the support of $\varphi^{\omega}$ in $M^\omega_{GR}$. The projection $e$ does not depend on the choice of $\varphi$ and the corner $e(M^\omega_{GR})e$ coincides with the \emph{Ocneanu ultraproduct} of $M$ with respect to $\omega$ \cite{Oc85, AH12}. It is simply denoted by $M^{\omega}$. For every $x \in M^{\omega}$, we can find $(x_i)^{\omega} \in (M, \| \cdot \|_\infty)^{\omega}$ such that $x=(x_i)^{\omega}e=e(x_i)^{\omega}$, and in that case, when no confusion is possible, we will abuse notation and write $x=(x_i)^{\omega}$. By construction, we have $(M^\omega)_*=e(M_*)^\omega e$ and $\rL^2(M^\omega)=e(\rL^2(M)^\omega) e$.

We will need the following lemma which allows us to lift matrix units from the Ocneanu ultrapower of a factor to the factor itself.
\begin{lem} \label{lift_matrix}
Let $M$ be any $\sigma$-finite factor. Let $\theta : F \rightarrow M^{\omega}$ be any $*$-homomorphism where $F$ is a finite dimensional factor and $\omega \in \beta (\N) \setminus \N$ is any nonprincipal ultrafilter. Then there exists a sequence of $*$-homomorphisms $\theta_n : F \rightarrow M$ such that $\theta=(\theta_n)^{\omega}$.

Moreover, if $\rE : M^\omega \rightarrow \theta(F)' \cap M^\omega$ is a conditional expectation, then for any faithful state $\varphi \in M_*$ and any $x=(x_n)^\omega \in M^\omega$, we have $\varphi^\omega \circ \rE=(\varphi \circ \rE_n)^\omega$ and $\rE(x)=(\rE_n(x_n))^\omega$, where $\rE_n : M \rightarrow \theta_n(F)' \cap M$ is the unique conditional expectation such that $\rE_n \circ \theta_n = \rE \circ \theta$.
\end{lem}
\begin{proof}
The first part of the lemma is proved in \cite[Proposition 1.1.3]{Co75a} for the asymptotic centralizer $M_\omega$ but the same proof also works for the ultrapower $M^{\omega}$.

For the second part of the lemma, put $\psi(x)1 = (\rE \circ \theta)(x)$ for every $x \in F$. Then $\psi$ is a state on $F$ and the explicit formula for the conditional expectation $\rE : M^\omega \rightarrow \theta(F)' \cap M^\omega$ is given by
$$\forall x \in M^\omega, \quad \rE(x)=\sum_{1\leq i,j,k \leq n} \psi(e_{ij})\theta(e_{ki})x\theta(e_{jk})$$
where $(e_{ij})_{1 \leq i,j \leq n}$ is a matrix unit for $F$. Then we also have
$$\forall x \in M^\omega, \quad \rE_n(x)=\sum_{1\leq i,j,k \leq n} \psi(e_{ij})\theta_n(e_{ki})x\theta_n(e_{jk}).$$
This finishes the proof.
\end{proof}

The following well-known fact about ultralimits will be used repeatedly. 
\begin{lem}\label{lem: ultralimit test}
Let $X$ be any Hausdorff space and $(x_n)_{n=1}^{\infty}$ any sequence in $X$. If there exists $x\in X$ such that $ \lim_{n\to \omega}x_n=x$ for every $\omega\in \beta (\mathbf{N})\setminus \mathbf{N}$, then $ \lim_{n\to \infty}x_n=x$.  
\end{lem}

\subsection{Iterated ultraproducts} 
We shall consider the iterated ultraproduct of von Neumann algebras. We will see that this procedure helps us to show that some sequence in a von Neumann algebra $M$, which defines an element in some ultrapower $M^{\omega}$, actually converges to an element, which is independent of the choice of an ultrafilter $\omega$.   
Let $I,J$ be any directed sets and  $\mathcal{U},\mathcal{V}$ any cofinal ultrafilters on $I$ and $J$, respectively. Then the {\it product ultrafilter}, denoted by $\mathcal{U}\otimes \mathcal{V}$, is a cofinal ultrafilter on $I\times J$ (with the partial ordering $(i,j)\le (i',j')$ if and only if $i\le i'$ and $j\le j'$) given by 
\[\mathcal{U}\otimes \mathcal{V}=\{A\subset I\times J \mid \{i\in I \mid  \{j\in J \mid (i,j)\in A\}\in \mathcal{V}\}\in \mathcal{U}\}.\]

\begin{prop}\label{prop: iterated ultralimit}
If $(x_{i,j})_{(i,j)\in I\times J}$ is a doubly indexed net in a compact Hausdorff space $X$, then 
\[\lim_{(i,j)\to \mathcal{U}\otimes \mathcal{V}}x_{i,j}=\lim_{i\to \mathcal{U}}\lim_{j\to \mathcal{V}}x_{i,j}.\]
\end{prop}
\begin{proof} 
Put $x=\lim_{(i,j)\to \mathcal{U}\otimes \mathcal{V}}x_{i,j}$ and $x_i=\lim_{j\to \mathcal{V}}x_{i,j}\ (i\in I)$. Let $W$ be an open neighborhood of $x$ in $X$. Since a compact Hausdorff space is regular, there exists an open neighborhood $W_1$ of $x$ such that $x\in W_1\subset \overline{W_1}\subset W$. Then $\{(i,j)\in I\times J \mid x_{i,j}\in W_1\}\in \mathcal{U}\otimes \mathcal{V}$, whence $I_0=\{i\in I \mid \{j\in J \mid x_{i,j}\in W_1\}\in \mathcal{V}\}\in \mathcal{U}$ holds. Let $i\in I_0$. Then $B=\{j\in J \mid x_{i,j}\in W_1\}\in \mathcal{V}$. If $V$ is any open neighborhood of $x_i$, then $B'=\{j\in J \mid x_{i,j}\in V\}\in \mathcal{V}$, whence $B\cap B'\in \mathcal{V}$ holds. In particular, we can take $j\in B\cap B'$. Then $x_{i,j}\in V\cap W_1\neq \emptyset$. Since $V$ is arbitrary, this shows that $x_i\in \overline{W_1}\subset W$. Therefore $\mathcal{U}\ni I_0\subset \{i\in I \mid x_i\in W\}$, which shows that $\{i\in I \mid x_i\in W\}\in \mathcal{U}$. Since $W$ is arbitrary, we have $\displaystyle \lim_{i\to \mathcal{U}}x_i=x$.  
\end{proof}
 
As a consequence of Proposition \ref{prop: iterated ultralimit}, we see that for any Banach space $E$, the natural isomorphism
$$ \ell^{\infty}(I \times J, E) \ni (x_{i,j})_{(i,j) \in I \times J} \mapsto ((x_{i,j})_{j \in J})_{i \in I} \in \ell^{\infty}(I , \ell^{\infty}(J, E))$$
induces an isomorphism of the ultrapowers
$$ E^{\mathcal{U} \otimes \mathcal{V}} \ni (x_{i,j})^{\mathcal{U} \otimes \mathcal{V}} \mapsto ((x_{i,j})^{\mathcal{V}})^{\mathcal{U}} \in (E^{\mathcal{V}})^{\mathcal{U}}.$$
If we apply this to $E=M_*$ where $M$ is a von Neumann algebra, we obtain the following proposition which extends \cite[Proposition 2.1]{CP12} on iterated ultrapowers of ${\rm II}_1$ factors to arbitrary $\sigma$-finite von Neumann algebras. We leave the details to the reader.

\begin{prop} \label{prop: double ultrapower}
Let $M$ be any $\sigma$-finite von Neumann algebra. There exists a natural isomorphism of the Groh--Raynaud ultrapowers
$$ \pi_{\mathrm{GR}} :  M_\mathrm{GR}^{\mathcal{U} \otimes \mathcal{V}}  \rightarrow (M_\mathrm{GR}^{\mathcal{V}})_\mathrm{GR}^{\mathcal{U}}$$
characterized by
$$ \pi_{\mathrm{GR}}((x_{i,j})^{\mathcal{U} \otimes \mathcal{V}})=((x_{i,j})^{\mathcal{V}})^{\mathcal{U}} \; \text{ for all } \; (x_{i,j})_{(i,j) \in I \times J} \in \ell^\infty(I \times J, M).$$
Its predual map is the isomorphism
$$ (\pi_{\mathrm{GR}})_* :  (M_*)^{\mathcal{U} \otimes \mathcal{V}}  \ni (\varphi_{i,j})^{\mathcal{U} \otimes \mathcal{V}} \mapsto ((\varphi_{i,j})^{\mathcal{V}})^{\mathcal{U}} \in  ((M_*)^{\mathcal{V}})^{\mathcal{U}}.$$
In particular, $\pi_{\mathrm{GR}}$ restricts to an isomorphism between the Ocneanu corners
$$ \pi : M^{\mathcal{U} \otimes \mathcal{V}} \rightarrow (M^{\mathcal{V}})^{\mathcal{U}}.$$
\end{prop}

Recall that if $N \subset M$ is a von Neumann subalgebra with faithful normal conditional expectation $\rE^{M}_{N} : M \to N$, then we have a natural embedding $N^{\mathcal{U}} \subset M^{\mathcal{U}}$ with faithful normal conditional expectation $\rE^{M^{\mathcal{U}}}_{N^{\mathcal{U}}}=\left(\rE^{M}_N  \right)^{\mathcal{U}} :M^{\mathcal{U}} \to N^{\mathcal{U}}$ and we have a commuting square
\begin{align*}
\begin{array}{ccc}
N & \subset  & M \\
\rotatebox{90}{$\supset$} & & \rotatebox{90}{$\supset$}  \\
N^{\mathcal{U}} & \subset  & M^{\mathcal{U}}\\
\end{array}
\end{align*}
where the conditional expectations satisfy
 $$\rE^{M^{\mathcal{U}}}_M \circ \rE^{M^{\mathcal{U}}}_{N^{\mathcal{U}}}  = \rE^{M^{\mathcal{U}}}_{N^{\mathcal{U}}} \circ \rE^{M^{\mathcal{U}}}_M =\rE^{N^{\mathcal{U}}}_N \circ \rE^{M^{\mathcal{U}}}_{N^{\mathcal{U}}}  = \rE^{M}_N \circ \rE^{M^{\mathcal{U}}}_M.$$
 By applying this to the inclusion $M \subset M^{\mathcal{V}}$ with the canonical faithful normal conditional expectation $\rE^{M^{\mathcal{V}}}_M : M^{\mathcal{V}} \rightarrow M$, we obtain the following result.

\begin{prop} \label{prop: commuting square}
Let $M$ be any $\sigma$-finite von Neumann algebra. Then we have a commuting square
\begin{align*}
\begin{array}{cccc}
M & \subset  & M^{\mathcal{V}} &\\
\rotatebox{90}{$\supset$} & & \rotatebox{90}{$\supset$}&  \\
M^{\mathcal{U}} & \subset  & (M^{\mathcal{V}})^{\mathcal{U}}& =M^{\mathcal{U} \otimes \mathcal{V}}\\
\end{array}
\end{align*}
where the canonical faithful normal conditional expectations satisfy
$$\rE^{M^{\mathcal{U} \otimes \mathcal{V}}}_M =\rE^{(M^{\mathcal{V}})^{\mathcal{U}}}_{M^{\mathcal{V}}} \circ \rE^{(M^{\mathcal{V}})^{\mathcal{U}}}_{M^{\mathcal{U}}} =\rE^{(M^{\mathcal{V}})^{\mathcal{U}}}_{M^{\mathcal{U}}} \circ  \rE^{(M^{\mathcal{V}})^{\mathcal{U}}}_{M^{\mathcal{V}}}=
\rE^{M^{\mathcal{U}}}_M \circ  \rE^{(M^{\mathcal{V}})^{\mathcal{U}}}_{M^{\mathcal{U}}} =\rE^{M^{\mathcal{V}}}_M \circ  \rE^{(M^{\mathcal{V}})^{\mathcal{U}}}_{M^{\mathcal{V}}}.$$
In particular, we have $M^{\mathcal{U}} \cap M^{\mathcal{V}}=M$.
\end{prop}

\subsection{The relative flow of weights}\label{subsection:flow}

Let $N \subset M$ be any irreducible inclusion of factors with separable predual and with expectation $\rE_N : M \rightarrow N$. Assume that $N$ is a type $\III_1$ factor. By using \cite[Theorem ${\rm XII}$.1.1]{Ta03}, we can identify the inclusions
$$\left( N \subset M \right) = \left ( N_{\psi} \rtimes_{\theta} \R^*_+ \subset M_{\psi} \rtimes_{\theta} \R^*_+ \right)$$ 
where $\theta : \R^*_+ \curvearrowright M_{\psi}$ is a trace-scaling action ($\tau \circ \theta_\lambda = \lambda^{-1}\tau$ for every $\lambda > 0$) that leaves $N_{\psi} \subset M_{\psi}$ globally invariant. We denote by $(v_{\lambda})_{\lambda > 0}$ the canonical unitaries in $N$ that implement the trace-scaling action $\theta : \R^*_+ \curvearrowright M_{\psi}$. 

The {\em relative flow of weights} $\theta^\psi : \R^*_+ \curvearrowright (N_\psi)' \cap M$ is defined by $\theta_\lambda^\psi(x) = v_\lambda x v_\lambda^{*}$ for every $x \in (N_\psi)' \cap M$ and every $\lambda > 0$. By Connes--Takesaki relative commutant theorem \cite[Chapter II, Theorem 5.1]{CT76}, ${\rE_N}|_{(N_\psi)' \cap M}$ defines a faithful normal state on $(N_\psi)' \cap M$ that is invariant under the flow $\theta^\psi$.

We observe that the relative flow of weights $\theta^\psi$ does not depend on the choice of the dominant weight $\psi$ on $N$ since all dominants weights on $N$ are unitarily conjugate\cite[Chapter II, Theorem 1.1]{CT76}. We moreover observe that the relative flow of weights $\theta^\psi$ is ergodic. Indeed, if $x \in (N_\psi)' \cap M$ is invariant under $\theta^\psi$, then $x \in (N_\psi \vee \{v_\lambda \mid \lambda > 0\})' \cap M = N' \cap M$ and so $x \in \C 1$.

\subsection{Kadison's property}

We introduce {\em Kadison's property} for arbitrary inclusions of von Neumann algebras with expectation.

\begin{df}
Let $N \subset M$ be any inclusion of $\sigma$-finite von Neumann algebras with expectation. We say that $N \subset M$ satisfies {\em Kadison's property} if there exists an abelian von Neumann subalgebra with expectation $A \subset N \subset M$ that is maximal abelian in $M$.
\end{df}

For any inclusion $N \subset M$ that satisfies Kadison's property, we have $N' \cap M \subset A' \cap M = A \subset N$ and so there is a unique faithful normal conditional expectation $\rE_N : M \to N$ (see \cite[Th\'eor\`eme 1.5.5]{Co72}). In particular, the inclusion of continuous cores $\core(N) \subset \core(M)$ is canonical. We next prove the following useful observation.

\begin{prop}\label{prop-masa}
Let $N \subset M$ be any inclusion of $\sigma$-finite von Neumann algebras with expectation $\rE_{N} : M \to N$. Assume that $N \subset M$ satisfies Kadison's property. Put $M^\infty = M \ovt \mathbf B(\rL^2(\R))$ and regard $M^\infty = \core(M) \rtimes_\theta \R^*_+$ where $\theta : \R^*_+ \curvearrowright \core(M)$ is the dual trace-scaling action. Then we have 
$$\core(N)' \cap M^\infty = \mathcal Z(\core(N)).$$ 
In particular, if $N$ is a type ${\rm III_{1}}$ factor, then so is $M$.
\end{prop}

\begin{proof}
Let $A \subset N \subset M$ be an abelian von Neumann subalgebra with expectation that is maximal abelian in $M$. Choose a faithful state $\varphi \in M_{\ast}$ such that $\varphi \circ \rE_{N} = \varphi$ and such that $A \subset N_{\varphi}$. Regard $\left( \core(N) \subset \core(M) \right ) = \left( \core_{\varphi}(N) \subset \core_{\varphi}(M) \right )$. Put $\core(A) = \core_{\varphi}(A) = A \ovt \rL(\R)$.

Since $A \subset M$ is maximal abelian and since $\rL(\R) \subset \mathbf B(\rL^{2}(\R))$ is maximal abelian, $\core(A) = A \ovt \rL(\R) \subset M \ovt \mathbf B(\rL^{2}(\R)) = M^\infty$ is maximal abelian. Then we have $\core(N)' \cap M^\infty \subset \core(A)' \cap M^\infty = \core(A) \subset \core(N)$ and so $\core(N)' \cap M^\infty = \mathcal Z(\core(N))$.
\end{proof}

\section{Structure of relative bicentralizers}

Following \cite{Co80, Ha85, Ma03}, we introduce the following terminology that we will use throughout.

\begin{df}
Let $N \subset M$ be any inclusion of $\sigma$-finite von Neumann algebras with expectation. Let $\varphi \in N_\ast$ be any faithful state. The {\em relative bicentralizer} $\rB(N \subset M, \varphi)$ of the inclusion $N \subset M$ with respect to $\varphi$ is defined by 
$$\rB(N \subset M, \varphi) = \left\{ x \in M \mid x a_{n} - a_{n} x \to 0 \text{ strongly}, \forall (a_{n})_{n} \in \AC(N, \varphi)\right\}$$
where $$\AC(N, \varphi) = \left \{ (a_{n})_{n} \in \ell^{\infty}(\N, N) \mid \lim_{n} \|a_{n} \varphi - \varphi a_{n}\| = 0\right\}$$ is the {\em asymptotic centralizer} of $\varphi$.
\end{df}

Observe that we always have $N' \cap M \subset \rB(N\subset M, \varphi) \subset (N_{\varphi})' \cap M$.  When $N = M$, we simply have that $\rB(N \subset M,\varphi) = \rB(M, \varphi)$ is the usual bicentralizer.

Let us point out that in the definition of $\rB(N \subset M, \varphi)$, it is enough to consider sequences in $\AC(N, \varphi)$ consisting of unitaries. Indeed, since $\AC(N, \varphi)$ is a $\rC^{*}$-algebra, every element is a linear combination of at most four unitaries. 

\begin{rem}
We note that $\rB(N\subset M, \varphi)$ is always with (canonical) expectation in $M$. Indeed, extend $\varphi$ to a faithful normal state on $M$ by using any faithful normal conditional expectation $\rE_N : M \rightarrow N$. Then $\sigma^{\varphi}$ leaves $\AC(N,\varphi)$ globally invariant and so $\rB(N \subset M, \varphi)$ is also globally invariant under $\sigma^{\varphi}$. Therefore, there exists a unique $\varphi$-preserving faithful normal conditional expectation $\rE^\varphi_{\rB(N \subset M, \varphi)} : M \rightarrow \rB(N \subset M, \varphi)$. Since $N' \cap M \subset \rB(N \subset M, \varphi)$ and since the $\varphi$-preserving faithful normal conditional expectation $\rE^{\varphi}_{N' \cap M} : M \rightarrow N' \cap M$ does not depend on the choice of $\rE_N$, it follows $\rE^\varphi_{\rB(N \subset M, \varphi)}$ does not depend either on the choice of $\rE_N$ and is characterized by
$$ \rE^{\varphi}_{N' \cap M} \circ \rE^\varphi_{\rB(N \subset M, \varphi)} = \rE^{\varphi}_{N' \cap M}.$$
\end{rem}

The relative bicentralizer $\rB(N \subset M, \varphi)$ has the following ultraproduct interpretation.

\begin{prop} \label{ultraproduct_bicentralizer}
Let $N \subset M$ be any inclusion of $\sigma$-finite von Neumann algebras with expectation.  Let $\varphi \in N_*$ be any faithful state. For any nonprincipal ultrafilter $\omega \in \beta (\N) \setminus \N$, we have
$$ \rB(N \subset M,\varphi)=(N^{\omega}_{\varphi^{\omega}})' \cap M$$
and 
$$ (N^{\omega}_{\varphi^{\omega}})' \cap M^{\omega} \subset \rB(N \subset M,\varphi)^{\omega}.$$
\end{prop}
\begin{proof}
Extend $\varphi$ to $M$ by using any faithful normal conditional expectation from $M$ to $N$. Then for any $x \in M$, we have $x \in (N^{\omega}_{\varphi^{\omega}})' \cap M$ if and only if for every $\varepsilon > 0$, there exists $\delta > 0$ such that for every $u \in \mathcal{U}(N)$ we have
 $$\|u \varphi-\varphi u\| \leq \delta \quad  \Rightarrow \quad \| ux-xu \|_\varphi \leq \varepsilon. $$
This is clearly equivalent to $x \in \rB(N \subset M, \varphi)$.

For the second part of the proposition, it is enough to show that for any $x \in M$ such that $\|x\|_{\infty} \leq 1$ and $\rE^\varphi_{\rB(N \subset M, \varphi)}(x)=0$ and any $\delta > 0$, we can find $u \in \mathcal{U}(N)$ such that $\|u \varphi-\varphi u\| \leq \delta$ and $\|uxu^*-x\|_\varphi \geq \|x \|_\varphi$. Let $y$ be the element of minimal $\| \cdot \|_\varphi$-norm in
$$ C_\delta(x) = \overline{\rm co}^{w} \{ uxu^{*} \mid u \in \mathcal{U}(N),   \| u \varphi - \varphi u \| < \delta \}.$$
Let us show that $y \in \rB(N \subset M, \varphi)$. Take a sequence $v_n \in \mathcal{U}(N)$ such that $\lim_n \|v_n \varphi-\varphi v_n\|=0$. Put $\varepsilon_n= \frac12 \max(\|v_n \varphi-\varphi v_n\|, \delta)$ for every $n \in \N$. Take a sequence $z_n \in \mathrm{co}\{ uxu^* \mid u \in \mathcal{U}(N),   \| u \varphi - \varphi u \| < \delta-\varepsilon_n \}$ such that $\|z_n-y\|_\varphi \to 0$. Then for every $n \in \N$, we have $v_nz_nv_n^* \in C_\delta(x)$. Note that $\|v_nz_nv_n^*-v_nyv_n^*\|_\varphi \to 0$ and $\|v_n yv_n^*\|_\varphi \to \|y\|_\varphi$, thus $\|v_nz_nv_n^*\|_\varphi \to \|y \|_\varphi$. Since $y$ is the element of minimal norm in $C_\delta(x)$, this forces $\|v_nz_nv_n^*-y \|_\varphi \to 0$ and therefore $\|v_nyv_n^*-y\|_\varphi \to 0$. This shows that $y \in \rB(N \subset M, \varphi)$. Therefore, we have 
$$0=\Re(\varphi(y^*x)) \in \overline{\rm co} \{\Re(\varphi(ux^*u^*x))  \mid u \in \mathcal{U}(N), \;  \| u \varphi - \varphi u \| < \delta \}.$$
This implies that there exists $u \in \mathcal{U}(N)$ with $ \| u \varphi - \varphi u \| < \delta $ such that $\Re(\varphi(ux^*u^*x)) \leq \delta/2$. Thus we obtain 
$$\|uxu^*-x\|_\varphi^2 = \|uxu^*\|_\varphi^2+ \|x\|_\varphi^2 - 2\Re(\varphi(ux^*u^*x))  \geq 2\|x\|_\varphi^2-2\delta.$$
Hence if we suppose that $2 \delta \leq \|x\|_\varphi^2$ (which we can always do, without loss of generality), we obtain $\|uxu^*-x\|_\varphi \geq \|x\|_\varphi$.
\end{proof}

To prove Theorem \ref{thm: Connes' isomorphism and bicentralizer flow}, we use the ultraproduct technology. The crucial ingredient is Connes--St\o rmer transitivity theorem \cite{CS76} which shows that the ultrapower of a type $\III_1$ factor has a \emph{strictly homogeneous state space}  \cite[Theorem 4.20]{AH12}. Namely, for any pair of faithful normal states $\varphi, \psi$ on $M^{\omega}$, we can find $u \in \mathcal U(M^{\omega})$ such that $u\varphi u^{*}=\psi$. To construct the bicentralizer flow, we will also need the following lemma.

\begin{lem}\label{lem: family of partial isometries} 
Let $M$ be any nontrivial factor with strictly homogeneous state space. Let $\varphi \in M_{\ast}$ be any faithful state. Then $M_\varphi$ is a type $\II_1$ factor and for any $\lambda > 0$, we can find a finite family $v_1, \dots, v_n$ of partial isometries in $M$ such that $v_k \varphi = \lambda \varphi v_k$ for all $k=1,\dots, n$ and $\sum_{k=1}^{n} v_k^{*}v_k=1$. If $\lambda \geq 1$, then we can take $n=1$.
\end{lem}
\begin{proof} 
By \cite[Proposition 4.24]{AH12}, we know that $M_\varphi$ is a type $\II_1$ factor and by the proof of \cite[Proposition 4.22]{AH12}, we know that if $p,q \in M_{\varphi}$ are two nonzero projections, then we can find $v \in M$ such that $v^{*}v=p$, $vv^{*}=q$ and $v\varphi=\frac{\varphi(p)}{\varphi(q)}\varphi v$. If $\lambda \geq 1$, we can take $p=1$ and $q \in M_{\varphi}$ such that $\varphi(q)=\frac{1}{\lambda}$ and we obtain an isometry $v \in M$ such that $v\varphi=\lambda \varphi v$. If $\lambda \leq 1$, choose $n \geq 1$ such that $\frac1n \leq \lambda$. Then we can find a finite partition of unity $p_1, \dots, p_n$ in $M_{\varphi}$ and some projections $q_1, \dots, q_n$ (not necessarily orthogonal) such that $\varphi(p_k)=\lambda\varphi(q_k)$. Then we can find a family $v_k \in M$ such that $v_k^{*}v_k=p_k$, $v_kv_k^{*}=q
_k$ and $v_k \varphi = \frac{\varphi(p_k)}{\varphi(q_k)}\varphi v_k=\lambda \varphi v_k$ as we wanted.
\end{proof}

\begin{proof}[Proof of Theorem \ref{thm: Connes' isomorphism and bicentralizer flow}]
$(\rm i)$ Let $\omega_1,\omega_2\in \beta (\mathbf{N})\setminus \mathbf{N}$ be any nonprincipal ultrafilters. Let $u \in \mathcal U(N^{\omega_1})$ (resp.\ $v \in \mathcal U(N^{\omega_2})$) such that $u\varphi^{\omega_1}u^*=\psi^{\omega_1}$ (resp.\ $v\varphi^{\omega_2}v^*=\psi^{\omega_2}$). Then, inside $M^{\omega_2 \otimes \omega_1}$, we have $v^*u \in N^{\omega_2 \otimes \omega_1}_{\varphi^{\omega_2 \otimes \omega_1}}$. For every $x \in \rB(N \subset M,\varphi)$, we have $v^{*}ux=xv^{*}u$ which means that $uxu^{*}=vxv^{*}$. Since $uxu^{*} \in M^{\omega_1}$ and $vxv^{*} \in M^{\omega_2}$, Proposition \ref{prop: commuting square} shows that $uxu^{*}=vxv^{*}$ is an element of $M$. Thus, we have shown that for every $x \in \rB(N \subset M, \varphi)$, there exists an element $\beta_{\psi,\varphi}(x) \in M$ given by $\beta_{\psi, \varphi}(x)=uxu^{*}$ where $u \in N^{\omega}$ is any unitary such that $u\varphi^{\omega}u^{*}=\psi^{\omega}$ and $\omega \in \beta(\N) \setminus \N$ is any nonprincipal ultrafilter. In particular, if $w$ is a unitary in $N^{\omega}_{\psi^{\omega}}$, we can replace $u$ by $wu$, so that we have $\beta_{\psi, \varphi}(x)=wuxu^{*}w^{*}=w\beta_{\psi,\varphi}(x)w^{*}$. This shows that $\beta_{\psi, \varphi}(x) \in \rB(N \subset M,\psi)$. Now, if $(a_n)_{n \in \N}$ is a uniformly bounded sequence in $N$ such that $\|a_n \varphi-\psi a_n \| \to 0$, then it defines an element $a=(a_n)^{\omega} \in M^{\omega}$ such that $a\varphi^{\omega}=\psi^{\omega}a$ and so $u^{*}a \in N^{\omega}_{\varphi^{\omega}}$. This shows that $u^{*}ax=xu^{*}a$, that is, $ax=uxu^{*}a=\beta_{\psi, \varphi}(x)a$. Since the nonprincipal ultrafilter $\omega \in \beta(\N) \setminus \N$ is arbitrary, by Lemma \ref{lem: ultralimit test}, we conclude that $a_n x - \beta_{\psi,\varphi}(x)a_n \to 0$ $*$-strongly as $n \to \infty$. It is straightforward to check that $\beta_{\psi, \varphi}$ is a $*$-homomorphism and that $\beta_{\varphi_3, \varphi_2} \circ \beta_{\varphi_2, \varphi_1} = \beta_{\varphi_3, \varphi_1}$ for every faithful state $\varphi_i \in N_*$, $ i \in \{1, 2, 3\}$. This shows in particular that $\beta_{\psi, \varphi} : \rB(N \subset M, \varphi) \rightarrow \rB(N \subset M, \psi) $ is an isomorphism with inverse $\beta_{\varphi, \psi}$. Let $\rE_N: M \rightarrow N$ be any faithful normal conditional expectation and use it to extend $\varphi$ and $\psi$ to faithful normal states on $M$. Then we clearly have $\psi \circ \beta_{\psi,\varphi} = \varphi$. Since $N' \cap M$ is clearly fixed by $\beta_{\psi,\varphi}$, this implies that
$$\rE_{N' \cap M}^{\psi} \circ \beta_{\psi, \varphi}  = \rE_{N' \cap M}^{\varphi}.$$

$(\rm ii)$ Let $\omega_1,\omega_2 \in \beta(\mathbf{N}) \setminus \mathbf{N}$ be any nonprincipal ultrafilters and $\lambda>0$. By Lemma \ref{lem: family of partial isometries}, there exists a family $v_1,\dots, v_n$ of partial isometries in $N^{\omega_1}$ such that $v_k\varphi^{\omega_1}=\lambda \varphi^{\omega_1}v_k$ for all $k \in \{1,\dots,n\}$ and $\sum_{k=1}^{n}v_kv_k^{*}=1$. Similarly, let $w_1,\dots, w_m$ a family of partial isometries in $N^{\omega_2}$ such that $w_l\varphi^{\omega_2}=\lambda \varphi^{\omega_2}w_l$ for all $l \in \{1,\dots,m \}$ and $\sum_{l=1}^{m}w_lw_l^{*}=1$. Then inside $M^{\omega_2 \otimes \omega_1}$, we have $v_k^{*}w_l \in N^{\omega_2 \otimes \omega_1}_{\varphi^{\omega_2 \otimes \omega_2}}$ for all $k\in\{1, \dots, n\}$ and all $l \in \{1, \dots, m\}$. Then for all $x \in \rB(N \subset M, \varphi)$, we have
$$ xv_k^{*}w_l=v_k^{*}w_lx$$
and so
$$ v_kxv_k^{*}(w_lw_l^{*})=(v_kv_k^{*})w_lxw_l^{*}.$$
By summing over $k$ and $l$, we obtain
\begin{equation}\label{eq:flow}
\sum_{k=1}^{n}v_kxv_k^{*}=\sum_{l=1}^{m}w_l x w_l^{*}.
\end{equation}
But the left hand side of \eqref{eq:flow} lies in $M^{\omega_1}$ and the right hand side of \eqref{eq:flow} lies in $M^{\omega_2}$. Then they are both in $M$ by Proposition \ref{prop: commuting square} and the element $\beta_\lambda^\varphi(x)=\sum_{k=1}^{n}v_kxv_k^{*} \in M$ is independent of the choice of the nonprincipal ultrafilter $\omega \in \beta(\N) \setminus \N$ and the family $v_1, \dots, v_n \in N^{\omega}$ as above. In particular, if $u$ is a unitary in $M^{\omega}_{\varphi^{\omega}}$, then we can replace $v_k$ by $uv_k$ for all $k\in \{1, \dots, n\}$ and we obtain $\beta_\lambda^\varphi(x)=u\beta_\lambda^\varphi(x)u^*$. This shows that $\beta_\lambda^\varphi(x) \in \rB(N \subset M,\varphi)$. Let $(a_n)_{n \in \N}$ be a uniformly bounded sequence in $N$ such that $\lim_n \|a_n\varphi-\lambda \varphi a_n \|=0$. Then it defines an element $a=(a_n)^\omega \in N^\omega$ such that $a\varphi^\omega=\lambda\varphi^\omega a$. Then we have $v_k^*a \in N^\omega_{\varphi^\omega}$ for all $k \in \{1, \dots, n\}$. Thus, for all $x \in \rB(N\subset M, \varphi)$, we have $ v_k^*ax=xv_k^*a$ and so
$$ ax=\sum_{k=1}^n v_kv_k^*ax=\sum_{k=1}^n v_kxv_k^*a=\beta_\lambda^\varphi(x)a.$$
Since the nonprincipal ultrafilter $\omega \in \beta(\N) \setminus \N$ is arbitrary, by Lemma \ref{lem: ultralimit test}, we conclude that $a_n x - \beta_\lambda^\varphi(x)a_n \to 0$ $\ast$-strongly as $n \to \infty$. 

It is straightforward to check that $\beta_\lambda^\varphi$ is a unital $*$-homomorphism for all $\lambda > 0$ and that $\beta_\lambda^\varphi  \circ \beta_\mu^\varphi = \beta_{\lambda \mu}^\varphi$ for all $\lambda, \mu > 0$. This shows that $\beta^\varphi : \lambda \mapsto \beta_\lambda^\varphi$ is a one-parameter group of automorphisms of $\rB(N \subset M, \varphi)$. Also, one checks easily that $\beta^{\psi}_\lambda \circ \beta_{\psi, \varphi}=\beta_{\psi, \varphi} \circ \beta^{\varphi}_\lambda$ for all faithful normal states $\varphi, \psi \in N_*$. Extend $\varphi$ to a state on $M$ by using any faithful normal conditional expectation from $M$ to $N$. Then $\beta^{\varphi}$ is $\varphi$-preserving.  Indeed, for all $\lambda > 0$, we have
$$ \varphi(\beta_\lambda^{\varphi}(x))=\sum_{k = 1}^{n} \varphi^{\omega}(v_k x v_k^{*})=\sum_{k = 1}^{n} \lambda^{-1}\varphi^{\omega}(xv_k^{*}v_k  )=\sum_{k = 1}^{n} \lambda^{-1}\varphi(x)\varphi^{\omega}(v_k^{*}v_k )$$
because $x$ commutes with the factor $N^{\varphi^{\omega}}$. Since $\varphi^{\omega}(v_k^{*}v_k  )=\lambda \varphi^{\omega}(v_kv_k^{*}  )$, we obtain
$$ \varphi(\beta_\lambda^{\varphi}(x))=\sum_{k = 1}^{n}\varphi(x)\varphi^{\omega}(v_kv_k^{*}  )=\varphi(x).$$
Thus $\beta^{\varphi}_\lambda$ is $\varphi$-preserving and since $\beta^{\varphi}$ clearly fixes $N' \cap M$, we obtain
$$ \rE_{N' \cap M}^{\varphi} \circ \beta_{\lambda}^{\varphi} =\rE_{N' \cap M}^{\varphi}.$$

At this point, we have proved all items $(\rm i),(\rm ii), (\rm iii)$ and $(\rm iv)$. It only remains to check that $\beta^\varphi$ is indeed a flow in the sense that it is continuous with respect to the $u$-topology on $\rB(N \subset M, \varphi)$. Take a sequence $\lambda_n \in \R^{*}_+$ such that $\lambda_n \to 1$ and $\lambda_n \leq 1$. We have to show that $\beta^{\varphi}_{\lambda_n} \to \id_{\rB(N\subset M, \varphi)}$ with respect to the $u$-topology. Since $\beta^{\varphi}$ is $\varphi$-preserving, it is enough to show that $\beta^{\varphi}_{\lambda_n}(x) \to x$ strongly for all $x \in \rB(N \subset M, \varphi)$. Let $\omega_1 \in \beta(\N) \setminus \N$ be any nonprincipal ultrafilter and pick, for every $n \in \N$, a co-isometry $v_n \in N^{\omega_1}$ such that $v_n \varphi^{\omega_1}=\lambda_n \varphi^{\omega_1}v_n$ (possible because $\lambda_n \leq 1$). Let $\omega_2 \in \beta (\N) \setminus \N$ be any other nonprincipal ultrafilter. Since $\lambda_n \to 1$, then $v=(v_n)^{\omega_2}$ defines a co-isometry of $N^{\omega_2 \otimes \omega_1}$ with $v\varphi^{\omega_2 \otimes \omega_1}=\varphi^{\omega_2 \otimes \omega_1}v$. Since $x \in \rB(N \subset M, \varphi)$, we get $x=vxv^{*}=(v_nxv_n^{*})^{\omega_2}=(\beta_{\lambda_n}^{\varphi}(x))^{\omega_2}$. Since the nonprincipal ultrafilter $\omega_{2} \in \beta(\N) \setminus \N$ is arbitrary, by Lemma \ref{lem: ultralimit test}, we conclude that $\beta^{\varphi}_{\lambda_n}(x) \to x$ strongly as $n \to \infty$.
 \end{proof}

\begin{rem}\label{rem:flow-computation}
Although we strongly believe that the bicentralizer $\rB(M, \varphi)$ should always be trivial for all type ${\rm III_{1}}$ factors $M$ with separable predual, we point out that the relative bicentralizer (flow) need not be trivial in general.

\begin{itemize}
\item [$(\rm i)$] Let $N$ be any type ${\rm III_{1}}$ factor with separable predual and with trivial bicentralizer (e.g.\ $N = R_{\infty}$). Choose a faithful state $\varphi \in N_{\ast}$. Fix $\mu \in (0, 1)$, put $T = \frac{2 \pi}{- \log(\mu)}$ and define $M = N \rtimes_{\sigma_{T}^{\varphi}} \Z$. Extend $\varphi$ to $M$ by using the canonical conditional expectation $\rE_{N} : M \to N$. Then $M$ is a type ${\rm III_{\mu}}$ factor by \cite[Lemma 1]{Co85} and the inclusion $N \subset M$ is irreducible and with expectation. Since $\rB(N, \varphi) = \C 1$, a straightforward argument using the Fourier expansion shows that $\rB(N \subset M, \varphi) = \rL(\Z)$ and the relative bicentralizer flow $\beta^{\varphi} : \R^{*}_{+} \curvearrowright \rB(N \subset M, \varphi)$ is given by the action $\R^{*}_{+} \curvearrowright \R^{*}_{+}/\mu^{\Z}$. 

\item [$(\rm ii)$] We can upgrade the example given in item $(\rm i)$ to obtain an irreducible inclusion of type ${\rm III_{1}}$ factors with expectation and with nontrivial bicentralizer flow. Indeed, define $\mathcal M = (M, \varphi) \ast (R_{\infty}, \psi)$ to be the free product von Neumann algebra, where $\psi \in (R_{\infty})_{\ast}$ is any faithful state. By \cite[Theorem 4.1]{Ue10}, $\mathcal M$ is a type ${\rm III_{1}}$ factor. Moreover, \cite[Corollary 3.2]{Ue10} shows that $N \subset \mathcal M$ is irreducible and with expectation. Fix $\omega \in \beta(\N) \setminus \N$ any nonprincipal ultrafilter. Noticing that $\mathcal M \subset (M^{\omega}, \varphi^{\omega}) \ast (R_{\infty}, \psi)$, \cite[Corollary 3.2]{Ue10} implies that 
$$\rB(N \subset \mathcal M, \varphi) = (N^{\omega}_{\varphi^{\omega}})' \cap \mathcal M =  (N^{\omega}_{\varphi^{\omega}})' \cap \mathcal M \cap M^{\omega} =   (N^{\omega}_{\varphi^{\omega}})' \cap M = \rB(N \subset M, \varphi).$$ 
The relative bicentralizer flow $\beta^{\varphi} : \R^{*}_{+} \curvearrowright \rB(N \subset \mathcal M, \varphi)$ is still given by the action $\R^{*}_{+} \curvearrowright \R^{*}_{+}/\mu^{\Z}$.
\end{itemize}
\end{rem}

\section{Self-bicentralizing factors}
To prove Theorem \ref{thm: self-bicentralizing}, we use again the ultraproduct technology. We start by proving the following result.

\begin{prop} \label{prop: approximately inner}
Let $M$ be any factor with separable predual, $\theta \in \Aut(M)$ any automorphism and $\omega \in \beta (\N) \setminus \N$ any nonprincipal ultrafilter. Consider the following properties:
\begin{itemize}
\item [$(\rm i)$] $\theta \in \overline{\Inn}(M)$.
\item [$(\rm ii)$] There exists a unitary $u \in M^\omega$ such that $\theta(x)=uxu^*$ for all $x \in M$ and $\theta(\varphi)^\omega=u\varphi^\omega u^*$ for some (or any) faithful normal state $\varphi \in M_*$.
\item [$(\rm iii)$] There exists a nonzero partial isometry $v \in M^\omega$ such that $\theta(x)v=vx$ for all $x \in M$ and $v\varphi^\omega=\theta(\varphi)^\omega v$ for some (or any) faithful normal state $\varphi \in M_*$.
\item [$(\rm iv)$] There exists a nonzero partial isometry $v \in M^\omega$ such that $\theta(x)v=vx$ for all $x \in M$.
\item [$(\rm v)$] The automorphism $\theta \odot \id$ of $M \odot M^{\op}$ extends to an automorphism of the $\rC^*$-algebra $\rC^*_{\lambda \cdot \rho}(M)$ generated by the standard representation $\lambda \cdot \rho$ of $M \odot M^{\op}$ on $\rL^2(M)$.
\end{itemize}
Then we have $(\rm i) \Leftrightarrow (\rm ii) \Leftrightarrow (\rm iii) \Rightarrow (\rm iv) \Rightarrow (\rm v)$.
\end{prop}
\begin{proof}
We only prove $(\rm iv) \Rightarrow (\rm v)$ since the other implications are well-known (for the implication $(\rm iii) \Rightarrow (\rm ii)$, see the end of the proof of \cite[Theorem 1]{Co85} starting from Lemma 4). Let $v \in M^{\omega}$ be a nonzero partial isometry such that $\theta(x)v=vx$ for every $x \in M$. Note that $vv^{*} \in M' \cap M^{\omega}$ and so $\rE_M(vv^{*})=\lambda \in \R^{*}_+$. Take $T=\sum_i x_i \otimes y_i^{\op} \in M \odot M^{\op}$. For every unit vector $\xi \in \rL^{2}(M)$, we have in $\rL^{2}(M^{\omega})$ the following equalities:
$$ \|\sum_i \theta(x_i)\xi y_i \|=\frac{1}{\lambda} \|v^{*}\sum_i \theta(x_i)\xi^{\omega} y_i \|=\frac{1}{\lambda} \|\sum_i x_i v^{*}\xi^{\omega} y_i \|.$$
Since $$\| T \|_{\rC^{*}_{\lambda \cdot \rho}(M)}=\| T \|_{\rL^2(M)}=\| T^{\omega} \|_{\rL^2(M^{\omega})}$$
we obtain
$$\|\sum_i x_i v^{*}\xi^{\omega} y_i \| \leq \| v^{*} \xi^{\omega} \| \cdot \|T\|_{\rC^{*}_{\lambda \cdot \rho}(M)} =\lambda \|T\|_{\rC^{*}_{\lambda \cdot \rho}(M)}.$$
Thus we have shown that 
$$ \|\sum_i \theta(x_i)\xi y_i \| \leq \|T\|_{\rC^{*}_{\lambda \cdot \rho}(M)}$$
for all unit vectors $\xi \in \rL^{2}(M)$. This means that
$$\|(\theta \otimes \id)(T)\|_{\rC^{*}_{\lambda \cdot \rho}(M)} \leq \|T\|_{\rC^{*}_{\lambda \cdot \rho}(M)}.$$
By the same reasoning, replacing $\theta$ by $\theta^{-1}$ and $T$ by $(\theta \otimes \id)(T)$, we obtain
$$\|(\theta \otimes \id)(T)\|_{\rC^{*}_{\lambda \cdot \rho}(M)} = \|T\|_{\rC^{*}_{\lambda \cdot \rho}(M)}$$
as we wanted.
\end{proof}

\begin{proof}[Proof of Theorem \ref{thm: self-bicentralizing}]
$(\rm i)$ Let $\psi$ be any bicentralizing state on $M$. Let $\omega \in \beta(\N) \setminus \N$ be any nonprincipal ultrafilter and $u \in M^{\omega}$ any unitary such that $u\varphi^{\omega} u^{*}=\psi^{\omega}$. Then, with $M=\rB(M,\varphi)$ and $M=\rB(M,\psi)$, the $\ast$-homomorphism $\beta_{\psi, \varphi} : M \rightarrow M$ satisfies $\beta_{\psi, \varphi}(x)=uxu^{*}$ for all $x \in M$. This shows that $\beta_{\psi, \varphi} \in \overline{\Inn}(M)$ and we have $\beta_{\psi, \varphi}(\varphi)=\psi$. Conversely,  if $\alpha \in \overline{\Inn}(M)$, then we can find a unitary $u \in M^\omega$ such that $\alpha(x)=uxu^*$ for all $x \in M$ and $u\varphi^\omega u^*=\alpha(\varphi)^\omega$. This shows that $\beta_{\alpha(\varphi),\varphi}=\alpha$. Hence the map
$$ \overline{\Inn}(M) \ni \alpha \mapsto \alpha(\varphi) \in \Delta(M)$$
is a continous bijection with inverse
$$ \Delta(M) \ni \psi \mapsto \beta_{\psi, \varphi} \in \overline{\Inn}(M).$$
It only remains to show that this inverse is continous. Let $(\psi_n)_{n \in \N}$ be a sequence in $\Delta(M)$ which converges to $\psi \in \Delta(M)$. Let $\theta_n=\beta_{\psi_n, \varphi}$. We have to show that $\theta_n$ converges to $\beta_{\psi, \varphi}$. Since $(\theta_n(\varphi))^\omega=(\psi_n)^\omega=\psi^\omega$, $\theta=(\theta_n)^\omega$ defines an automorphism of $M^\omega$ and we just have to show that $\theta(x)=\beta_{\psi, \varphi}(x)$ for all $x \in M$. Let $\omega_0 \in \beta(\N) \setminus \N$ be any other nonprincipal utltrafilter, and for every $n \in \N$, take a unitary $u_n \in M^{\omega_0}$ such that $u_n \varphi^{\omega_0} u_n^*=\psi_n^{\omega_0}$. Then we have $\theta_n(x) = \beta_{\psi_n, \varphi}(x) =u_nxu_n^*$ for all $x \in M$ and all $n \in \N$. Let $u=(u_n)^\omega \in M^{\omega \otimes \omega_0}$. Then, by construction, we have $\theta(x)=uxu^*$ for all $x \in M$. But we have $u (\varphi^{\omega_0})^\omega u^*=(\psi_n^{\omega_0})^\omega=(\psi^{\omega_0})^\omega$ and so $uxu^*=\beta_{\psi, \varphi}(x)$ for all $x \in M$. This shows that $\theta(x)=\beta_{\psi, \varphi}(x)$ for all $x \in M$, as we wanted.

$(\rm ii)$ Since a continuous bijective homomorphism between Polish groups is automatically a homeomorphism, we just need to show that the homomorphism $\iota$ is bijective. Let $(g,h) \in \overline{\Inn}(M) \rtimes \Aut_\varphi(M)$ and assume that $g \circ h=\id$. Then $h=g^{-1}$ is approximately inner. Since $h(\varphi)=\varphi$, by item $(\rm i)$, we conclude that $h=\beta_{\varphi, \varphi}=\id$. This shows the injectivity. For the surjectivity, we just write every $\alpha \in \Aut(M)$ as $\alpha=\beta_{\alpha(\varphi), \varphi} \circ \left( \beta_{\varphi, \alpha( \varphi)} \circ \alpha \right)$ and we note that $\beta_{\alpha(\varphi), \varphi} \in \overline{\Inn}(M)$ by item $(\rm i)$ while $\beta_{\varphi, \alpha( \varphi)} \circ \alpha \in \Aut_\varphi(M)$. We conclude that $\iota$ is indeed an isomorphism of topological groups. Moreover, it shows that $\sigma_t^\varphi \in \overline{\Inn}(M)$ if and only if $\sigma_t^\varphi=\id$ and so $\sigma_t^\varphi \in \overline{\Inn}(M)$ if and only if $t = 0$ because $M$ is a type ${\rm III_1}$ factor.

$(\rm iii)$ Note first that $\beta_\lambda^\varphi \in \Aut_\varphi(M)$ and so $\beta_\lambda^\varphi \in \overline{\Inn}(M)$ if and only if $\beta_\lambda^\varphi=\id$, thanks to item $(\rm ii)$. If $M \cong M \ovt R_\lambda$ for $0 < \lambda < 1$, then we can find a co-isometry $v \in M' \cap M^\omega$ such that $v \varphi^\omega=\lambda \varphi^\omega v$. This shows that $\beta_\lambda^\varphi(x)v=vx=xv$ for all $x \in M$. Since $v$ is a co-isometry, we obtain $\beta_\lambda^\varphi=\id$. Conversely, take $v$ any nonzero partial isometry $v \in M^\omega$ such that $vv^*+v^*v=1$ and $v\varphi^\omega=\lambda \varphi^\omega v$. If $\beta_\lambda^\varphi=\id$, then we have $v \in M' \cap M^\omega$. This shows that $M$ satisfies Araki's property L$_\lambda'$ and so $M \cong M \ovt R_\lambda$ by \cite[Theorem 1.3]{Ar70}.

$(\rm iv)$ This follows from item $(\rm ii)$ in Theorem \ref{thm: Connes' isomorphism and bicentralizer flow} and Proposition \ref{prop: approximately inner}.
\end{proof}
 
\begin{remark}
We observe that $\Delta(M)$ is a dense $G_\delta$ subset of the set of all faithful normal states $\rS_{\rm{nf}}(M)$. The density of $\Delta(M)$ follows of course from Connes--St\o rmer transitivity. Since $\overline{\rm{Inn}}(M)$ is a Polish space, so is $\Delta(M)$, and since $\rS_{\rm{nf}}(M)$ is Polish as well, $\Delta(M)$ must be a $G_{\delta}$ subset.
\end{remark}

\section{Irreducible hyperfinite subfactors in inclusions of type III factors}

\subsection{Proof of Theorem \ref{thm: hyperfinite subfactor III_1}} In this section, we prove Theorem \ref{thm: hyperfinite subfactor III_1}. We first need to prove a few technical lemmas. The next result is a straightforward variation of \cite[Lemma 2.3]{Po81}.

\begin{lem} \label{convex_automorphism}
Let $M$ be any $\sigma$-finite von Neumann algebra and $\varphi \in M_{\ast}$ any faithful state. Let $G \subset \Aut_{\varphi}(M)$ be any subgroup. Let $x \in M$ be any element that satisfies $\rE_{M^G}^\varphi(x)=0$ where $\rE_{M^G}^\varphi : M \to M^{G}$ denotes the unique $\varphi$-preserving conditional expectation on the fixed point algebra $M^G$. Then there exists $\alpha \in G$ such that $\|x-\alpha(x)\|_\varphi \geq \|x\|_\varphi$.
\end{lem}

\begin{proof}
Without loss of generality, we may assume that $x \neq 0$. Denote by $y$ the unique element of minimal $\| \cdot \|_\varphi$-norm in the weakly closed convex hull $\mathcal C$ of $\{ \alpha(x) \mid \alpha \in G \}$. For every $\alpha \in G$, we have $\| \alpha(y)\|_\varphi=\|y \|_\varphi$ and $\alpha(y) \in \mathcal C$. By uniqueness of $y \in \mathcal C$, we obtain $\alpha(y)=y$ for every $\alpha \in G$. This shows that $y \in M^{G}$. On the other hand, we have $\rE^{\varphi}_{M^{G}}(\alpha(x))=\rE^{\varphi}_{M^{G}}(x)=0$ for all $\alpha \in G$. Thus, we also have $\rE^{\varphi}_{M^{G}}(y)=0$ and so $y=0$. By contradiction, if $\|x-\alpha(x)\|_\varphi < \|x\|_\varphi$ for every $\alpha \in G$, then we have $\|x\|_{\varphi}^{2} - 2 \Re(\varphi(x^{*}\alpha(x))) < 0$ for every $\alpha \in G$. By taking convex combinations and weak limits and since $y = 0$, we conclude that $\|x\|_{\varphi}^{2} \leq 0$, which is a contradiction. Therefore, there exists $\alpha \in G$ such that $\|x-\alpha(x)\|_\varphi \geq \|x\|_\varphi$.
\end{proof}

For all $0 < \lambda <1$, we denote by $\tau_{\lambda}$ the canonical periodic faithful normal state on the Powers factor $R_{\lambda}$ arising from the infinite tensor product decomposition
 $$(R_{\lambda}, \tau_{\lambda})=\overline{\bigotimes}_{n \in \N} (\mathbf M_2(\C), \omega_{\lambda})$$
where $\omega_{\lambda}$ is the state on $\mathbf M_2(\C)$ given by $\omega_{\lambda}(e_{11})=\lambda\omega_{\lambda}(e_{22})=\frac{\lambda}{1+\lambda}$ and $\omega_{\lambda}(e_{12})=\omega_{\lambda}(e_{21})=0$.

\begin{lem} \label{lem: R_lambda}
Let $N$ be any factor and $\varphi \in N_*$ any faithful state such that $N_{\varphi}$ is a type $\II_1$ factor. If $S$ is a finite subset of the point spectrum of $\varphi$, then we can find a subfactor $P \subset N$, globally invariant under $\sigma^\varphi$, such that 
$$ (P,\varphi|_P) \cong \overline{\bigotimes_{\lambda \in S}} (R_\lambda, \tau_\lambda).$$
\end{lem}

\begin{proof}
Let $S=\{ \lambda_1, \dots, \lambda_p\}$. Since $N_{\varphi}$ is a type $\II_1$ factor, we can find a partial isometry $v \in N$ such that $vv^{*}+v^{*}v=1$ and $v\varphi =\lambda_1 \varphi v$. Then $v$ generates a finite dimensional factor $F_1$ that is globally invariant under $\sigma^{\varphi}$ and such that $(F_1, \varphi |_{F_1}) \cong (\mathbf M_2(\C), \omega_{\lambda_1})$. Observe that $F_1' \cap N \cong eNe$ where $e=vv^{*} \in N_{\varphi}$ and so $(F_1' \cap N)_{\varphi}$ is again a type $\II_1$ factor and $S$ is in the point spectrum of $\varphi |_{F_1' \cap N}$. Thus, we can find $F_2 \subset F_1' \cap N$ that is globally invariant under $\sigma^{\varphi}$ and such that $(F_2, \varphi |_{F_2}) \cong (\mathbf M_2(\C), \omega_{\lambda_2})$. By repeating this procedure inductively, we obtain a finite dimensional factor $Q_1 \subset N$ that is globally invariant under $\sigma^{\varphi}$ and such that 
$$(Q_1, \varphi |_{Q_1}) \cong \bigotimes_{k=1}^{p} (\mathbf M_2(\C), \omega_{\lambda_k}).$$
By repeating this procedure with $Q_1' \cap N$, we contruct a sequence of mutually commuting factors $(Q_n)_{n \in \N}$ that are all globally invariant under $\sigma^{\varphi}$ and such that $(Q_n, \varphi |_{Q_n}) \cong (Q_1, \varphi |_{Q_1})$ for all $n \in \N$. Then $P=\bigvee_{n \in \N} Q_n$ provides the desired factor.
\end{proof}

\begin{lem} \label{lem: maximality eigenvalue}
Let $N$ be any factor and $\varphi \in N_*$ any faithful state such that $N_{\varphi}$ is a type $\II_1$ factor. If $\lambda$ is in the point spectrum of $\varphi$, then we can find a sequence $(a_n)_{n \in \N}$ in $N$ with $a_n \varphi = \lambda \varphi a_n$ for all $n \in \N$ such that
$ \sum_{n \in \N} a_n^{*}a_n= \lambda$ and $\sum_{n \in \N} a_na_n^{*}= 1$.
\end{lem}

\begin{proof}
Take a maximal subset $A \subset N \setminus \{0\}$ such that $a\varphi=\lambda \varphi a$ for all $a \in A$ and $\sum_{a \in F} a^{*}a \leq \lambda$ and $\sum_{a \in F} aa^{*} \leq 1$. Let $s = \lambda- \sum_{a \in F} a^{*}a $ and $t=1-\sum_{a \in F} aa^{*}$. Observe that $s,t$ are two positive elements of $N_{\varphi}$ and that $\varphi(s)=\lambda \varphi(t)$. Suppose that $s \neq 0$ and so $t \neq 0$. Since $N_{\varphi}$ is diffuse, we can find $\varepsilon > 0$ and two nonzero projections $p, q \in N_{\varphi}$ such that $\varepsilon p \leq s$, $\varepsilon q \leq t$ and $\varphi(p)=\lambda \varphi(q)$. Since $\lambda$ is in the point spectrum of $\varphi$ and $N_{\varphi}$ is a type $\II_1$ factor, we can find a partial isometry $v \in N$ such that $v \varphi = \lambda \varphi v$, $v^{*}v=p$ and $vv^{*}=q$. Then $A'=A \cup \{ \sqrt{\varepsilon}v \}$ contradicts the maximality of $A$.  Thus, we have $s=t=0$ and the set $A$, which is necessarily countable, provides the desired sequence.
\end{proof}

We will prove Theorem \ref{thm: hyperfinite subfactor III_1} by using inductively the following key lemma. The last part of the lemma will be useful to control the type of the hyperfinite subfactor we want to construct.

\begin{lem} \label{induction_step}
Let $N \subset M$ be any inclusion of $\sigma$-finite factors and $\varphi \in M_*$ any faithful state such that $N$ is globally invariant under $\sigma^{\varphi}$. Denote by $\rE_{N' \cap M}^\varphi : M \rightarrow N' \cap M$ the unique $\varphi$-preserving conditional expectation. Assume that $N$ is a type $\III_1$ factor and that $\rB(N \subset M, \varphi)^{\beta^\varphi}=N'\cap M$. 

Then for every $\varepsilon > 0$ and every $x \in M$ such that $\rE^\varphi_{N' \cap M}(x)=0$, we can find a finite dimensional subfactor $F \subset N$ and a faithful state $\psi \in F_*$ such that
\begin{itemize}
\item [$(\rm i)$]  $\| \varphi- \varphi \circ \rE_{F' \cap M} \| \leq \varepsilon$
\item [$(\rm ii)$] $\| \rE_{F' \cap M}(x) \|_\varphi \leq (1-2^{-13})  \|x \|_\varphi$
\end{itemize}
where $\rE_{F' \cap M} : M \rightarrow F' \cap M$ is the conditional expectation induced by $\psi$. 

Moreover, for any given $0 < \mu, \nu < 1$ such that $\mu/\nu \notin \Q$, we can choose $(F,\psi)$ so that
$$ (F,\psi) \cong (\mathbf M_2(\C), \tau)^{\otimes p} \otimes (\mathbf M_2(\C), \omega_\mu)^{\otimes q} \otimes (\mathbf M_2(\C), \omega_\nu)^{\otimes r}  $$
where $p,q,r \in \N$. We can always take $q,r \geq 1$. If $\rB(N \subset M, \varphi)^{\beta_\mu^\varphi}=N'\cap M$, then we can take $q \geq 1$ and $r=0$. If $\rB(N \subset M, \varphi)=N'\cap M$, then we can take $q=r=0$.
\end{lem}

\begin{proof}
Let $S=\{\mu, \nu\}$. If $\rB(N \subset M, \varphi)^{\beta_\mu^\varphi}=N'\cap M$, take $S=\{ \mu \}$. If $\rB(N \subset M, \varphi)=N'\cap M$, take $S=\emptyset$. In all cases, we have $\rB(N \subset M, \varphi)^{G}=N'\cap M$, where $G$ is the subgroup generated by $\beta^\varphi_s$ for $s \in S$.

Fix $\varepsilon > 0$ and $x \in M$ such that $\rE^\varphi_{N' \cap M}(x)=0$. Write $x=y+z$ where $y=x-\rE^\varphi_{\rB(N \subset M, \varphi)}(x)$ and $z=\rE^\varphi_{\rB(N \subset M, \varphi)}(x)$. By Proposition \ref{ultraproduct_bicentralizer}, we know that $\rE_{(N^{\omega}_{\varphi^{\omega}})' \cap M^{\omega}}(y)=0$. Lemma \ref{convex_automorphism} yields a unitary $u \in N^\omega_{{\varphi}^\omega}$ such that $\|uyu^*-y\|_{\varphi^\omega} \geq \|y \|_\varphi$. Note that $N^\omega_{\varphi^\omega}$ is a type $\II_1$ factor. Thus, up to a small perturbation (approximate $u$ by a unitary with finite spectrum and such that all of its spectral projections have dyadic traces), we may assume that $u$ is contained in some finite dimensional factor $\mathbf M_{2^p}(\C) \subset N^\omega_{\varphi^\omega}$. Then, by Lemma \ref{lem: R_lambda}, we can find a subfactor $Q \subset \mathbf M_{2^p}(\C)' \cap N^\omega$ that is globally invariant under $\sigma^{\varphi^\omega}$ and such that $$(Q,\varphi^\omega |_Q) \cong \overline{\bigotimes_{\mu \in S}}(R_\mu, \tau_\mu).$$
Put $P= \mathbf M_{2^p}(\C) \vee Q$. Let us show that
$$\| \rE^{\varphi^\omega}_{P' \cap M^\omega}(x) \|_{\varphi^\omega} \leq (1-2^{-11})  \|x \|_\varphi.$$

Assume first that $\| y \|_\varphi \geq 2^{-4} \|x \|_\varphi$. Since $u \in P$ and $u$ commutes with $z$, we have
\begin{align*}
 \|x \|_\varphi^{2}-\|\rE_{P' \cap M^\omega}^{\varphi^\omega}(x) \|_{\varphi^\omega}^{2} &= \|x-\rE_{P' \cap M^\omega}^{\varphi^\omega}(x) \|_{\varphi^\omega}^{2}\\
& \geq \frac{1}{4}\| u(x-\rE_{P' \cap M^\omega}^{\varphi^\omega}(x))-(x-\rE_{P' \cap M^\omega}^{\varphi^\omega}(x))u \|_{\varphi^\omega}^{2}\\
 & = \frac{1}{4}\| ux-xu \|_{\varphi^\omega}^{2}\\
&=\frac{1}{4}\| uy-yu \|_{\varphi^\omega}^{2} \\
&\geq \frac{1}{4}\|y\|_\varphi^2.
\end{align*}
Then $\|\rE_{P' \cap M^\omega}^{\varphi^\omega}(x) \|_{\varphi^\omega}^2 \leq \|x \|_\varphi^2- \frac{1}{4} \|y \|_\varphi^2$. Since $\| y \|_\varphi \geq 2^{-4} \|x \|_\varphi$, this shows that
$$\|\rE_{P' \cap M^\omega}^{\varphi^\omega}(x) \|_{\varphi^\omega}^2 \leq \|x \|_\varphi^2- \frac{1}{4} \|y \|_\varphi^2 \leq (1-2^{-10})\|x \|_\varphi^2$$
and so
$$\|\rE_{P' \cap M^\omega}^{\varphi^\omega}(x) \|_{\varphi^\omega} \leq \sqrt{1-2^{-10}} \|x \|_\varphi \leq (1-2^{-11}) \|x \|_\varphi.$$

Assume next that $\| y \|_\varphi \leq 2^{-4} \|x \|_\varphi$. By Lemma \ref{convex_automorphism}, since $\rB(N \subset M, \varphi)^{G}=N'\cap M$ and since $\rE^{\varphi}_{N' \cap M}(z)=0$, we can find some $\lambda \in \langle S \rangle$ such that $\|z-\beta_\lambda^\varphi(z)\|_\varphi \geq  \|z \|_\varphi$. By construction, $\lambda$ is an eigenvalue of $\Delta_{\varphi^\omega|_P}$ and the centralizer of $\varphi^\omega|_P$ is a factor. Thus, by Lemma \ref{lem: maximality eigenvalue}, we can find a sequence $(a_n)_{n \in \N}$ in $P$ with $a_n \varphi^\omega=\lambda \varphi^\omega a_n$ such that $\sum_{n \in \N} a_n^*a_n=\lambda$ and $\sum_{n \in \N} a_na_n^*=1$. Then we have
\begin{align*}
\|z-\rE_{P' \cap M^\omega}^{\varphi^\omega}(z) \|_{\varphi^\omega}^{2} 
&=\frac{1}{\lambda} \sum_{n \in \N} \|a_n(z-\rE_{P' \cap M^\omega}^{\varphi^\omega}(z) )\|_{\varphi^\omega}^{2} \\
&=\frac{1}{\lambda} \sum_{n \in \N} \|(\beta_\lambda^\varphi(z)-\rE_{P' \cap M^\omega}^{\varphi^\omega}(z) )a_n\|_{\varphi^\omega}^{2} \\
&=\frac{1}{\lambda} \sum_{n \in \N} \varphi^\omega(a_n^* \, |\beta_\lambda^\varphi(z)-\rE_{P' \cap M^\omega}^{\varphi^\omega}(z) |^2 \, a_n)\\
&=\sum_{n \in \N} \varphi^\omega(a_na_n^* \, |\beta_\lambda^\varphi(z)-\rE_{P' \cap M^\omega}^{\varphi^\omega}(z) |^2)\\
&= \|\beta_\lambda^\varphi(z)-\rE_{P' \cap M^\omega}^{\varphi^\omega}(z) \|_{\varphi^\omega}^{2}.
\end{align*}
Then, by the triangle inequality, we obtain $\|z-\rE_{P' \cap M^\omega}^{\varphi^\omega}(z) \|_{\varphi^\omega} \geq \frac{1}{2}\| z - \beta_\lambda^\varphi(z) \|_\varphi \geq \frac{1}{2}\|z\|_\varphi$. This implies that $\|\rE_{P' \cap M^\omega}^{\varphi^\omega}(z) \|_{\varphi^\omega}^2 \leq \frac{3}{4} \|z \|_\varphi^2$ and so $\|\rE_{P' \cap M^\omega}^{\varphi^\omega}(z) \|_{\varphi^\omega} \leq \frac{7}{8} \|z \|_\varphi $. Since $\| y \|_\varphi \leq \frac{1}{16} \|x \|_\varphi$, this shows that
$$\|\rE_{P' \cap M^\omega}^{\varphi^\omega}(x) \|_{\varphi^\omega} \leq \|\rE_{P' \cap M^\omega}^{\varphi^\omega}(y) \|_{\varphi^\omega} + \|\rE_{P' \cap M^\omega}^{\varphi^\omega}(z) \|_{\varphi^\omega} \leq \| y \|_\varphi + \frac{7}{8} \|z \|_\varphi  \leq \frac{15}{16} \|x \|_\varphi.$$
Thus, in all cases, we have
$$\| \rE^{\varphi^\omega}_{P' \cap M^\omega}(x) \|_{\varphi^\omega} \leq (1-2^{-11}) \|x \|_\varphi.$$

Now, by construction, it is clear that we can write $P$ as an increasing union of finite dimensional subfactors $(F_n)_{n \in \N}$ which are all globally invariant under $\sigma^{\varphi^{\omega}}$ and such that
$$(F_n, \varphi^{\omega}|_{F_n}) \cong (\mathbf M_{2^p}(\C),\tau) \otimes \bigotimes_{k=1}^{n} \left( \bigotimes_{\mu \in S} (\mathbf M_2(\C), \omega_\mu) \right).$$
We have $\lim_{n \to \infty} \| \rE^{\varphi^{\omega}}_{F_n'\cap M^{\omega}}(x) \|_{\varphi^{\omega}}=\| \rE^{\varphi^{\omega}}_{P'\cap M^{\omega}}(x) \|_{\varphi^{\omega}} \leq (1-2^{-11}) \|x \|_\varphi$. Thus, for $n \in \N$ large enough, we have $\| \rE^{\varphi^{\omega}}_{F_n'\cap M^{\omega}}(x) \|_{\varphi^{\omega}} \leq (1-2^{-12}) \|x \|_\varphi$. Finally, thanks to Lemma \ref{lift_matrix}, we can find a copy $F=F_n$ inside $N$ such that
\begin{itemize}
\item [$(\rm i)$]  $\| \varphi- \varphi \circ \rE_{F' \cap M} \| \leq \varepsilon$
\item [$(\rm ii)$] $\| \rE_{F' \cap M}(x) \|_\varphi \leq (1-2^{-13})  \|x \|_\varphi$.
\end{itemize}
where $\rE_{F' \cap M} : M \rightarrow F' \cap M$ is the conditional expectation induced by $\psi=\varphi^{\omega} |_{F}$.
\end{proof}

\begin{proof}[Proof of Theorem \ref{thm: hyperfinite subfactor III_1}]
Denote by $\rE_{N' \cap M}^{\varphi} : M \to N' \cap M$ the unique $\varphi$-preserving conditional expectation. Let $X=\{x_n \mid n \in \N\}$ be a countable $*$-strongly dense subset in $\{ x \in \Ball(M) \mid \rE_{N' \cap M}^{\varphi}(x)=0 \}$. Using Lemma \ref{induction_step}, we can construct by induction a sequence $(F_n)_{n \in \N}$ of finite dimensional subfactors of $N$ with faithful states $\psi_n \in (F_n)_*$ that satisfy the following properties: $F_0=\C1$, $F_{n+1} \subset (R_n)' \cap N$ with $R_n=F_0 \vee F_1 \vee \cdots \vee F_n$, and 
\begin{itemize}
\item [$(\rm i)$] $ \| \varphi \circ \rE_{R_{n+1}' \cap M} - \varphi \circ \rE_{R_{n}' \cap M} \| \leq 2^{-(n+1)}$
\item [$(\rm ii)$] $ \| \rE_{R_{n+1}' \cap M}(y_n)\|_\varphi \leq (1-2^{-13}) \|y_n \|_\varphi$
\end{itemize}
where $\rE_{R_n'\cap M}$ is the conditional expectation induced by $\psi_0 \otimes \psi_1 \otimes \dots \otimes \psi_n$ on $R_n$ and $y_n=\rE_{R_n'\cap M}(x_n)-\rE_{N' \cap M}(\rE_{R_n'\cap M}(x_n))$.

Thanks to the condition $(\rm i)$, we know that the sequence of states $\varphi_n=\varphi \circ \rE_{R_n' \cap M}$ is a Cauchy sequence in $M_{\ast}$ and so it converges to some state $\varphi_\infty \in M_*$. For all $n \in \N$, we moreover have 
$$\varphi_\infty = \lim_{k} \varphi \circ \rE_{R_k' \cap M} = \lim_{k} \varphi \circ \rE_{R_k' \cap M} \circ \rE_{R_n' \cap M} =\varphi_\infty \circ \rE_{R_n' \cap M}.$$ 
Denote by $e$ the support projection of $\varphi_\infty$ in $M$. Then for all $n \in \N$, we have $e \in R_n'\cap M$ and so $\varphi_n(e)=\varphi(e)$. Thus, by letting $n \to \infty$, we obtain $1=\varphi_\infty(e)=\varphi(e)$. This means that $e=1$ and so $\varphi_\infty$ is a faithful normal state on $M$. By construction, all the algebras $R_n$ are globally invariant under $\sigma^{\varphi_\infty}$. It follows that their union generates a hyperfinite factor $P=\bigvee_n R_n$ that is also globally invariant under $\sigma^{\varphi_\infty}$ and such that 
$$ (P, \varphi_\infty) \cong \bigovt_{n \in \N} (F_n, \psi_n).$$
Moreover, by the last part of Lemma \ref{induction_step}, we can always make $P$ of type $\III_1$. We can make $P$ of type $\III_\lambda$ if $\rB(N \subset M, \varphi)^{\beta^\varphi_\lambda}=N' \cap M$ for some $0 < \lambda <1$. We can make $P$ of type $\II_1$ if $\rB(N \subset M, \varphi)=N' \cap M$. 

Now, let us show that $P' \cap M=N' \cap M$. Take $x \in \Ball(P' \cap M)$ with $\rE_{N' \cap M}^{\varphi}(x)=0$. Then we can find a subsequence $(x_{n_k})_{k \in \N}$ such that $x_{n_k} \to x$ in the $*$-strong topology. Since $x \in R_{n_k}' \cap M$ and since $\rE_{R_{n_k}' \cap M}$ is $\varphi_{\infty}$-preserving for all $k \in \N$, we have $\rE_{R_{n_k}' \cap M}(x_{n_k}) \to x$ in the $*$-strong topology. Since $\rE_{N' \cap M}^{\varphi}(x)=0$, we also have $\rE_{N' \cap M}^{\varphi}(\rE_{R_{n_k}' \cap M}(x_{n_k})) \to 0$ in the $*$-strong topology. This shows that $y_{n_k} \to x$ in the $*$-strong topology. And again, since $x \in R_{n_k+1}' \cap M$, we obtain $\rE_{R_{n_k+1}' \cap M}(y_{n_k}) \to x$ in the $*$-strong topology. Therefore, by taking the limit in condition $(\rm ii)$, we obtain $\|x\|_\varphi \leq (1-2^{-13}) \|x\|_\varphi$. Thus, $x=0$ as we wanted.
\end{proof}

\begin{proof}[Proof of Application \ref{app-almost-periodic}]
Denote by $\rE_{Q} : M \to Q$ the (unique) faithful normal conditional expectation and choose a faithful state $\varphi \in M_{\ast}$ such that $\varphi \circ \rE_{Q} = \varphi$ and such that $\varphi|_{Q}$ is almost periodic. Then $\rB(N \subset M, \varphi)^{\beta^{\varphi}} \subset Q' \cap M = \C1$ and the conclusion follows from Theorem~\ref{thm: hyperfinite subfactor III_1}. 
\end{proof}

\begin{proof}[Proof of Corollary \ref{kadison bicentralizer}]
This is a consequence of Theorem \ref{thm: hyperfinite subfactor III_1} and \cite[Theorem 3.2]{Po81}.
\end{proof}

\subsection{A type ${\rm III_\lambda}$ analogue of Theorem \ref{thm: hyperfinite subfactor III_1}}
In this subsection, we prove a type $\III_\lambda$ analogue of Theorem \ref{thm: hyperfinite subfactor III_1} when $0 < \lambda < 1$. Since the proof is similar (and easier), we only sketch it.

\begin{thm} \label{thm: hyperfinite subfactor III_lambda}
Let $N \subset M$ be any inclusion of von Neumann algebras with separable predual and with expectation. Assume that $N$ is a type $\III_\lambda$ factor $(0 < \lambda < 1)$ with a $T$-periodic faithful state $\varphi \in N_*$ where $T=\frac{2\pi}{- \log \lambda}$. Then there exists a hyperfinite type $\III_\lambda$ subfactor $P \subset N$ that is globally invariant under $\sigma^\varphi$ and such that $P' \cap M=N' \cap M$.

We can find a hyperfinite type $\II_1$ subfactor with expectation $P \subset N$ such that $P' \cap M=N' \cap M$ if and only if $(N_\varphi)' \cap M=N' \cap M$. In that case, we can moreover take $P \subset N_\varphi$.
\end{thm} 

The proof relies on the following analogue of Lemma \ref{induction_step}.
\begin{lem}
Let $N \subset M$ be any inclusion of von Neumann algebras with separable predual and with expectation. Assume that $N$ is a type $\III_\lambda$ factor $(0 < \lambda < 1)$ with a $T$-periodic faithful state $\varphi \in N_*$ where $T=\frac{2\pi}{- \log \lambda}$. Extend $\varphi$ to $M$ by using any faithful normal conditional expectation on $N$.  Denote by $\rE^\varphi_{N' \cap M} : M \to N' \cap M$ the unique $\varphi$-preserving conditional expectation.

Then for every $x \in M$ such that $\rE^\varphi_{N' \cap M}(x)=0$, we can find a finite dimensional factor $F \subset N$ that is globally invariant by $\sigma^\varphi$ and such that
$$ \| \rE^\varphi_{F' \cap M}(x) \|_\varphi \leq (1-2^{-13})\|x \|_\varphi. $$
\end{lem}

\begin{proof}
Let $B=(N_\varphi)' \cap M$. The algebra $B$ will play the role of $\rB(N \subset M, \varphi)$. There exists a $\varphi$-preserving automorphism $\theta$ of $B$ such that for all $x \in B$, all $a \in N$ and all $n \in \Z$, we have
$$ a\varphi=\lambda^n \varphi a \quad \Rightarrow \quad ax=\theta^n(x)a.$$
Moreover, we have $B^{\theta}=N' \cap M$. 

Let $x=y+z$ where $y=x-\rE_B^\varphi(x)$ and $z=\rE^\varphi_B(x)$. By Lemma \ref{convex_automorphism}, we can find a unitary $u \in N_{\varphi}$ such that $\|uyu^*-y\|_{\varphi} \geq \|y \|_\varphi$. Since $N_\varphi$ is a type $\II_1$ factor, up to a small perturbation (approximate $u$ by a unitary with finite spectrum and such that all of its spectral projections have dyadic traces), we may assume that $u$ is contained in some finite dimensional factor $\mathbf M_{2^p}(\C) \subset N_{\varphi}$. Then, by Lemma \ref{lem: R_lambda}, we can find a subfactor $Q \subset \mathbf M_{2^p}(\C)' \cap N$ that is globally invariant under $\sigma^{\varphi}$ and such that 
$$(Q,\varphi |_Q) \cong (R_\lambda, \tau_\lambda).$$
Put $P= \mathbf M_{2^p}(\C) \vee Q$. Let us show that 
$$\| \rE^{\varphi}_{P' \cap M}(x) \|_{\varphi} \leq (1-2^{-13})\|x \|_\varphi.$$

Assume first that $\| y \|_\varphi \geq 2^{-4} \|x \|_\varphi$. Then, as in Lemma \ref{induction_step}, we can show that
$$\| \rE^{\varphi}_{P' \cap M}(x) \|_{\varphi} \leq (1-2^{-11})\|x \|_\varphi.$$
Assume next that $\| y \|_\varphi \leq 2^{-4} \|x \|_\varphi$.  By Lemma \ref{convex_automorphism}, since $B^{\theta}=N' \cap M$ and since $\rE^{\varphi}_{N' \cap M}(z)=0$, we can find $k \in \Z$ such that $\|z-\theta^{k}(z) \|_\varphi \geq \|z \|_\varphi$. By Lemma \ref{lem: maximality eigenvalue}, we can find a sequence $a_n \in P$ with $a_n \varphi = \lambda^{k} \varphi a_n$ such that $\sum_{n \in \N} a_n^{*}a_n=\lambda^{k}$ and $\sum_{n \in \N} a_n a_n^{*}=1$. Then we conclude as in Lemma \ref{induction_step} by showing that 
$$ \|z-\rE^{\varphi}_{P' \cap M}(z)\|_\varphi = \| \theta^{k}(z)-\rE^{\varphi}_{P' \cap M}(z)\|_\varphi.$$
This finishes the proof.
\end{proof}

\begin{proof}[Proof of Theorem \ref{thm: hyperfinite subfactor III_lambda}]
The proof of the first part of the theorem is the same as in Theorem \ref{thm: hyperfinite subfactor III_1}. It is easier though because all the finite dimensional factors we construct are globally invariant under $\sigma^{\varphi}$ and the sequence of states $(\varphi_n)_{n \in \N}$ is constant.

For the second part, assume that $P \subset N$ is a hyperfinite type $\II_1$ subfactor with expectation such that $P' \cap M=N' \cap M$. Observe that $P' \cap N = \C1$. Let $\psi$ be a faithful normal state on $N$ such that $P \subset N_\psi$. We have $N_{\psi}' \cap N = \C 1$. Since $N$ is of type $\III_\lambda$, we know that $\sigma_T^{\psi}$ is inner and so there exists a unitary $u \in N$ such that $\sigma_{T}^{\psi} = \Ad(u)$. Note that $u \in (N_\psi)'  \cap N = \mathbf C 1$. This implies that $\psi$ is $T$-periodic. We have $(N_\psi)' \cap M = N' \cap M$ and since $\psi$ and $\varphi$ are stably unitarily conjugate (see \cite[Th\'eor\`eme 4.3.2]{Co72}), we also have $(N_\varphi)' \cap M =N' \cap M$.
\end{proof}

We derive the following consequence of Theorem \ref{thm: hyperfinite subfactor III_lambda} that is related to \cite[Problem 2]{Po85}.

\begin{cor}\label{cor:popa-problem}
Let $M$ be any type $\III_\lambda$ factor $(0 < \lambda < 1)$ with separable predual and with a $T$-periodic faithful state $\varphi \in M_*$ where $T=\frac{2\pi}{- \log \lambda}$. Then there exists an irreducible hyperfinite type $\III_\lambda$ subfactor $P \subset M$ that is globally invariant under $\sigma^\varphi$.
\end{cor}

We end this subsection with a characterization of Kadison's property for irreducible inclusions of factors $N \subset M$ with expectation and with separable predual where $N$ is a type ${\rm III_\lambda}$ factor $(0 < \lambda < 1)$.

\begin{thm}\label{thm:Kadison-III-lambda}
Let $N \subset M$ be any irreducible inclusion of factors with separable predual and with expectation. Assume that $N$ is a type $\III_\lambda$ factor $(0 < \lambda < 1)$. Let $\varphi \in N_*$ be any  $T$-periodic faithful state where $T=\frac{2\pi}{- \log \lambda}$ and $\psi$ any dominant weight on $N$.

The following assertions are equivalent:
\begin{itemize}
\item [$(\rm i)$] $(N_\varphi)' \cap M = \C 1$.
\item [$(\rm ii)$] $(N_\psi)' \cap M = \mathcal Z(N_\psi)$
\item [$(\rm iii)$] The inclusion $N \subset M$ satisfies Kadison's property.
\end{itemize}
\end{thm}

\begin{proof}
$(\rm i) \Rightarrow (\rm iii)$ This follows from \cite[Theorem 3.2]{Po81}.

$(\rm iii) \Rightarrow (\rm ii)$ This follows from Proposition \ref{prop-masa} and the fact that $\left( N \subset M \right) \cong \left( N^{\infty} \subset M^{\infty}\right)$.

$(\rm ii) \Rightarrow (\rm i)$ Since $N$ is a type ${\rm III_\lambda}$ factor, $\sigma_T^\psi$ is an inner automorphism of $N$. Let $u \in \mathcal U(N)$ be a unitary such that $\sigma_T^\psi = \Ad(u)$. We have $u \in (N_\psi)' \cap N$ and so $u \in \mathcal Z(N_\psi)$ by Connes--Takesaki relative commutant theorem \cite[Chapter II, Theorem 5.1]{CT76}. Choose a nonsingular positive selfadjoint operator $h$ affiliated with $\mathcal Z(N_\psi)$ such that $u = h^{{\rm i}T}$. Define the faithful normal semifinite weight $\phi$ on $N$ by the formula $\phi = \psi (h^{-1} \, \cdot \,)$. Then we have $\sigma_T^\phi = \Ad(h^{-{\rm i}T}) \circ \sigma_T^\psi = \id_N$ and so $\phi$ is $T$-periodic. We moreover have $N_\psi \subset N_\phi$. Using the assumption, we have
$$(N_\phi)' \cap M = (N_\psi)' \cap M = \mathcal Z(N_\psi) \subset N_\phi.$$
Then \cite[Th\'eor\`eme 4.2.6]{Co72} shows that $(N_\phi)' \cap M = \C1$. Since $\phi$ and $\varphi$ are stably unitarily conjugate (see \cite[Th\'eor\`eme 4.3.2]{Co72}), we also have $(N_\varphi)' \cap M =\C 1$.
\end{proof}

\section{Discrete inclusions}

\subsection{Proof of Theorem \ref{thm-relative-bicentralizer}}

Let $N \subset M$ be any inclusion of von Neumann algebras with separable predual and with expectation $\rE_{N} : M \to N$. Assume that $N$ is a type $\III_1$ factor. Using \cite{Ta71}, we may regard the standard form $(N, \rL^{2}(N), J_{N}, \rL^{2}(N)^{+})$ as a substandard form of the standard form $(M, \rL^{2}(M), J_{M}, \rL^{2}(M)^{+})$ via the isometric embedding $$\rL^{2}(N)^{+} \hookrightarrow \rL^{2}(M)^{+} : \xi \mapsto (\langle \, \cdot \, \xi , \xi\rangle \circ \rE_{N})^{1/2}.$$ 
By this embedding, we have that $\rL^{2}(N)$ is an $N$-$N$-sub-bimodule of $\rL^{2}(M)$ and $J_{M} |_{\rL^{2}(M)} = J_{N}$ (see \cite[p.\ 317 Equation (8)]{Ta71}). Since no confusion is possible, we simply write $J = J_{M}$. Denote by $e_{N} : \rL^{2}(M) \to \rL^{2}(N)$ the corresponding Jones projection. Since $e_{N} \in \langle M, N\rangle = (JNJ)' \cap \mathbf B(\rL^{2}(M))$, for every $x\in M$, every $y \in N$ and every $\xi \in \rL^{2}(N)^{+}$, we have
\begin{equation}\label{eq:Jones}
e_{N}(x \xi y) = e_{N} \, JyJ (x\xi) = JyJ \, e_{N} (x\xi) = JyJ \, \rE_{N}(x)\xi = \rE_{N}(x)\xi y.
\end{equation}
We refer to \cite[Proposition A.2]{HI15} for further details.

Let $\varphi \in M_{\ast}$ be any faithful state such that $\varphi \circ \rE_{N} = \varphi$.  Since the inclusion $N' \cap M \subset M$ is globally invariant under $\sigma^{\varphi}$, we denote by $\rE^{\varphi}_{N' \cap M} :  M \to N' \cap M$ the unique $\varphi$-preserving conditional expectation and  we may regard the standard form $(N' \cap M, \rL^{2}(N' \cap M), J_{N' \cap M}, \rL^{2}(N' \cap M)^{+})$ as a substandard form of the standard form $(M, \rL^{2}(M), J, \rL^{2}(M)^{+})$ via the isometric embedding $$\rL^{2}(N' \cap M) \hookrightarrow \rL^{2}(M) : x \xi_{\varphi} \mapsto x \xi_{\varphi}.$$ 
Observe that the corresponding Jones projection $e_{N' \cap M}^{\varphi} : \rL^{2}(M) \to \rL^{2}(N' \cap M)$ satisfies $J e_{N' \cap M}^{\varphi} = e_{N' \cap M}^{\varphi} J$ and $e_{N' \cap M}^{\varphi}(x \xi_{\varphi}) = \rE_{N' \cap M}^{\varphi}(x)\xi_{\varphi}$ for every $x \in M$.

\begin{proof}[Proof of Theorem \ref{thm-relative-bicentralizer}]

$(\rm i) \Rightarrow (\rm ii)$ By Theorem \ref{thm: hyperfinite subfactor III_1}, there exists a hyperfinite type $\II_1$ subfactor with expectation $P \subset N$  such that $P' \cap M=N' \cap M$. Take a faithful state $\varphi \in N_*$ such that $P \subset N_\varphi$. Put $N^{\infty} = N \ovt \mathbf B(\rL^{2}(\R))$ and $M^{\infty} = M \ovt \mathbf B(\rL^{2}(\R))$. Since $N$ is a type ${\rm III}$ factor, there is a canonical $\ast$-isomorphism $\pi : M \to M^{\infty}$ such that $\pi(N) = N^{\infty}$ and $\pi \circ \rE_{N} \circ \pi^{-1} = \rE_{N} \otimes \id_{\mathbf B(\rL^{2}(\R))}$. Choose a faithful normal semifinite weight $\omega$ on $\mathbf B(\rL^{2}(\R))$ such that $[\rD \omega : \rD \Tr] = \lambda_{t}$ for every $t \in \R$. Then $\psi = \varphi \otimes \omega$ is a dominant weight on $N^{\infty}$ (see \cite[Theorem 1.3]{CT76}). The proof of \cite[Theorem 4.2, page 57]{Po95} shows that $(N^{\infty})_{\psi}' \cap M^{\infty} = (N^{\infty})' \cap M^{\infty} = N' \cap M$ (the argument does not use the facts that $N \subset M$ has finite index or $N' \cap M$ is finite dimensional).  It remains to prove that the inclusion $(N^{\infty})_{\psi} \subset M^{\infty}$ satisfies the weak relative Dixmier property.

Let $x \in M^{\infty}$ be any element.  Denote by $A = \{\lambda_{t} : t \in \R\}\dpr$ and observe that $A \subset \mathbf B(\rL^{2}(\R))$ is maximal abelian. Then we have $Q' \cap M^{\infty}=(N' \cap M) \ovt A$ where $Q=P \ovt A$. Since $Q$ is amenable, the inclusion $Q \subset M^{\infty}$ satisfies the weak relative Dixmier property. Therefore, we have
$$ \mathcal{K}_Q(x) \cap (N' \cap M) \ovt A  \neq \emptyset.$$
Choose $z \in \mathcal K_{Q}(x) \cap \left( (N'\cap M) \ovt A \right)$ and observe that $\mathcal K_{(N^{\infty})_{\psi}}(z) \subset \mathcal K_{(N^{\infty})_{\psi}}(x)$ because $Q \subset (N^{\infty})_{\psi}$.

Regard $(N'\cap M) \ovt A \subset (N'\cap M) \vee (N^{\infty})_{\psi} = ((N^{\infty})'\cap M^{\infty}) \vee (N^{\infty})_{\psi}$. Since $N^{\infty}$ is a factor with expectation, we have $(N'\cap M) \vee (N^{\infty})_{\psi} \cong (N'\cap M) \ovt (N^{\infty})_{\psi}$. Since $(N^{\infty})_{\psi} \subset (N'\cap M) \ovt (N^{\infty})_{\psi}$ is a semifinite von Neumann subalgebra with expectation, the inclusion $(N^{\infty})_{\psi} \subset (N'\cap M) \ovt (N^{\infty})_{\psi}$ satisfies the weak relative Dixmier property (see e.g.\ \cite[Corollary 1.3]{Po98}). Since $(N^{\infty})_{\psi}' \cap \left( (N'\cap M) \vee (N^{\infty})_{\psi}\right) = N' \cap M $, we have
$$\mathcal K_{(N^{\infty})_{\psi}}(z) \cap N'\cap M \neq \emptyset.$$
Therefore, we obtain $\mathcal K_{(N^{\infty})_{\psi}}(x) \cap N'\cap M \neq \emptyset$ and so the inclusion $(N^{\infty})_{\psi} \subset M^{\infty}$ satisfies the weak relative Dixmier property.

$(\rm ii) \Rightarrow (\rm i)$ We divide the proof into a series of claims following the proof of \cite[Theorem 2.1]{Ha85}.

Firstly, we observe that using the assumption in item $(\rm i)$, we obtain the following straightforward  generalization of \cite[Lemma 2.7]{Ha85} (see also \cite[Theorem 4.2, Step ${\rm III}$]{Po95}). 

\begin{claim}\label{claim1}
Let $\delta > 0$, $\varphi \in M_{\ast}$ any faithful state such that $\varphi \circ \rE_{N} = \varphi$ and $x \in M$ any element such that $\rE^{\varphi}_{N' \cap M}(x) = 0$. Then there exists a sequence $(a_{i})_{i}$ in $\Ball(N)$ such that the following three properties hold:
\begin{itemize}
\item [$(\rm i)$] $\spec_{\sigma^{\varphi}}(a_{i}) \subset [- \delta, \delta]$ for every $i \in \N$;
\item [$(\rm ii)$] $\sum_{i \in \N} a_{i}^{\ast} a_{i} = 1$;
\item [$(\rm iii)$] $\sum_{i \in \N} \|a_{i} x - x a_{i}\|_{\varphi}^{2} \geq \frac12 \|x\|_{\varphi}^{2}$.
\end{itemize}
\end{claim}

Next, using Claim \ref{claim1} and the fact that the inclusion $N \subset M$ is discrete, we prove the following generalization of \cite[Lemma 2.9]{Ha85}. Let us point out that this is the main technical novelty of the proof of Theorem \ref{thm-relative-bicentralizer} compared to the proofs of \cite[Theorem 3.1]{Ha85} (where $N = M$) and \cite[Theorem 4.2]{Po95} (where $N \subset M$ has finite index).

\begin{claim}\label{claim2}
Let $\delta > 0$, $\xi \in \rL^2(N)^+$ any unit $M$-cyclic vector and $\eta \in \rL^{2}(M)$ any unit vector such that $J\eta =  \eta$ and $\eta \in \left( (N' \cap M) \xi\right)^{\perp}$. Then there exists $a \in \Ball(N) \setminus \{0\}$ such that 
\begin{equation}\label{eq:microscopic-a}
\|a\xi\|^{2} + \|a \eta\|^{2} < 8 \|a \eta - \eta a\|^{2}  \quad \text{ and } \quad \|a\xi - \xi a\|^{2} < \delta \|a \eta - \eta a\|^{2}.
\end{equation}
\end{claim}

\begin{proof}[Proof of Claim \ref{claim2}]
Put $\varphi = \langle \, \cdot \, \xi , \xi\rangle \in M_{\ast}$ and $\psi = \langle \, \cdot \, \eta, \eta\rangle \in M_{\ast}$. Observe that $\varphi \in M_{\ast}$ is a faithful state such that $\varphi \circ \rE_{N} = \varphi$.

Firstly, assume that $\psi|_{N}$ is not bounded by some multiple scalar of $\varphi|_{N}$. Then the same reasoning as in \cite[Lemma 2.9, page 113]{Ha85} shows that we may choose projections $p, q \in N$ such that $\psi(p) > (16/\delta) \varphi(p)$ and $\psi(q) = (1/16) \psi(p)$. Since $N$ is a type ${\rm III}$ factor, there exists a partial isometry $v \in N$ such that $v^{*}v = p - q$ and $vv^{*} = q$. Then the  same calculations in the standard form $\rL^{2}(M)$ as in \cite[Lemma 2.9, pages 113-114]{Ha85} show that $a = v$ satisfies \eqref{eq:microscopic-a}.

Secondly, assume that $\psi|_{N}$ is bounded by some multiple scalar of $\varphi|_{N}$. Then the mapping $S_{0} : M \xi \to \rL^{2}(M) : x \xi \mapsto \rE_{N}(x) \eta$ extends to well defined bounded operator $S \in \mathbf B(\rL^{2}(M))$ such that $S \xi = \eta$, $Se_{N} = S$ and $S \in N' \cap \mathbf B(\rL^{2}(M))$. Define $T = J S J \in (JNJ)' \cap \mathbf B(\rL^{2}(M)) = \langle M, N\rangle$. We have $T \xi = J S J \xi = J S \xi = J \eta = \eta$ and $T e_{N} = J S J e_{N} = JS e_{N} J = JSJ = T$.

Since the inclusion $N \subset M$ is discrete in the sense of \cite[Definition 3.7]{ILP96}, there exists an increasing sequence of projections $(p_{n})_{n}$ in $N' \cap \langle M, N\rangle$ such that $\widehat \rE_{N}(p_{n}) \in M$ for every $n \in \N$ and $p_{n} \to 1$ strongly. Observe that $\widehat \rE_{N}((p_{n} T)(p_{n} T)^{\ast}) = \widehat \rE_{N}(p_{n} T T^{\ast} p_{n}) \leq \|T\|_{\infty}^{2} \,\widehat \rE_{N}(p_{n})$ and so $(p_{n}T)^{\ast} \in \mathfrak n_{\widehat \rE_{N}}$. Applying \cite[Proposition 2.2]{ILP96} to $(p_{n}T)^{*}$ and taking the adjoint, we obtain that $\widehat \rE_{N}(p_{n} T e_{N}) e_{N} = p_{n} T e_{N} = p_{n} T$. Letting $x_{n} = \widehat \rE_{N}(p_{n} T e_{N}) \in M$, we have $x_{n} e_{N} = p_{n} T e_{N} = p_{n} T$. It follows that  $x_{n} \xi = x_{n} e_{N} \xi = p_{n} T \xi = p_{n} \eta$. We then have $$\lim_{n} \|x_{n}\|_{\varphi} = \lim_{n}\|x_{n} \xi\| = \lim_{n} \|p_{n} \eta\| = \|\eta\| = 1$$ and 
$$\lim_{n} \|\rE^{\varphi}_{N' \cap M}(x_{n})\|_{\varphi} = \lim_{n} \|e^{\varphi}_{N' \cap M}(x_{n} \xi)\|  = \lim_{n} \|e^{\varphi}_{N' \cap M}(p_{n} \eta)\| = \|e^{\varphi}_{N' \cap M}( \eta)\| =  0.$$

Choose $n \in \N$ large enough so that $\|x_{n} - \rE^{\varphi}_{N' \cap M}(x_{n})\|_{\varphi}^{2} \geq (15/16)^{2}$. Put 
$$\delta_{1} = \min \left\{ (\delta/8)^{1/2}, 2^{-9/2} \|x_{n} \|_{\infty}^{-1}\right\}.$$ 
Applying Claim \ref{claim1} to $x = x_{n} -  \rE^{\varphi}_{N' \cap M}(x_{n})$ with $\delta_{1}$, there exists a sequence $(a_{i})_{i}$ in $\Ball(N)$ such that the following three properties hold:
\begin{itemize}
\item [$(\rm i)$] $\spec_{\sigma^{\varphi}}(a_{i}) \subset [- \delta_{1}, \delta_{1}]$ for every $i \in \N$;
\item [$(\rm ii)$] $\sum_{i \in \N} a_{i}^{\ast} a_{i} = 1$;
\item [$(\rm iii)$] $\sum_{i \in \N} \|a_{i} x - x a_{i}\|_{\varphi}^{2} \geq \frac12 \left(\frac{15}{16}\right)^{2}$.
\end{itemize}
Observe that item $(\rm i)$ implies that $\|a_{i}\xi - \xi a_{i}\|^{2} \leq \delta_{1}^{2}\|a_{i}\xi\|^{2} \leq (\delta/8) \|a_{i} \xi\|^{2}$ for every $i \in \N$.

Since $p_{n} \in N' \cap \langle M, N\rangle$ and $a_{i} \in N$, we have
\begin{align}\label{eq:claim2-1}
p_{n}(a_{i} \eta - \eta a_{i}) &= a_{i} \, p_{n} \eta - p_{n} \eta \, a_{i} \\ \nonumber
&= a_{i} \, x_{n} \xi - x_{n} \xi \, a_{i} \\ \nonumber
&= (a_{i} x_{n} - x_{n} a_{i}) \xi +  x_{n} (a_{i} \xi - \xi a_{i}) \\ \nonumber
&= (a_{i} x - x a_{i}) \xi +  x_{n} (a_{i} \xi - \xi a_{i}).
\end{align}
The reasoning is now identical to the one in \cite[Lemma 2.9, page 112]{Ha85}. Using the triangle inequality in $\bigotimes_{\N} \rL^{2}(M)$ and \eqref{eq:claim2-1}, we have
\begin{align*}
\left( \sum_{i \in \N} \|a_{i} \eta - \eta a_{i}\|^{2} \right)^{1/2} &\geq \left( \sum_{i \in \N} \|p_{n}(a_{i} \eta - \eta a_{i})\|^{2} \right)^{1/2} \\
&\geq \left( \sum_{i \in \N} \|a_{i} x - x a_{i}\|_{\varphi}^{2} \right)^{1/2} - \|x_{n}\|_{\infty} \cdot \left( \sum_{i \in \N} \|a_{i} \xi - \xi a_{i}\|^{2} \right)^{1/2}\\
&\geq \frac{15}{16 \sqrt{2}} - \delta_{1} \|x_{n}\|_{\infty} \\
&\geq  \frac{15}{16 \sqrt{2}} - \frac{1}{16 \sqrt{2}} = \frac{7}{8 \sqrt{2}}.
\end{align*}
On the other hand, we have
\begin{equation*}
\sum_{i \in \N} \left( \|a_{i} \xi\|^{2} + \|a_{i} \eta\|^{2} + 8 \delta^{-1} \|a_{i}\xi - \xi a_{i}\|^{2} \right) \leq 1 + 1 + 1 = 3.
\end{equation*}
Since $8 \sum_{i \in \N} \|a_{i} \eta - \eta a_{i}\|^{2} \geq 49/16 > 3$ and $\sum_{i \in \N} \left( \|a_{i} \xi\|^{2} + \|a_{i} \eta\|^{2} + 8 \delta^{-1} \|a_{i}\xi - \xi a_{i}\|^{2}  \right) \leq 3$, there exists $i \in \N$ such that 
$$8 \|a_{i} \eta - \eta a_{i}\|^{2} > \|a_{i} \xi\|^{2} + \|a_{i} \eta\|^{2} + 8 \delta^{-1} \|a_{i}\xi - \xi a_{i}\|^{2}.$$
Thus, $a = a_{i} \in \Ball(N) \setminus \{0\}$ satisfies \eqref{eq:microscopic-a}.
\end{proof}

Using Claim \ref{claim2} and proceeding exactly as in \cite[Lemmas 2.10, 2.11, 2.12, 2.13]{Ha85}, we obtain the following straightforward generalization of \cite[Lemma 2.13]{Ha85} (see also \cite[Theorem 4.2, Step ${\rm VI}$]{Po95}). 

\begin{claim}\label{claim3}
Let $\delta > 0$, $\xi \in \rL^{2}(N)^{+}$ any unit $M$-cyclic vector and $\eta \in \rL^{2}(M)$ any unit vector such that $J\eta =  \eta$ and $\eta \in \left( (N' \cap M) \xi\right)^{\perp}$. Then there exists a nonzero projection $p \in N$ such that 
\begin{equation}\label{eq:microscopic-p}
\|p\xi\|^{2} + \|p \eta\|^{2} < 2^{7} \|p \eta - \eta p\|^{2}  \quad \text{ and } \quad \|p\xi - \xi p\|^{2} < \delta \|p \eta - \eta p\|^{2}.
\end{equation}
\end{claim}

Whenever $K \subset H$ is a closed subspace of a Hilbert space $H$ and $\eta \in H$ is a unit vector such that $\eta \notin K$, we define the {\em angle} between $\eta$ and $K$ by the formula 
$$\angle_{K}(\eta) = \arccos(\|P_{K}(\eta)\|)$$
where $P_{K} : H \to K$ denotes the orthogonal projection. 

Using Claim \ref{claim3}, we prove the following generalization of \cite[Lemma 2.14]{Ha85} (see also \cite[Theorem 4.2, Step ${\rm VII}$]{Po95}). Let us point out that unlike \cite[Theorem 4.2, Step ${\rm VII}$]{Po95}, $N' \cap M$ is not assumed to be finite dimensional. For that reason, we provide a detailed proof.

\begin{claim}\label{claim4}
Let $\delta > 0$, $\xi \in \rL^{2}(N)^{+}$ any unit $M$-cyclic vector and $\eta \in \rL^{2}(M)$ any unit vector such that $J\eta =  \eta$ and $\eta \not\in \overline{(N' \cap M) \xi}$.  
Then there exists a nonzero projection $p \in N$ such that 
\begin{equation}\label{eq:microscopic-p-theta}
\|p\xi\|^{2} + \|p \eta\|^{2} < \frac{2^{10}}{\sin^{2}\theta} \|p \eta - \eta p\|^{2}  \quad \text{ and } \quad \|p\xi - \xi p\|^{2} < \delta \|p \eta - \eta p\|^{2}
\end{equation}
where $\theta = \angle_{\overline{(N' \cap M)\xi}}(\eta)$.
\end{claim}

\begin{proof}[Proof of Claim \ref{claim4}]
It is sufficient to consider $\delta < 1$. Put $\varphi = \langle \, \cdot \, \xi, \xi\rangle \in M_{\ast}$ and observe that $\varphi \circ \rE_{N} = \varphi$. Since $J e^{\varphi}_{N' \cap M}  = e^{\varphi}_{N' \cap M}  J$ and $J \eta = \eta$, we have $J e^{\varphi}_{N' \cap M} (\eta) = J e^{\varphi}_{N' \cap M} J (\eta) = e^{\varphi}_{N' \cap M}(\eta)$. 

Put $\theta = \angle_{\rL^{2}(N' \cap M)}(\eta) = \arccos(\|e^{\varphi}_{N' \cap M}(\eta)\|)$. Put $\zeta = \frac{\eta - e^{\varphi}_{N' \cap M}(\eta)}{\|\eta - e^{\varphi}_{N' \cap M}(\eta)\|}$ and observe that $\|\zeta\| = 1$, $J \zeta = \zeta$, $\zeta \perp (N' \cap M)\xi$ and $\eta = e^{\varphi}_{N' \cap M}(\eta) + \sin\theta \, \zeta$. By Claim \ref{claim3}, there exists a nonzero projection $p \in N$ such that 
\begin{equation}\label{eq:claim4-1}
\|p\xi\|^{2} + \|p \zeta\|^{2} < 2^{7} \|p \zeta - \zeta p\|^{2}  \quad \text{ and } \quad \|p\xi - \xi p\|^{2} < \frac14 \delta \sin^{2}\theta \, \|p \zeta - \zeta p\|^{2}.
\end{equation}
We have 
\begin{equation}\label{eq:claim4-2}
p \eta - \eta p = p e^{\varphi}_{N' \cap M}(\eta) - e^{\varphi}_{N' \cap M}(\eta) p+ \sin\theta \, (p \zeta - \zeta p).
\end{equation}

Since $ e^{\varphi}_{N' \cap M}(\eta) \in \rL^{2}(N' \cap M) = \overline{(N' \cap M) \xi}$, we may choose a sequence $(x_{n})_{n}$ in $N' \cap M$ such that $ e^{\varphi}_{N' \cap M}(\eta) = \lim_{n} x_{n} \xi$. Note however that $(x_{n})_{n}$ need not be uniformly bounded. Since $p \in N$ and $\xi \in \rL^2(N) \subset \rL^2(M)$, we have $p \xi - \xi p \in \rL^{2}(N) \subset \rL^{2}(M)$. Then we obtain
\begin{align}\label{eq:claim4-3}
\|p e^{\varphi}_{N' \cap M}(\eta) - e^{\varphi}_{N' \cap M}(\eta) p\|^{2} &= \lim_{n} \|p \, x_{n} \xi - x_{n}\xi \, p\|^{2} \\ \nonumber
&= \lim_{n} \| x_{n} (p \xi - \xi p)\|^{2} \quad (\text{since } px_{n} = x_{n}p, \forall n \in \N)\\ \nonumber
&= \lim_{n} \langle x_{n}^{*}x_{n}(p \xi - \xi p), e_{N}(p \xi - \xi p)\rangle \quad (\text{since } p \xi - \xi p \in \rL^{2}(N))\\ \nonumber
&= \lim_{n} \langle e_{N}(x_{n}^{*}x_{n}(p \xi - \xi p)), p \xi - \xi p\rangle \\ \nonumber
&= \lim_{n} \langle \rE_{N}(x_{n}^{*}x_{n})(p \xi - \xi p), p \xi - \xi p\rangle \quad (\text{using } \eqref{eq:Jones} \text{ with } y = p \in N)\\ \nonumber
&=  \lim_{n} \|x_{n}\xi\|^{2} \cdot  \|p \xi - \xi p\|^{2} \quad (\text{since } \rE_{N}(x_{n}^{*}x_{n}) = \varphi(x_{n}^{*}x_{n})1, \forall n \in \N)\\ \nonumber
&=  \|e^{\varphi}_{N' \cap M}(\eta)\|^{2} \cdot  \|p \xi - \xi p\|^{2} \\ \nonumber
&= \cos^{2}\theta \,  \|p \xi - \xi p\|^{2}.
\end{align}
Likewise, we obtain 
\begin{equation}\label{eq:claim4-4}
\|p e^{\varphi}_{N' \cap M}(\eta) \|^{2} = \cos^{2}\theta \,  \|p \xi \|^{2}.
\end{equation}

Combining \eqref{eq:claim4-1},\eqref{eq:claim4-2},\eqref{eq:claim4-3}, we obtain
\begin{align}\label{eq:claim4-5}
 \|p\xi - \xi p\|^{2} &< \frac14 \delta \sin^{2}\theta \, \|p \zeta - \zeta p\|^{2} \\ \nonumber
 &\leq \frac12 \delta \left( \|p \eta - \eta p\|^{2} + \|p e^{\varphi}_{N' \cap M}(\eta) - e^{\varphi}_{N' \cap M}(\eta) p\|^{2} \right) \\ \nonumber
 &= \frac12 \delta \left( \|p \eta - \eta p\|^{2} + \cos^{2}\theta \,  \|p \xi - \xi p\|^{2} \right) \\ \nonumber
 &\leq \frac12 \delta \left( \|p \eta - \eta p\|^{2} +  \|p \xi - \xi p\|^{2} \right) \\ \nonumber
& \leq  \delta \|p \eta - \eta p\|^{2} \quad (\text{since } \delta \leq 1).
\end{align}
Using \eqref{eq:claim4-4}, we have
\begin{align}\label{eq:claim4-6}
\|p \eta\| &\leq \| p e^{\varphi}_{N' \cap M}(\eta)\| + \sin \theta \, \|p \zeta \| \\ \nonumber
&\leq \cos \theta \, \| p \xi\| + \sin \theta \, \|p \zeta \| \\ \nonumber
&\leq \left( \| p \xi\|^{2} + \|p \zeta \|^{2}\right)^{1/2}.
\end{align}
Combining \eqref{eq:claim4-1},\eqref{eq:claim4-2},\eqref{eq:claim4-3},\eqref{eq:claim4-5},\eqref{eq:claim4-6} we obtain
\begin{align*}
\|p\xi\|^{2} + \|p \eta\|^{2} &\leq 2 \left(\| p \xi\|^{2} + \|p \zeta \|^{2} \right) \\
&< 2^{8} \|p \zeta - \zeta p\|^{2} \\
&\leq \frac{2^{9}}{\sin^{2}\theta}\left( \|p \eta - \eta p\|^{2} + \|p e^{\varphi}_{N' \cap M}(\eta) - e^{\varphi}_{N' \cap M}(\eta) p\|^{2} \right) \\
&\leq \frac{2^{9}}{\sin^{2}\theta}\left( \|p \eta - \eta p\|^{2} +  \|p \xi - \xi p\|^{2} \right)  \\
&\leq \frac{2^{10}}{\sin^{2}\theta}  \|p \eta - \eta p\|^{2}.
\end{align*}
This finishes the proof of Claim \ref{claim4}.
\end{proof}

Using Claim \ref{claim4}, we prove the following generalization of \cite[Lemma 2.15]{Ha85} (see also \cite[Theorem 4.2, Step ${\rm VIII}$]{Po95}). Since our notion of angle is different from the one in \cite[Lemma 2.14]{Ha85}, we provide a detailed proof and then explain how to use the proof of \cite[Lemma 2.15]{Ha85}.

\begin{claim}\label{claim5}
Let $\delta > 0$, $\xi \in \rL^2(N)^+$ any unit $M$-cyclic vector and $\eta \in \rL^{2}(M)$ any unit vector such that $J\eta =  \eta$ and $\eta \in \left( (N' \cap M) \xi\right)^{\perp}$. Then there exists a family of pairwise orthogonal projections $(e_i)_i$ in $N$ such that $\sum_{i \in I} e_i = 1$ and 
\begin{equation}\label{eq:maximal}
2^{-18} \leq  \left\| \eta - \sum_{i \in I}  e_i\eta e_i \right\|^{2}  \quad \text{ and } \quad \left\|\xi - \sum_{i \in I}e_i \xi e_i \right\|^{2} \leq \delta.
\end{equation}
\end{claim}

\begin{proof}[Proof of Claim \ref{claim5}]
Following the proof of \cite[Lemma 2.15]{Ha85}, we denote by $\mathcal F$ the nonempty inductive set of all families of nonzero pairwise orthogonal projections $(p_{i})_{i \in I}$ in $N$  that satisfy
\begin{align*}
\|\xi - p\xi p\|^{2} + \|\eta - p\eta p\|^{2} &\leq 2^{14} \left\|\eta - p\eta p - \sum_{i \in I} p_{i}\eta p_{i} \right\|^{2}  \\
\left\|\xi - p\xi p - \sum_{i \in I} p_{i}\xi p_{i} \right\|^{2} &\leq \delta \left\|\eta - p\eta p - \sum_{i \in I} p_{i}\eta p_{i} \right\|^{2}
\end{align*}
where $p = 1 - \sum_{i \in I} p_{i}$. Let $(q_{i})_{i \in I}$ be a maximal element in $\mathcal F$ and put $q - \sum_{i \in I} q_{i}$. We show that the family of pairwise orthogonal projections $((q_{i})_{i \in I}, q)$ satisfies the conclusion of Claim \ref{claim5}. We have to show that $\left\|\eta - q\eta q - \sum_{i \in I} q_{i}\eta q_{i} \right\|^{2} \geq 2^{-18}$.

Assume by contradiction that $\left\|\eta - q\eta q - \sum_{i \in I} q_{i}\eta q_{i} \right\|^{2} < 2^{-18}$. Since 
\begin{align*}
\|\xi - q\xi q\|^{2} + \|\eta - q\eta q\|^{2} &\leq 2^{14} \left\|\eta - q\eta q - \sum_{i \in I} q_{i}\eta q_{i} \right\|^{2} \\
&< 2^{14} \cdot 2^{-18} = \frac{1}{16},
\end{align*}
we have
\begin{align}\label{eq:claim5-0}
\max \left\{\|\xi - q \xi q\|, \|\eta - q \eta q\| \right\} &\leq \frac14 \\ \nonumber
\min \left\{\|q \xi q\|, \|q \eta q\| \right\} &\geq \frac34.
\end{align}
Then $q \neq 0$. Denote by $\rE_{qNq} : qMq \to qNq$ the faithful normal conditional expectation obtained by restricting $\rE_{N}$ to $qMq$. We regard the standard form $(qNq, \rL^2(qNq), J_{qNq}, \rL^2(qNq)^+)$ as a substandard form of the standard form $(qMq, \rL^{2}(qMq), J_{qMq}, \rL^{2}(qMq)^{+})$ via the isometric embedding $$\rL^{2}(qNq)^{+} \hookrightarrow \rL^{2}(qMq)^{+} : \zeta \mapsto (\langle \, \cdot \, \zeta , \zeta\rangle \circ \rE_{qNq})^{1/2}.$$
Observe that $\left( \rL^2(qNq) \subset \rL^2(qMq) \right) = qJqJ \left( \rL^2(N) \subset \rL^2(M)\right)$ and $J_{qMq} = J \, qJqJ$ where $J = J_{M}$. 

Put $\xi_q = \frac{q \xi q}{\|q \xi q\|} \in \rL^2(qNq)^+$ and $\eta_q = \frac{q \eta q}{\|q\eta q\|} \in \rL^2(qMq)$. Observe that $\xi_q \in \rL^2(qNq)^+$ is a $qMq$-cyclic vector in $\rL^2(qMq)$. Put $\varphi_q = \langle \, \cdot \, \xi_q, \xi_q\rangle \in (qMq)_\ast$ and observe that $\varphi_{q} \circ \rE_{qNq} = \varphi_{q}$. Since the inclusion $(qNq)' \cap qMq \subset qMq$ is globally invariant under $\sigma^{\varphi_{q}}$, we may regard the standard form $((qNq)' \cap qMq, \rL^{2}((qNq)' \cap qMq), J_{(qNq)' \cap qMq}, \rL^{2}((qNq)' \cap qMq)^{+})$ as a substandard form of the standard form $(qMq, \rL^{2}(qMq), J_{qMq}, \rL^{2}(qMq)^{+})$ via the isometric embedding $$\rL^{2}((qNq)' \cap qMq) \hookrightarrow \rL^{2}(qMq) : x \xi_q \mapsto x \xi_q.$$ 
Denote by $e^{\varphi_q}_{(qNq)' \cap qMq} : \rL^{2}(qMq) \to \rL^{2}((qNq)' \cap qMq)$ the corresponding Jones projection. 

We show that the angle $\angle_{\overline{((qNq)' \cap qMq)\xi_{q}}}(\eta_{q}) = \arccos(\|e^{\varphi_q}_{(qNq)' \cap qMq}(\eta_q)\|)$, which generalizes the angle $\theta = \arccos(\langle \eta_{q}, \xi_{q}\rangle)$ that appears in \cite[Lemma 2.15]{Ha85}, satisfies $$\cos(\angle_{\overline{((qNq)' \cap qMq)\xi_{q}}}(\eta_{q})) = \|e^{\varphi_q}_{(qNq)' \cap qMq}(\eta_q)\| \leq \frac12.$$ 
Since $(qNq)' \cap qMq = (N' \cap M)q$ by \cite[Lemma 2.1]{Po81}, we may choose a sequence $(x_n)_n$ in $N' \cap M$ such that $e^{\varphi_q}_{(qNq)' \cap qMq}(\eta_q) = \lim_n x_nq \, \xi_q$. Note however that $(x_{n})_{n}$ need not be uniformly bounded. Regarding $\xi_q \in \rL^{2}(qNq)^{+} \subset \rL^2(N)^+$, for every $n \in \N$, we have
\begin{align}\label{eq:claim5-1}
\|x_n q \, \xi_q\|^2 &= \frac{1}{\|q\xi q\|}\langle x_n \xi, x_n \xi_q\rangle \quad (\text{since } qx_n = x_n q)\\ \nonumber
&= \frac{1}{\|q\xi q\|} \langle x_n^*x_n \xi, e_N(\xi_q)\rangle  \quad (\text{since } \xi_q \in \rL^2(N))\\ \nonumber
&= \frac{1}{\|q\xi q\|}\langle e_N(x_n^*x_n \xi), \xi_q\rangle \\ \nonumber
&= \frac{1}{\|q\xi q\|}\langle \rE_N(x_n^*x_n) \xi, \xi_q\rangle \\ \nonumber
&= \|x_n \xi \|^2 \quad (\text{since } \rE_N(x_n^*x_n) = \varphi(x_n^*x_n)1).
\end{align}

As before, since the inclusion $N' \cap M \subset M$ is globally invariant under $\sigma^{\varphi}$, we may regard the standard form $(N' \cap M, \rL^{2}(N' \cap M), J_{N' \cap M}, \rL^{2}(N' \cap M)^{+})$ as a substandard form of the standard form $(M, \rL^{2}(M), J, \rL^{2}(M)^{+})$ via the isometric embedding $$\rL^{2}(N' \cap M) \hookrightarrow \rL^{2}(M) : x \xi_{\varphi} \mapsto x \xi_{\varphi}.$$ 
Denote by $e_{N' \cap M}^{\varphi} : \rL^{2}(M) \to \rL^{2}(N' \cap M)$ the corresponding Jones projection. Regarding $\eta_{q} \in \rL^{2}(qMq) \subset \rL^{2}(M)$, we can compare $\|e^{\varphi_q}_{(qNq)' \cap qMq}(\eta_q) \|$ and $\|e^{\varphi}_{N' \cap M}(\eta_q) \|$ using \eqref{eq:claim5-1}. Indeed, we have
\begin{align*}
\|e^{\varphi_q}_{(qNq)' \cap qMq}(\eta_q) \|^2 &=  |\langle e^{\varphi_q}_{(qNq)' \cap qMq}(\eta_q), \eta_q \rangle| \\
&= \lim_n |\langle x_nq \, \xi_q, \eta_q \rangle| \\
&= \lim_n \frac{1}{\|q\xi q\|} |\langle x_n \xi, \eta_q \rangle| \quad (\text{since } qx_n = x_nq, \forall n \in \N)\\
&= \lim_n \frac{1}{\|q\xi q\|} |\langle e_{N' \cap M}^\varphi(x_n \xi), \eta_q \rangle| \quad (\text{since } x_{n} \in N' \cap M, \forall n \in \N)\\
&= \lim_n \frac{1}{\|q\xi q\|} |\langle x_n \xi, e_{N' \cap M}^\varphi(\eta_q) \rangle| \\
&\leq \frac{1}{\|q\xi q\|}\limsup_n \|x_n \xi\| \cdot \|e_{N' \cap M}^\varphi(\eta_q)\| \\
&\leq \frac{1}{\|q\xi q\|}\limsup_n \|x_nq \, \xi_q\| \cdot \|e_{N' \cap M}^\varphi(\eta_q)\| \quad (\text{using } \eqref{eq:claim5-1}) \\
&= \frac{1}{\|q\xi q\|} \|e^{\varphi_q}_{(qNq)' \cap qMq}(\eta_q) \| \cdot \|e_{N' \cap M}^\varphi(\eta_q)\|
\end{align*}
and thus we obtain
\begin{equation}\label{eq:claim5-2}
\|e^{\varphi_q}_{(qNq)' \cap qMq}(\eta_q) \| \leq \frac{1}{\|q\xi q\|} \|e_{N' \cap M}^\varphi(\eta_q)\|.
\end{equation}
Since by assumption we have $e^{\varphi}_{N' \cap M}(\eta) = 0$, combining \eqref{eq:claim5-0} and \eqref{eq:claim5-2}, we obtain
\begin{align}\label{eq:claim5-3}
\|e^{\varphi_q}_{(qNq)' \cap qMq}(\eta_q) \| &\leq \frac{1}{\|q\xi q\| \cdot \|q \eta q\|} \|e_{N' \cap M}^\varphi(q\eta q)\| \\ \nonumber
&= \frac{1}{\|q\xi q\| \cdot \|q \eta q\|} \|e_{N' \cap M}^\varphi(q\eta q - \eta)\| \\ \nonumber
&\leq \frac{1}{\|q\xi q\| \cdot \|q \eta q\|} \|q\eta q - \eta\| \\ \nonumber
&\leq \frac14 \left(\frac43\right) < \frac12.
\end{align}

Since $N$ is a type ${\rm III}$ factor, there exists an isometry $v \in N$ such that $vv^{*} = q$. Then $\Ad(v^{*}) : qMq \to M$ is a $\ast$-isomorphism such that $\Ad(v^{*})(qNq) = N$ and such that $\Ad(v^{*}) \circ \rE_{qNq} \circ \Ad(v) = \rE_{N}$. We can now use \eqref{eq:claim5-3} in combination with Claim \ref{claim4} applied to the inclusion $\left( qNq \subset qMq \right) \cong \left( N \subset M\right)$ and the vectors $\xi_{q} \in \rL^{2}(qNq)^{+}$ and $\eta_{q} \in \rL^{2}(qMq)$ with $\delta >0$, and apply the rest of the proof of \cite[Lemma 2.15, pages 124-127]{Ha85} to obtain a contradiction.
\end{proof}

Using Claim \ref{claim5}, we obtain the following straightforward generalization of \cite[Lemma 2.16]{Ha85}.

\begin{claim}\label{claim6}
Let $\delta > 0$, $\xi \in \rL^{2}(N)^{+}$ any unit $M$-cyclic vector and $\eta \in \rL^{2}(M)$ any unit vector such that $\eta \in \left( (N' \cap M) \xi\right)^{\perp}$. Then there exists a nonzero projection $p \in N$ such that 
\begin{equation}\label{eq:macroscopic-p}
2^{-21} \leq  \left\| p\eta - \eta p\right\|^{2}  \quad \text{ and } \quad \left\|p\xi - \xi p \right\|^{2} \leq \delta.
\end{equation}
\end{claim}

We can now finish the proof of $(\rm ii) \Rightarrow (\rm i)$. Let $\varphi \in M_{\ast}$ be any faithful state such that $\varphi \circ \rE_{N} = \varphi$. Then the inclusions $N' \cap M \subset \rB(N \subset M, \varphi) \subset M$ are globally invariant under $\sigma^{\varphi}$. Let $x \in \rB(N \subset M, \varphi)$ be any element such that $\rE^{\varphi}_{N' \cap M}(x) = 0$. We  show that $x = 0$. Put $\xi = \xi_{\varphi} \in \rL^{2}(N)^{+}$ and $\eta = x \xi \in \rL^{2}(M)$. Observe that $\eta \in \left( (N' \cap M) \xi\right)^{\perp}$. For every $n \geq 1$, applying Claim \ref{claim6}, there exists a projection $p_{n} \in N$ such that 
\begin{equation*}
2^{-21} \|\eta\|^{2} \leq  \left\| p_{n}\eta - \eta p_{n}\right\|^{2}  \quad \text{ and } \quad \left\|p_{n}\xi - \xi p_{n} \right\|^{2} \leq \frac1n.
\end{equation*}
Then we have
\begin{align*}
\liminf_{n} \|p_{n}x - xp_{n}\|_{\varphi} &= \liminf_{n} \|p_{n} x\xi - xp_{n}\xi\| \\
&\geq \liminf_{n} \left(\|p_{n} \eta - \eta p_{n}\| - \|x\|_{\infty} \cdot \|p_{n}\xi - \xi p_{n}\| \right) \\
&= \liminf_{n} \|p_{n} \eta - \eta p_{n}\| \\
& \geq 2^{-21/2} \|\eta\|.
\end{align*}
Since $x \in \rB(N \subset M, \varphi)$, we have $\lim_n \|p_{n}x - xp_{n}\|_{\varphi} = 0$ and so $\|\eta\| = 0$. Therefore, we have $x  = 0$. This finally shows that $\rB(N \subset M, \varphi) = N' \cap M$ and finishes the proof of $(\rm ii) \Rightarrow (\rm i)$.
\end{proof}

\subsection{A generalization of Connes--Takesaki relative commutant theorem}

In the setting of discrete irreducible inclusions with expectation, we prove a generalization of Connes--Takesaki relative commutant theorem \cite[Chapter II, Theorem 5.1]{CT76} (see also \cite[Theorem 4.3]{Po95} when $N \subset M$ is a finite index inclusion of type ${\rm III_{1}}$ factors).

\begin{thm}\label{thm-relative-connes-takesaki}
Let $N \subset M$ be any discrete irreducible inclusion of factors with separable predual and with expectation $\rE_N : M \to N$. Assume that $N$ is a type $\III_1$ factor. Let $\psi$ be any dominant weight on $N$ and extend it to a dominant weight on $M$ by using $\rE_N$. 

Then we have $(N_{\psi})' \cap M = (N_\psi)' \cap M_\psi$. If the inclusion $N \subset M$ has finite index, then we moreover have $(N_\psi)' \cap M=\C 1$.
\end{thm}

\begin{proof}
By using \cite[Theorem ${\rm XII}$.1.1]{Ta03}, we can identify the inclusions
$$\left( N \subset M \right) = \left ( N_{\psi} \rtimes_{\theta} \R^*_+ \subset M_{\psi} \rtimes_{\theta} \R^*_+ \right)$$ 
where $\theta : \R^*_+ \curvearrowright M_{\psi}$ is a trace-scaling action that leaves $N_{\psi} \subset M_{\psi}$ globally invariant. We denote by $(v_{\lambda})_{\lambda > 0}$ the canonical unitaries in $N$ that implement the trace-scaling action $\theta : \R^*_+ \curvearrowright M_{\psi}$.

Let us first prove that $(N_\psi)' \cap M=\C 1$ when $N \subset M$ has finite index. Observe that $\rE_N$ sends $(N_\psi)' \cap M$ onto $(N_\psi)' \cap N$ and $(N_\psi)' \cap N=\C 1$ by Connes--Takesaki relative commutant theorem \cite[Chapter II, Theorem 5.1]{CT76}. Then the restriction of $\rE_N$ to $(N_\psi)' \cap M$, which is still of finite index, is in fact a faithful normal state. This means that $(N_\psi)' \cap M$ is finite dimensional. In particular, this implies that the relative flow of weights is inner on $(N_\psi)' \cap M$. Since the relative flow of weights is also ergodic on $(N_\psi)' \cap M$, this forces $(N_\psi)' \cap M=\C 1$ (see \cite[Theorem XI.3.11]{Ta03}).

Now, we deal with the case when $N \subset M$ is an arbitrary discrete irreducible inclusion. We freely use the terminology and the notation of \cite[Section 3]{ILP96}. We let $\Xi \subset \mathrm{Sect}(N)$ to be the set of all irreducible sectors of $N$ which appear in the $N$-$N$-bimodule decomposition of $\rL^2(M)$. Each $\xi \in \Xi$ can be represented by a unital normal endomorphism $\rho_{\xi} : N \to N$ such that the inclusion $\rho_{\xi}(N) \subset N$ is irreducible, with expectation and has finite index. We denote by $\rE_{\xi} : N \to \rho_{\xi}(N)$ the unique faithful normal conditional expectation. By \cite[Lemma 2.12 (\rm ii)]{ILP96} and up to replacing each $\rho_{\xi}$ by $\rho_{\xi} \circ \Ad(u)$ where $u \in \mathcal U(N)$, we may assume that $\psi|_{N} \circ \rE_{\xi} = \psi|_{N}$ and $\psi|_{N} \circ \rho_{\xi} = \psi|_{N}$ for every $\xi \in \Xi$. For every $\xi \in \Xi$, denote by 
$$\mathcal H_{\xi} = \{T \in M \mid \rho_{\xi}(x)T = Tx, \forall x \in N\}$$
the set of intertwiners in $M$ between $\id_{N}$ and $\rho_{\xi}$. Since $N' \cap M = \C 1$, $\mathcal H_{\xi}$ is a finite dimensional Hilbert space with inner product given by $\langle V, W\rangle 1 = V^{*}W$ (see \cite[Theorem 3.3]{ILP96}). We let $n_{\xi} = \dim \mathcal H_{\xi}$ and we fix an orthogonal basis $(V(\xi)_{i})_{1 \leq i \leq n_{\xi}}$ of $\mathcal H_{\xi}$. 

Denote by $\Lambda \subset \Xi$ the subset of sectors of the form $\xi = [\sigma_t^\psi]$ for some $t \in \R$. For every sector $\xi = [\sigma_{t}^{\psi}] \in \Lambda$, we may and will take $\rho_\xi=\sigma_t^\psi$ as a representing element. For every $\xi \in \Lambda$, $\mathcal{H}_\xi$ is one-dimensional and $V(\xi)_1$ is a unitary in $M_{\psi}$. Indeed, let $\xi = [\sigma_{t}^{\psi}] \in \Lambda$ and simply write $u  = V(\xi)_1 \in \mathcal U(M)$ such that $\sigma_t^\psi(x) = uxu^*$ for every $x \in N$. Since $N ' \cap M = \C 1$, \cite[Th\'eor\`eme 1.5.5]{Co72} implies that $\rE_N : M \to N$ is the unique faithful normal conditional expectation from $M$ onto $N$. Then we have $\Ad(u) \circ \rE_N = \rE_N \circ \Ad(u)$ and so $\psi = u \psi u^*$. This shows that $u \in M_\psi$.

Let $x \in (N_{\psi})' \cap M$ be any element. For every $\xi \in \Xi$ and every $i \in \{1, \dots, n_{\xi}\}$, denote by $x(\xi)_{i} = \rE_{N}(V(\xi)_{i} x) \in N$ the corresponding `Fourier coefficient' of $x \in M$. Fix $\xi \in \Xi$ and $i \in \{1, \dots, n_{\xi}\}$ such that $x(\xi)_{i} \neq 0$. For every $a \in N_{\psi}$, since $a x = x a$, we have $\rho_{\xi}(a) x(\xi)_{i} = x(\xi)_{i} a$ and so $x(\xi)_{i}^{*}x(\xi)_{i} \in (N_{\psi})' \cap N$. Since $(N_{\psi})' \cap N = \C 1$, we have $x(\xi)_{i}^{*}x(\xi)_{i} \in \C 1$. We also have $x(\xi)_{i}x(\xi)_{i}^{*} \in \rho_\xi(N_{\psi})' \cap N$. Since $\rho_\xi(N) \subset N$ is an irreducible inclusion of type $\III_1$ factors with finite index, it follows from the first part, that $\rho_\xi(N_{\psi})' \cap N=(\rho_\xi(N)_{\psi})' \cap N=\C 1$. This means that we also have $x(\xi)_{i}x(\xi)_{i}^{*} \in \C 1$. This shows that $x(\xi)_{i}=\mu  u$ for some $\mu > 0$ and some unitary $u \in N$. Then we obtain $\rho_\xi(a)u=ua$ for every $a \in N_\psi$. Put $\pi=\Ad(u^*) \circ \rho_\xi$ so that $\pi(a)=a$ for every $a \in N_\psi$. For every $a \in N_\psi$ and every $\lambda > 0$, since $\theta_\lambda(a)v_\lambda= v_\lambda a$, we also have $\theta_\lambda(a)\pi(v_\lambda)=\pi(v_\lambda)a$ and so $u_\lambda^*\pi(u_\lambda) \in (N_\psi)' \cap N=\C 1$. The map $\chi : \lambda \mapsto v_\lambda^*\pi(v_\lambda) \in \T$ is a character, that is, there exists $t \in \R$ such that $\chi(\lambda)= \lambda^{{\rm i}t}$ for every $\lambda > 0$. Since $\pi(a)=a$ for every $a \in N_\psi$ and since $\pi(v_\lambda)=\lambda^{{\rm i}t} \, v_\lambda$ for every $\lambda > 0$, we conclude that $\pi = \sigma_t^\psi$ and so $[\rho_\xi]=[\sigma_t^\psi]$. By assumption on the choice of $\rho_\xi$, this forces $\rho_\xi=\sigma_t^\psi$ and $u \in \C 1$. We have shown that if $x(\xi)_i \neq 0$, then we must have $\xi \in \Lambda$ and $x(\xi)_1 \in \C$. Observe that for every $s \in \R$, we have $\sigma_{s}^{\psi}(x) \in (N_{\psi})' \cap M$. Then for every $s \in \R$, every $\xi \in \Xi$ and every $i \in \{1, \dots, n_{\xi}\}$, we have
$$(\sigma_{s}^{\psi}(x))(\xi)_i = 0 = x(\xi)_i \quad \text{if} \quad \xi \notin \Lambda$$
and 
$$(\sigma_{s}^{\psi}(x))(\xi)_1 = \rE_{N}(V(\xi)_{1} \sigma_{s}^{\psi}(x)) = \rE_{N}(\sigma_{s}^{\psi}(V(\xi)_{1} x)) =\sigma_{s}^{\psi}( \rE_{N}(V(\xi)_{1} x)) = x(\xi)_1 \quad \text{if} \quad \xi \in \Lambda$$
since $V(\xi)_{1} \in M_{\psi}$ and $x(\xi)_{1} \in \C1$. Since the Fourier coefficients uniquely determine $x \in M$ (see \cite[p.\ 45]{ILP96}), we have $\sigma_{s}^{\psi}(x) = x$ for every $s \in \R$. This implies that $x \in M_{\psi}$ and so $x \in (N_{\psi})' \cap M_{\psi}$.
\end{proof}

We give below a precise description of the relative commutant $(N_{\psi})' \cap M_{\psi}$. We need to introduce some more terminology. Let $G$ be any discrete group, $N$ any von Neumann algebra, $\alpha : G \curvearrowright N$ any action and $c \in \rZ^2(G,\T)$ any scalar $2$-cocycle. The \emph{twisted crossed product} $M = N \rtimes_{\alpha, c} G$ is the von Neumann algebra generated by $N$ and a family of unitaries $(u_g)_{g \in G}$ that is characterized by the following properties:
\begin{itemize}
\item [$(\rm i)$] $u_g x u_g^*=\alpha_g(x)$ for all $x \in N$ and all $g \in G$.
\item [$(\rm ii)$] $u_{gh}=c(g,h) \, u_gu_h$ for all $g,h \in G$.
\item [$(\rm iii)$] There exists a faithful normal conditional expectation $\rE_N : M \rightarrow N$ that satisfies $\rE_N(u_g)=0$ for all $g \in G \setminus \{e\}$.
\end{itemize}
When $N=\C 1$, we obtain the \emph{twisted group von Neumann algebra} that we denote by $\rL_c(G)$ (see e.g.\ \cite{Su79}). Observe that $\rL_c(G)$ is a tracial von Neumann algebra with canonical faithful normal tracial state $\tau$ that satisfies $\tau(u_{g}) = 0$ for all $g \in G \setminus \{e\}$.

We also need the following technical result that is probably known to specialists. We nevertheless include a proof for the reader's convenience.

\begin{lem}\label{lem:location}
Let $M$ be any $\sigma$-finite von Neumann algebra and denote by $(M, \rL^{2}(M), J, \rL^{2}(M)^{+})$ its standard form. Let $A \subset M$ by any von Neumann subalgebra and $\xi \in \rL^2(M)^{+}$ any unit cyclic separating vector such that the faithful state $\varphi = \langle \, \cdot \, \xi, \xi\rangle \in M_{\ast}$ is tracial on $A$. Denote by $\overline{A \xi}$ the closure of $A \xi$ in $\rL^2(M)$.

Then for every element $x \in M$ such that  $x\xi \in \overline{A \xi}$, we have $x \in A$. 
\end{lem}

\begin{proof}
Let $x \in M$ be any element such that $x \xi \in \overline{A \xi}$. We show that $x \in A$. Take a sequence $(x_n)_n$ of elements in $A$ such that $\lim_n \|(x - x_n)\xi\| = 0$. Note that $(x_{n})_{n}$ need not be uniformly bounded. Since $\varphi |_{A}$ is tracial and since $x_n \in A$ for every $n \in \N$, we have
\begin{align*}
\|x_n^*\xi - x_m^*\xi\|^2 &= \|x_n^* - x_m^*\|_\varphi^2 \\
&= \varphi((x_n - x_m)(x_n - x_m)^*) \\
&= \varphi((x_n - x_m)^*(x_n - x_m)) \\
&=  \|x_n - x_m\|_\varphi^2 \\
&= \|x_n\xi - x_m\xi\|^2
\end{align*}
for all $m, n \in \N$. In particular, the sequence $(x_n^*\xi)_n$ is Cauchy in $\rL^2(M)$ and so is convergent in $\rL^2(M)$. Let $\eta \in \rL^2(M)$ be such that $\lim_n \|\eta - x_n^*\xi\| = 0$. Recall that the (possibly unbounded) operator $S_0 : M\xi \to \rL^2(M) : x\xi \to x^* \xi$ is closable and denote by $S$ its closure. Since $(x_n \xi, S(x_n \xi)) = (x_n \xi, x_n^*\xi) \to (x\xi, \eta)$ as $n \to \infty$ and since the graph of $S$ is closed, it follows that $\eta = S(x\xi) = x^*\xi$. This shows that $\lim_n \|x^* - x_n^*\|_\varphi = \lim_n \|x^*\xi - x_n^*\xi\| = 0$ and so 
$$\lim_n \|x - x_n\|_\varphi^\sharp = 0.$$ 
Then \cite[Lemma 2]{EW76} shows that $x \in A$.
\end{proof}

Under the same assumption as in Theorem \ref{thm-relative-connes-takesaki}, we obtain the following precise description of the relative commutant $(N_{\psi})' \cap M_{\psi}$.

\begin{thm}\label{thm-description}
Let $N \subset M$ be any discrete irreducible inclusion of factors with separable predual and with expectation $\rE_N : M \to N$. Assume that $N$ is a type $\III_1$ factor. Let $\psi$ be any dominant weight on $N$ and extend it to a dominant weight on $M$ by using $\rE_N$. 

Denote by $G \subset \R$ the subgroup of all the elements $t \in \R$ for which there exists $u_{t} \in \mathcal U(M)$ such that $\sigma_{t}^{\psi}(x) = u_{t}xu_{t}^{*}$ for every $x \in N$. Then there exists a scalar $2$-cocycle $c \in \rZ^{2}(G, \mathbf T)$ such that  
$$N \vee ((N_{\psi})' \cap M_{\psi}) = N \rtimes_{\sigma^{\psi}, c} G \quad  \text{and} \quad (N_{\psi})' \cap M_{\psi} = \rL_{c}(G).$$
Moreover, $\rE_{N}|_{(N_{\psi})' \cap M_{\psi}} = \tau 1$ where $\tau$ is the canonical trace on $\rL_{c}(G)$ and the relative flow of weights on $(N_{\psi})' \cap M_{\psi}$ satisfies $\theta^\psi_\lambda(u_t)=\lambda^{{\rm i}t} \, u_t$ for all $\lambda > 0$ and all $t \in G$.
\end{thm}

\begin{proof}
Write $W = \bigvee \{u_{t} \in \mathcal U(M) \mid  t \in G\} \subset M$. Write $Q = N \vee W \subset M$.  Since $\sigma^\psi : G \to \Aut(M)$ is a group homomorphism and since $N' \cap M = \C 1$, the map $c : G \times G \to \mathbf T : (s, t) \mapsto c(s, t)$ defined by $u_{s + t} = c(s, t) \, u_s u_t$ is a scalar $2$-cocycle, that is $c \in \rZ^{2}(G, \mathbf T)$, such that  
$$N \vee W = N \rtimes_{\sigma^{\psi}, c} G \quad  \text{and} \quad W = \rL_{c}(G).$$
(Observe that we have $\rE_{N}(u_{t}) = 0$ for every $t \in G \setminus \{0\}$ by \cite[Lemme 1.5.6]{Co72}). Since $Q \subset M$ is globally invariant under $\sigma^{\psi}$ and since $\psi|_{Q}$ is semifinite, \cite[Theorem]{Ta71} implies that there exists a $\psi$-preserving conditional expectation $\rE_{Q} : M \to Q$. Since $N' \cap M = \C1$, \cite[Th\'eor\`eme 1.5.5]{Co72} implies that $\rE_{N} \circ \rE_{Q} = \rE_{N}$.

We show that $W = (N_{\psi})' \cap M_{\psi}$. Following the notation of the proof of Theorem \ref{thm-relative-connes-takesaki}, we have $\Lambda = \{[\sigma_{t}^{\psi}] \mid t \in G\} \subset \Xi$. For every $\xi = [\sigma_{t}^{\psi}] \in \Lambda$, we may assume that $V(\xi)_{1} = u_{t}$. We already observed that for every $t \in G$, $u_{t} \in (N_{\psi})' \cap M_{\psi}$. Thus, we have $W \subset (N_{\psi})' \cap M_{\psi}$. Let now $x \in (N_{\psi})' \cap M_{\psi}$ be any element. We proved in Theorem \ref{thm-relative-connes-takesaki} that for every $\xi \notin \Lambda$ and every $i \in \{1, \dots, n_{\xi}\}$, we have $x(\xi)_{i} = 0$ and for every $\xi \in \Lambda$, we have $x(\xi)_{1} \in \C$. Observe that $\rE_{Q}(x) \in (N_{\psi})' \cap M_{\psi}$. Then for every $\xi \in \Xi$ and every $i \in \{1, \dots, n_{\xi}\}$, we have
$$(\rE_{Q}(x))(\xi)_i = 0 = x(\xi)_i \quad \text{if} \quad \xi \notin \Lambda$$
and 
$$(\rE_{Q}(x))(\xi)_1 = \rE_{N}(V(\xi)_{1} \rE_{Q}(x)) = \rE_{N}(\rE_{Q}(V(\xi)_{1} x)) = \rE_{N}(V(\xi)_{1} x) = x(\xi)_1 \quad \text{if} \quad \xi \in \Lambda$$
since $V(\xi)_{1} \in W \subset Q$ and $\rE_{N} \circ \rE_{Q} = \rE_{N}$. Since the Fourier coefficients uniquely determine $x \in M$ (see \cite[p.\ 45]{ILP96}), we have $\rE_{Q}(x) = x$ and so $x \in Q$. Choose any faithful state $\varphi \in M_{\ast}$ such that $\varphi \circ \rE_{N} = \varphi$. Observe that $\xi_{\varphi} \in \rL^{2}(N)^{+} \subset \rL^{2}(Q)^{+}$. The result obtained in Theorem \ref{thm-relative-connes-takesaki} and \cite[p.\ 45]{ILP96} show that in the standard form $\rL^{2}(Q)$, we have the following convergence
$$x \xi_{\varphi} = \sum_{\xi \in \Lambda} x(\xi)_{1} V(\xi)_{1}\xi_{\varphi} = \sum_{t \in G} x_{t} u_{t}\xi_{\varphi} \in \overline{\rL_{c}(G)\xi_{\varphi}}$$
where $x_{t} 1 =  x(\xi)_{1} \in \C 1$ for $\xi = [\sigma_{t}^{\psi}]$. Since $\varphi |_{\rL_{c}(G)}$ is the canonical trace $\tau$, Lemma \ref{lem:location} shows that $x \in \rL_{c}(G) = W$. 

It remains to compute the relative flow of weights on $(N_{\psi})' \cap M_{\psi}$. It is plain to see that $\rE_{N}|_{(N_{\psi})' \cap M_{\psi}} = \tau 1$. By using \cite[Theorem ${\rm XII}$.1.1]{Ta03}, we can identify the inclusions
$$\left( N \subset M \right) = \left ( N_{\psi} \rtimes_{\theta} \R^*_+ \subset M_{\psi} \rtimes_{\theta} \R^*_+ \right)$$ 
where $\theta : \R^*_+ \curvearrowright M_{\psi}$ is a trace-scaling action that leaves $N_{\psi} \subset M_{\psi}$ globally invariant. We denote by $(v_{\lambda})_{\lambda > 0}$ the canonical unitaries in $N$ that implement the trace-scaling action $\theta : \R^*_+ \curvearrowright M_{\psi}$. For every $\lambda > 0$ and every $t \in \R$, we have 
$$u_{t} v_{\lambda} u_{t}^{*} = \sigma_{t}^{\psi}(v_{\lambda}) = \lambda^{-{\rm i}t} \, v_{\lambda}$$
and so $\theta^\psi_{\lambda}(u_{t}) = v_{\lambda} u_{t} v_{\lambda}^{*} = \lambda^{{\rm i}t} \, u_{t}$.
\end{proof}

We deduce the following interesting consequence from Theorem \ref{thm-description}.

\begin{cor}
Let $N \subset M$ be any discrete irreducible inclusion of factors with separable predual and with expectation $\rE_N : M \to N$. Assume that $N$ is a type $\III_1$ factor. Let $\psi$ be any dominant weight on $N$ and extend it to a dominant weight on $M$ by using $\rE_N$. 

The following assertions are equivalent:
\begin{itemize}
\item [$(\rm i)$] $(N_\psi)' \cap M=\C 1$.

\item [$(\rm ii)$] Every intermediate subfactor $N \subset P \subset M$ is of type $\III_1$.
\end{itemize}
\end{cor}

\begin{proof}
$(\rm i) \Rightarrow (\rm ii)$ Let $N \subset P \subset M$ be any intermediate subfactor. Observe that $\psi |_{P}$ is a dominant weight on $P$ and $P_\psi$ is a factor  since $\mathcal{Z}(P_\psi) \subset (N_\psi)' \cap M = \C 1$. Thus, $P$ is a type ${\rm III_{1}}$ factor.

$(\rm ii) \Rightarrow (\rm i)$ By contraposition, if $(N_\psi)' \cap M \neq \C 1$, then by Theorem \ref{thm-description}, there exists $T > 0$ and a unitary $u \in M$ such that
$$ \forall x \in N, \quad \sigma_T^\psi(x)=uxu^*.$$
Then $P = \langle N, u\rangle$ is an intermediate subfactor such that $P \cong N \rtimes_{\sigma^\psi_T} \Z$ (observe that we have $\rE_{N}(u^{n}) = 0$ for every $n \in \Z \setminus \{0\}$ by \cite[Lemme 1.5.6]{Co72}). Then \cite[Lemma 1]{Co85} implies that $P$ is a type $\III_\lambda$ factor where $\lambda = \exp(-\frac{2\pi}{T})$.
\end{proof}

The next proposition provides examples for which the open question from the introduction has a positive solution.

\begin{prop}\label{prop:open-question}
Let $N$ be any type ${\rm III_{1}}$ factor with separable predual and with trivial bicentralizer. Let $\iota : G \hookrightarrow \R$ be any injective homomorphism, where $G$ is an abelian countable discrete group and $c \in \rZ^{2}(G, \mathbf T)$ any scalar $2$-cocycle. Let $\psi$ be any dominant weight on $N$ and define $\alpha = \sigma^\psi \circ \iota : G \curvearrowright N$ and put $M = N \rtimes_{\alpha, c} G$. Extend $\psi$ to $M$ by using the canonical conditional expectation $\rE_{N} : M \to N$. Let $\varphi \in N_{\ast}$ be any faithful state.

Then $N \subset M$ is a discrete irreducible inclusion and there exists an isomorphism $$\pi : \rB(N \subset M, \varphi) \rightarrow (N_\psi)' \cap M$$ 
such that $\theta^\psi =\pi \circ \beta^\varphi\circ \pi^{-1}$ and $\rE_N(\pi(x))=\varphi(\rE_N(x))1$ for all $x \in \rB(N \subset M, \varphi)$.
\end{prop}

\begin{proof}
We know that $N \subset M$ is a discrete irreducible inclusion. Write $P = N \rtimes_{\sigma^\varphi \circ \iota, c} G$. By Connes' Radon--Nikodym cocycle theorem \cite[Theorem 1.2.1]{Co72}, the actions $\sigma^\varphi \circ \iota : G \curvearrowright N$ and $\sigma^\psi \circ \iota : G \curvearrowright N$ are cocycle conjugate and so the inclusions $N \subset P$ and $N \subset M$ are isomorphic. We also denote by $\rE_{N} : P \to N$ the unique faithful normal conditional expectation.

Since $\rB(N, \varphi) = \C1$, a straightforward argument using the Fourier expansion shows that $\rB(N \subset P, \varphi) = \rL_{c}(G) \subset P$ and $\rE_{N}|_{\rB(N \subset P, \varphi)} = \tau 1$ where $\tau$ is the canonical trace on $\rL_{c}(G)$. Moreover, for every $x \in \rB(N \subset P, \varphi)$, every $\lambda > 0$ and every $(a_{n})_{n} \in \ell^{\infty}(\N, N)$ such that $\lim_{n} \|a_{n}\varphi - \lambda \varphi a_{n}\| = 0$, we have $a_{n}x - \beta^{\varphi}_{\lambda}(x)a_{n} \to 0$ $\ast$-strongly. Since $\sigma_{t}^{\varphi}(a_{n}) - \lambda^{-{\rm i}t} a_{n} \to 0$ $\ast$-strongly for every $t \in \R$, it follows that $u_{g} a_{n} u_{g}^{*} - \lambda^{-{\rm i}\iota(g)} a_{n} \to 0$ $\ast$-strongly for every $g \in G$. This implies that $\beta^{\varphi}_{\lambda}(u_{g}) = \lambda^{{\rm i}\iota(g)} u_{g}$ for every $g \in G$. Combining this with the conclusion of Theorem \ref{thm-description} and since $\rB(N \subset P, \varphi) \cong \rB(N \subset M, \varphi)$, we obtain the desired isomorphism $\pi$.
\end{proof}

\begin{rem}
Let $N$ be any type ${\rm III_{1}}$ factor with separable predual and with trivial bicentralizer. Put $G = \Z^{2}$ and let $\iota : G \hookrightarrow \R$ be any injective homomorphism. Choose a scalar $2$-cocycle $c \in \rZ^{2}(G, \mathbf T)$ such that $\rL_{c}(G) \cong R$ is the hyperfinite type ${\rm II_{1}}$ factor (realize $R = \rL^{\infty}(\mathbf T) \rtimes \Z$ where the action $\Z \curvearrowright \mathbf T$ comes from an irrational rotation). Let $\psi$ be any dominant weight on $N$ and define $\alpha = \sigma^\psi \circ \iota : G \curvearrowright N$ and put $M = N \rtimes_{\alpha, c} G$.

By Theorem \ref{thm-description}, since $(N_{\psi})' \cap M_{\psi} = \rL_{c}(G) \cong R$ is a factor, it follows that $M_{\psi}$ is a factor and so $M$ is a type ${\rm III_{1}}$ factor. Then $N \subset M$ is a discrete irreducible inclusion of type ${\rm III_{1}}$ factors. Proposition \ref{prop:open-question} implies that the relative bicentralizer flow (which coincides with the relative flow of weights) is ergodic. We also point out that $M$ has trivial bicentralizer. Indeed, let $\varphi \in M_{\ast}$ be any faithful state such that $\varphi \circ \rE_{N} = \varphi$.  By Proposition \ref{prop:open-question}, we may identify $M = N \rtimes_{\sigma^{\varphi} \circ \iota, c} G$ and we have $\rB(N \subset M, \varphi) = \rL_{c}(G)$. Since $\rL_{c}(G)$ is a  factor and $\rL_{c}(G) \subset M_{\varphi}$ and since $\rB(M, \varphi) \subset \rB(N \subset M, \varphi) = \rL_{c}(G)$ and $\rB(M, \varphi) \subset (M_{\varphi})' \cap M$, it follows that $\rB(M, \varphi) = \C 1$.
\end{rem}

\subsection{Proof of Corollary \ref{cor-characterization}}

\begin{proof}[Proof of Corollary \ref{cor-characterization}]
$(\rm ii) \Rightarrow (i)$ This follows from Proposition \ref{prop-masa}.

$(\rm i) \Rightarrow (i\rm i)$ Fix a dominant weight $\psi$ on $N$. Since $\core(N)' \cap \core(M) = \C 1$, we have $(N_{\psi})' \cap M = \C 1$ by Theorem \ref{thm-relative-connes-takesaki}. Since $N$ is amenable, so is $N_{\psi}$ and thus the inclusion $N_{\psi} \subset M$ satisfies the weak relative Dixmier property. By Theorem \ref{thm-relative-bicentralizer}, there exists a faithful state $\varphi \in M_{\ast}$ such that $\varphi \circ \rE_{N} = \varphi$ and $(N_{\varphi})' \cap M = \C 1$. Finally, by \cite[Theorem 3.2]{Po81}, there exists an abelian von Neumann subalgebra $A \subset N_{\varphi}$ (with expectation) that is maximal abelian in $M$.
\end{proof}

\begin{proof}[Proof of Application \ref{app-crossed-product}]
We canonically have $\core(N \rtimes \Gamma) = \core(N) \rtimes \Gamma$. Application  \ref{app-crossed-product} is now a consequence of Corollary \ref{cor-characterization} and \cite[Proposition 5.4]{HS88}.
\end{proof}

\begin{proof}[Proof of Application \ref{app-minimal-action}]
This is a consequence of Corollary \ref{cor-characterization} and \cite[Corollary 5.14]{Iz01} and its proof where it is shown that $\core(M^{\mathbf G})' \cap \core(M) = \C1$.
\end{proof}

\section{Bicentralizers of tensor product factors}

We first prove a relation between the bicentralizer of a tensor product and the product of the bicentralizers.

\begin{prop} \label{tensor-formula}
Let $M$ and $N$ be any two $\sigma$-finite factors with faithful states $\varphi \in M_{\ast}$ and $\psi \in N_{\ast}$. Then we have
$$ \rB(M \ovt N,\varphi \otimes \psi) \subset \rB(M,\varphi) \ovt \rB(N,\psi).$$
If $M$ is a type $\III_1$ factor, then for all $x \in \rB(M \ovt N,\varphi \otimes \psi)$ and all $\lambda > 0$, we have
$$ \beta^{\varphi \otimes \psi}_\lambda(x)=(\beta_\lambda^\varphi \otimes \id)(x).$$
\end{prop}

\begin{proof}
Let $a \in \rB(M \ovt N,\varphi \otimes \psi)$ and write $a\xi_{\varphi \otimes \psi}  = \sum_{i \in I} \xi_i \otimes \eta_i$ where $(\xi_i)_{i \in I}$ is an orthonormal basis of $\rL^2(M)$. Take $(x_n)_n \in \AC(N, \psi)$. Since $a \in \rB(M \ovt N,\varphi \otimes \psi)$, we have $(1 \otimes x_n) a\xi_{\varphi \otimes \psi} - a\xi_{\varphi \otimes \psi} (1 \otimes x_n) \rightarrow 0$. Since $(\xi_i)_{i \in I}$ is orthonormal, this easily implies that $x_n \eta_i - \eta_{i} x_{n} \rightarrow 0$ for all $i \in I$. This shows that $\eta_i \in \overline{\rB(N,\psi)\xi_{\psi}}$ and so $a \in M \ovt \rB(N,\psi)$. Similarly, by decomposing over an orthonormal basis of $\rL^2(N)$, we obtain $a \in \rB(M,\varphi) \ovt \rB(N,\psi)$.
\end{proof}

Now, we prove Theorem \ref{tensor_products_trivial}. For this, we will need the following criterion which is extracted from \cite[Theorem 2.3]{Ha85}.

\begin{thm} \label{local_haagerup}
Let $M$ be any $\sigma$-finite type $\III_1$ factor. Assume that there exists $\kappa > 0$ such that for every $\delta > 0$, every faithful state $\varphi \in M_{\ast}$ and every $x \in M$ such that $x \xi_\varphi=\xi_\varphi x^*$, $\varphi(x)=0$ and $\|x\|_\varphi=1$, we can find $a \in M$ such that
$$  \| a \|_\varphi + \| ax\|_\varphi < \kappa \| xa-a x \|_\varphi$$
and
$$ \|a \xi_\varphi - \xi_\varphi  a\| < \delta \| xa-a x \|_\varphi. $$
Then $M$ has trivial bicentralizer.
\end{thm}

\begin{proof}
Note that in the proof of \cite[Theorem 2.3 $(1) \Rightarrow (2)$]{Ha85}, Condition $(1)$ is only used to obtain the following claim \cite[Lemma 2.9]{Ha85}:
\begin{itemize}
\item There exists a constant $\kappa > 0$ such that for every $\delta > 0$, every cyclic separating vector $\xi \in \rL^2(M)^+$ and every unit vector $\eta=\eta^* \in \rL^2(M)$ that is orthogonal to $\xi$, there exists $a \in M$ such that
$$  \| a \xi \| + \| a\eta\| < \kappa \| a\eta-a \eta \|$$
and
$$ \|a \xi - \xi  a\| < \delta \| a\eta-\eta a \|. $$
\end{itemize}
Once this claim is obtained, the proof of \cite[Theorem 2.3]{Ha85} can be continued without any other assumption on $M$. Moreover, a closer look at the proof of \cite[Lemma 2.9]{Ha85} shows that one only needs to prove this claim when $\eta$ is of the form $\eta=x\xi=\xi x^*$ for some $x \in M$. Finally, by writing $\xi=\xi_\varphi$ for some faithful state $\varphi \in M_*$ and using the fact that 
$$ \| a \eta- \eta a \| = \| a \, x \xi_\varphi- x \xi_\varphi \, a \| \leq \| ax-xa \|_\varphi + \| x \|_\infty \cdot \| a \xi_\varphi -  \xi_\varphi a \|$$
we see that this claim can be reformulated as in our statement.
\end{proof}

\begin{proof}[Proof of Theorem \ref{tensor_products_trivial}]
Let $N$ be any factor such that $M \ovt N$ has trivial bicentralizer. Put $P=R_\infty$. If $N$ is of type $\III_\lambda$ $(0 < \lambda < 1)$, put instead $P=R_\lambda$. If $N$ is semifinite, then $M \ovt pNp$ has trivial bicentralizer for any nonzero finite projection $p \in N$ and so we may assume that $N$ is finite. In that case, put instead $P=\C1$. We have to show that $M \ovt P$ has trivial bicentralizer. In fact, since $(M \ovt P) \ovt N$ has trivial bicentralizer by Proposition \ref{tensor-formula}, we may assume that $M \cong M \ovt P$.

We use the criterion of Theorem \ref{local_haagerup}. Let $\delta > 0$ and $\varphi \in M_{\ast}$ be any faithful state. Let $x \in M$ be such that $x\xi_\varphi=\xi_\varphi x^*$ with $\| x \|_\varphi=1$ and $\varphi(x)=0$. Let $\psi \in N_{\ast}$ be any faithful state (if $N$ is of type $\III_\lambda$ with $0 < \lambda < 1$, we let $\psi$ to be a periodic state and if $N$ is finite, we let $\psi$ to be the trace). Since $M \ovt N$ has trivial bicentralizer, we can find an element $b \in \Ball(M \ovt N)$ such that $\| b\xi_{\varphi \otimes \psi}-\xi_{\varphi \otimes \psi} b\| \leq \delta$ and $\| b(x \otimes 1)-(x \otimes 1)b \|_{\varphi \otimes \psi} > \frac{1}{2}$ (see \cite[Lemma 3.2]{Ha85}).
Use the spectral theorem to identify $\rL^2(N)$ with $\rL^2(T,\mu)$ for some probability space $(T,\mu)$ in a such a way that $\Delta_\psi$ is identified with a multiplication operator by a borel function $f : (T, \mu) \rightarrow \R^*_+$ (if $N$ is of type $\III_\lambda$ with $0 < \lambda < 1$, then $f$ takes its values in $\lambda^\Z$ and if $N$ is finite then $f$ is constant equal to $1$). Identify $\zeta=b \xi_{\varphi \otimes \psi} $ with a function $t \mapsto \zeta(t)$ in $\rL^2(T,\mu, \rL^2(M))$. Note that
$$ \| b\xi_{\varphi \otimes \psi} -\xi_{\varphi \otimes \psi} b\|^2=\| (1-\Delta_{\varphi \otimes \psi}^{1/2}) \zeta \|^2=\int_T \| (1-f(t)^{1/2} \Delta_\varphi^{1/2})\zeta(t) \|^2 \, \rd \mu(t). $$
Therefore, $\zeta(t)$ is in the domain of $\Delta_\varphi^{1/2}$ for almost every $t \in T$ and we have
$$ \int_T \| (1-f(t)^{1/2} \Delta_\varphi^{1/2})\zeta(t) \|^2 \, \rd \mu(t)  \leq \delta^2.$$
We also have
$$ \| (x \otimes 1)b -b(x\otimes 1)\|_{\varphi \otimes \psi}^2=\| (x \otimes 1) \zeta - \zeta (x^* \otimes 1) \|^2=\int_T \| x \zeta(t)-\zeta(t) x^*\|^2 \, \rd \mu(t). $$ 
Then we have
$$ \int_T \| x \zeta(t)-\zeta(t) x^*\|^2 \, \rd \mu(t) > \frac{1}{4}.$$ 
Similarly, we also have
$$ \int_T (\| \zeta(t) \|^2 + \|\zeta(t)x^*\|^2) \, \rd \mu(t) = \|b\|_{\varphi \otimes \psi}^2+\|b(x \otimes 1)\|_{\varphi \otimes \psi}^2 \leq 2. $$
Gathering all these inequalities, we obtain
$$ \int_T \left( \| \zeta(t) \|^2 + \|\zeta(t)x^*\|^2+\frac{1}{\delta^{2}}\| (1-f(t)^{1/2} \Delta_\varphi^{1/2})\zeta(t) \|^2  \right) \rd \mu(t) < 12 \int_T \| x \zeta(t)-\zeta(t) x^*\|^2 \, \rd \mu(t). $$
Therefore, there exists $t \in T$ such that $\zeta(t)$ is in the domain of $\Delta_\varphi^{1/2}$ and
$$ \| \zeta(t) \|^2 + \|\zeta(t)x^*\|^2+\frac{1}{\delta^{2}}\| (1-f(t)^{1/2} \Delta_\varphi^{1/2})\zeta(t) \|^2 < 12 \| x \zeta(t)-\zeta(t) x^*\|^2. $$
Recall that the graph of $\Delta_\varphi^{1/2}$ is the closure of $\{ (c\xi, \xi c) \mid c \in M \}$. Then we can find $c \in M$ such that
$$ \| c\xi \|^2 + \|c \xi x^*\|^2+\frac{1}{\delta^{2}}\| c\xi-f(t)^{1/2} \xi c \|^2 < 12 \| x c \xi - c\xi x^*\|^2. $$
Now, since $M \cong M \ovt P$, we can find a nonzero partial isometry $v \in M' \cap M^\omega$ such that $v \xi^\omega=f(t)^{-1/2} \xi^\omega v$ (note that $f(t) \in \lambda^{\Z}$ if $N$ is of type $\III_\lambda$ with $0 < \lambda < 1$ and $f(t)=1$ if $N$ is finite). Since $\rE_M(v^*v) > 0$ is a nonzero scalar, then in $\rL^2(M^\omega)$, we obtain
$$ \| vc\xi^\omega \|^2 + \|vc \xi^\omega x^*\|^2+\frac{1}{\delta^{2}}\| vc\xi^\omega-\xi^\omega vc \|^2 < 12 \| x vc \xi^\omega - vc\xi^\omega x^*\|^2. $$
By letting $a=v_nc$ for $n$ large enough where $v=(v_n)^\omega$, we obtain
$$ \| a\xi_\varphi \|^2 + \|a \xi_\varphi x^*\|^2+\frac{1}{\delta^{2}}\| a\xi_\varphi-\xi_\varphi a \|^2 < 12 \| x a\xi_\varphi - a\xi_\varphi x^*\|^2. $$
Since $x\xi_\varphi=\xi_\varphi x^*$, this implies that
$$ \| a\|_\varphi^2 + \| ax\|_\varphi^2 < 12 \| xa - ax\|_\varphi^2 $$
and 
$$\| a\xi_\varphi-\xi_\varphi a \|^2 < 12\delta^2 \| xa - ax\|_\varphi^2 $$
Since $\delta > 0$ is arbitrary, this finishes the proof.
\end{proof}

\begin{proof}[Proof of Application \ref{UPF}]
It is enough to prove the theorem for $n=2$. Also, it is enough to show that we can find some $i \in \{1, \dots, m \}$ such that $M_i \preceq_M N_1$. Indeed, we can then use \cite[Proposition 6.3 $(\rm v)$]{HMV16} and conclude by induction over $m \geq 1$. 

We first note that $M_i$ has trivial bicentralizer for all $i \in \{1, \dots, m \}$ by \cite[Theorem C]{HI15}. Proposition \ref{tensor-formula} implies that $M$ has trivial bicentralizer. Therefore $N_1 \ovt R_\infty$ and $N_2 \ovt R_\infty$ also have trivial bicentralizer by Theorem \ref{tensor_products_trivial}. Therefore, we can apply \cite[Lemma 5.2]{HI15} to the decomposition
$$\mathcal{M}= (R_\infty  \ovt R_\infty) \ovt M_1 \ovt \cdots \ovt M_m=(N_1 \ovt R_\infty) \ovt (N_2 \ovt R_\infty).$$
We obtain some $i \in \{1, \dots, m \}$ such that $M_i \preceq_{\mathcal{M}} N_1 \ovt R_\infty$. But this implies that $M_i \preceq_M N_1$ by \cite[Lemma 6.1]{HMV16}.
\end{proof}

\bibliographystyle{plain}

\end{document}